\documentclass[10pt]{amsart}
\usepackage[margin=0.8in]{geometry}
\usepackage{amssymb}
\usepackage{amsmath}
\usepackage{eurosym}
\usepackage{paralist}
\usepackage{graphics}
\usepackage{epsfig}
\usepackage{tabu}
\usepackage{graphicx}
\usepackage{epstopdf}
\usepackage[colorlinks=true]{hyperref}
\usepackage{hyperref}
\usepackage{amsmath}
\usepackage{amsfonts}
\usepackage{fancyhdr}
\usepackage{color}

\newtheorem{theorem}{Theorem}[section]

\newtheorem{claim}[theorem]{Claim}

\newtheorem{corollary}[theorem]{Corollary}

\newtheorem{definition}[theorem]{Definition}

\newtheorem{lemma}[theorem]{Lemma}

\newtheorem{proposition}[theorem]{Proposition}
\newtheorem{remark}[theorem]{Remark}

\numberwithin{equation}{section}

\begin{document}
\title{Dynamics of threshold solutions for the energy-critical inhomogeneous NLS}

 \author[Liu]{Xuan Liu}
\address{ School of Mathematics, Hangzhou Normal   University, Hangzhou 311121, China}
\email{liuxuan95@hznu.edu.cn }

\author[Yang]{Kai Yang}
\address{School of Mathematics, Southeast University, Nanjing, 211189, China}
\email{kaiyang@seu.edu.cn,yangkai99sk@gmail.com }

\author[Zhang]{Ting Zhang}
\address{ School of Mathematical Sciences,  Zhejiang University, Hangzhou 310058, China}
\email{zhangting79@zju.edu.cn }

\date{}

\maketitle
 
\begin{abstract}
In this article, we study the long-time dynamics of threshold solutions for the focusing energy-critical inhomogeneous Schr\"odinger equation and classify the corresponding threshold solutions in dimensions  $d=3,4,5$. We first show the existence of special threshold solutions  $W^\pm$ by constructing a sequence of approximate solutions in suitable Lorentz space, which exponentially approach the ground state  $W$ in one of the time directions. We then prove that solutions with threshold energy either behave as in the subthreshold case or agree with $W, W^+$ or  $W^-$ up to the symmetries of the equation. The proof relies on detailed spectral analysis of the linearized Schr\"odinger operator, the relevant modulation analysis, the global Virial analysis,  and the concentration compactness argument in the Lorentz space. 
\end{abstract}
\section{Introduction}\label{s1}
We consider the following  focusing energy-critical inhomogeneous  nonlinear Schr\"odinger equation 
\begin{equation}\label{nls}
		\begin{cases}
		i\partial_{t}u+\Delta u+|x|^{-b}|u|^\alpha u=0,\qquad (t,x)\in \mathbb{R} \times \mathbb{R} ^d\\
		u|_{t=0}=u_0\in \dot H^{1}(\mathbb{R}^d),
	\end{cases}
\end{equation}
 where  $d\ge3,0<b<\min \left\{ 2,\frac{d}{2} \right\}$ and  $\alpha =\frac{4-2b}{d-2}$.   
  
 This model arises in the setting of nonlinear optics, where the factor  $|x|^{-b}$  represents some inhomogeneity in the medium (see, e.g., \cite{Gill2000,LT1994}).  As pointed out by Genoud and Stuart \cite{GS2008},  the factor  $|x|^{-b}$  appears naturally as a limiting case of potentials that decay polynomially at infinity.
 
 On the interval of existence, the solution preserves its energy 
 \begin{equation}
 	E(u):=\int _{\mathbb{R} ^d}\frac{1}{2}|\nabla u(t,x)|^2-\frac{1}{\alpha +2}|x|^{-b}|u(t,x)|^{\alpha +2}dx=E(u_0).\notag
 \end{equation}
 Equation (\ref{nls}) is referred to as focusing as the  potential energy is negative. Equation (\ref{nls}) is also referred to as energy-critical as the natural scaling of the equation  $u(t,x)\rightarrow \lambda^{\frac{d-2}{2}}u(\lambda^2t,\lambda x)$  keeps the energy invariant.    
 
 The  $\dot H^1$ local well-posedness of (\ref{nls}) was  noted as an open problem in \cite[Remark 1.7, p.252]{Guzman} and \cite[line 38, p. 171]{Dinh2018}. The main difficulty   comes from the singularity of  $|x|^{-b}$ at the origin. Using  Strichartz estimates in some weighted Lebesgue spaces, the authors in \cite{KimLee,LeeSeo} established the  $\dot H^1$ local existence of (\ref{nls}) under some restrictive hypotheses on  $d$ and  $b$.  Recently, by considering Strichartz estimates in Lorentz spaces, Aloui and Tayachi\cite{Aloui}  ultimately established the  $\dot H^1$ local well-posedness  of the Cauchy problem (\ref{nls}) (see Theorem \ref{T:CP} below).
 
The local theory in \cite{Aloui,KimLee,LeeSeo}  also proves scattering for sufficiently small initial data.
Here scattering refers to the fact that
\begin{equation} 
	\exists u_\pm\in \dot{H}^1(\mathbb{R}^d)\quad\text{ such that }\quad  \lim_{t\to\pm\infty}\|u(t)-e^{it\Delta}u_\pm\|_{\dot H^1} = 0,\notag
\end{equation}
where $e^{it\Delta}$ is the linear Schr\"odinger group. 
  The existence of the non-scattering solution (the ground state solution)
 \begin{equation}
 	W(x):= \left(1+\frac{|x|^{2-b}}{(d-2)(d-b)}\right) ^{-\frac{d-2}{2-b}}     \label{W}
 \end{equation}
shows that scattering does not hold for all initial data  $u_0\in \dot H^1(\mathbb{R} ^d)$. Instead, the threshold between blowup and scattering is proved to be determined by the ground state:  
\begin{theorem}[\cite{ChoHongLee,ChoLee,2021JDE, Liu-Zhang}]\label{T:0}
	Let  $d\ge3, 0<b<\min \left\{ 2,\frac{d}{2} \right\}$. Suppose  $u_0\in \dot H^1(\mathbb{R} ^d)$  satisfies  $E(u_0)<E(W)$. \\
	(a) If $ \|\nabla u_0\|_{L^2(\mathbb{R} ^d)}< \|\nabla W\|_{L^2(\mathbb{R} ^d)}  $, then   the corresponding solution  $u$  to (\ref{nls}) is global and scatters as  $t\rightarrow\pm \infty $. \\
	(b)  If  $ \|\nabla u_0\|_{L^2(\mathbb{R} ^d)}\ge\|\nabla W\|_{L^2(\mathbb{R} ^d)}  $ and  either   $xu_0\in L^2(\mathbb{R} ^d)$ or  $u_0\in  H^1(\mathbb{R} ^d)$   is radial, then the corresponding solution  $u$ to (\ref{nls}) blows up in finite time.
\end{theorem}

Theorem~\ref{T:0}  was first obtained by Cho-Hong-Lee \cite{ChoHongLee} for the radial inhomogeneous NLS when  $d=3,0<b<\frac{4}{3}$,   and then extended by    Cho-Lee \cite{ChoLee}   to the radial inhomogeneous NLS with  $d=3,\frac{4}{3}\le b<\frac{3}{2}$.
In  the paper \cite{ 2021JDE},   Guzm\'an-Murphy proved Theorem \ref{T:0}  for non-radial initial data in the case  $d=3,b=1$.   
The results of \cite{ 2021JDE} are shown by   a concrete concentration compactness argument based on  Hardy's inequality and the fact that, (\ref{nls}) is well-approximated by the linear equation in the regime  $|x|\rightarrow\infty $.  
 The restriction on indices  $b$ and  $d$ are due to the lack of local theory of (\ref{nls}).  
Recently, Aloui and Tayachi \cite{Aloui} established the  $\dot H^1$ local well-posedness  of (\ref{nls})  by considering the contraction argument in the Lorentz spaces, which are properly suited for handing the inhomogeneity  $|x|^{-b}$.  Based on the work of \cite{Aloui},
 Liu-Zhang \cite{Liu-Zhang}   developed the  concentration compactness argument in the Lorentz space and  proved Theorem \ref{T:0} for all  $d\ge3,0<b<\min \left\{ 2,\frac{d}{2} \right\}$.

Observe that the above characterization is obtained only in the subthreshold case, i.e.  $E(u_0)<E(W)$.  
Our purpose of this paper is to continue the study in \cite{ChoHongLee,ChoLee,2021JDE, Liu-Zhang} on what will happen if the solution has the threshold energy, i.e.     $E(u_0)=E(W)$.    We call these solutions "threshold solutions". 
The classification of threshold solutions  was initiated by Duyckaerts–Merle for the focusing energy critical nonlinear Schr\"odinger and wave equation in their seminal works \cite{DM2008,2008GFA} in dimensions  $d=3,4,5$.   
These  results were later extended to higher dimensions in \cite{CFRoud:threshold,2009JFA,LZ2011}. 
See also    \cite{Su}  on the same topic in  the nonradial subcritical case.   
The study on the threshold scattering  has been a topic of recent mathematical interest.  
See e.g. \cite{AHI,AI,AM,AMZ,Campos-Murphy,2022DLR,Hamano,2021KMV,Lixuemei,2023MMZ,MWX,YZZ,YZ} and references therein for more related works on this topic.

To classify the threshold solutions to the focusing energy critical inhomogeneous NLS (\ref{nls}), we  first establish the following theorem, which shows the existence of special threshold solutions converging exponentially to  $W$.   
\begin{theorem}
	\label{T1}
Let  $3\le d\le5$.  There exist   radial global solutions  $W^\pm$  to  (\ref{nls}) with  $E(W^\pm)=E(W)$, defined on  $I^\pm \supset [0,\infty )$, which satisfy 
\begin{equation}
	 \|W^\pm (t)-W\|_{\dot H^1}\lesssim e^{-ct},\notag 
\end{equation}
for some  $c>0$ and all  $t>0$. 

The solution  $W^-$ is global ($I^-=\mathbb{R} $), satisfies    
	\begin{equation}
		\|\nabla W^-(t)\|_{L^2}<\|\nabla W\|_{L^2},\qquad \forall t\in \mathbb{R},\notag
	\end{equation}
	and  scatters in   $\dot H^1(\mathbb{R} ^d)$ as  $t\rightarrow-\infty $.   
	
	The solution  $W^+$ satisfies  
\begin{equation}
	\|\nabla W^+(t)\|_{L^2}>\|\nabla W\|_{L^2},\qquad \forall t\in \mathbb{R}.\notag
\end{equation}
Moreover, if  $d=5$,   $W^+$  blows up in finite negative time ($I^+=(T_-,\infty )$ for some $T_-<\infty $). 
\end{theorem}
\begin{remark}
	The proof of  $T_-<\infty $ relies heavily on the  $L^2$ regularity of  $W^+$.    
As   $W$   belongs to  $L^2(\mathbb{R} ^d)$  if and only if  $d\ge5$,  there is no  $L^2$ regularity for  $W^+$ in dimensions  $d=3,4$ from the construction in  Proposition \ref{P:approximate}.  
We still expect  $T_-<\infty $ for   $d=3,4$.  
\end{remark}
\begin{remark}
	The restriction  $d\le5$ together with the restriction for  $b$ in Theorem \ref{T2}  ensures that   $\alpha=\frac{4-2b}{d-2} \ge1$, 
	which provides sufficient regularity to handle the non-smoothness issue of the nonliearity.   Using the method of \cite{CFRoud:threshold}, our results can be extended to  $d\ge6$ with   certain restrictions on  $b$.   We will address the high dimension problem elsewhere.  
\end{remark}

Using the special threshold solutions  $W^\pm$  constructed in Theorem \ref{T1} and the ground state  $W$, we can classify all threshold solutions to (\ref{nls}). 
\begin{theorem}
	\label{T2}
	Let  $3\le d\le 5,0<b<-\frac{(d-4)^2}{2}+1$,  $u_0\in \dot H^{1}(\mathbb{R}^d)$ be such that  $E(u_0)=E(W)$. Let  $u$ be the corresponding maximal-lifespan solution of (\ref{nls}) on  $I\times \mathbb{R}^d$. We have\\
	(a) If  $\|\nabla u_0\|_{L^2}<\|\nabla W\|_{L^2}$, then either  $u=W^-$ up to symmetries or u scatters in both time directions.\\
	(b) If  $\|\nabla u_0\|_{L^2}=\|\nabla W\|_{L^2}$, then  $u=W$ up to symmetries. \\
	(c) If  $\|\nabla u_0\|_{L^2}>\|\nabla W\|_{L^2}$ and  $u_0\in L^2(\mathbb{R}^d)$ is radial, then either  $|I|$ is finite or  $u=W^+$ up to symmetries.  
\end{theorem}
\begin{remark}
	The assertion that  $u=v$ up to symmetries means that there exist  $\lambda _0>0,\ \theta _0\in \mathbb{R}/2\pi \mathbb{Z}$ and  $t_0\in \mathbb{R}$ such that either 
	\begin{equation}
		u(t,x)=e^{i\theta _0}\lambda _0^{\frac{d-2}{2 }}v(\lambda _0^2t-t_0,\lambda _0x) \quad\text{or}\quad u(t,x)=e^{i\theta _0}\lambda _0^{\frac{d-2}{2 }}\overline{v}(\lambda _0^2t-t_0,\lambda _0x).\notag
	\end{equation}
\end{remark}
\begin{remark}
	Case (b) follows directly from the variational characterization of the ground state (see  Proposition \ref{P:GN}). Furthermore, using assumption  $E(u_0)=E(W)$, it follows that the assumptions  $ \|\nabla u_0\|_{L^2}< \|\nabla W\|_{L^2},  \|\nabla u_0\|_{L^2}> \|\nabla W\|_{L^2} $  do not depend on the choice of the initial time  $t_0$ (see Lemma \ref{l:coercive}).      We call  $ \|\nabla u_0\|_{L^2}< \|\nabla W\|_{L^2}  $ "subcritical case" and  $ \|\nabla u_0\|_{L^2}> \|\nabla W\|_{L^2}  $ "supercritical case". 
\end{remark}
\begin{remark}
	In the subcritical case, Theorem \ref{T2} does not require a radial assumption about  $u_0$.   
	 Note that for focusing energy critical NLS,  the corresponding results were proved in the radial case (\cite{CFRoud:threshold,2008GFA,2009JFA}). Recently, Su-Zhao \cite{Su} removed the radial assumption in  $d\ge5$ by using the interaction Morawetz estimate.     
\end{remark}
\begin{remark}
	Due to the singularity of  $|x|^{-b}$  at the origin,  the modulation analysis  in Lemma \ref{L:modulation} leads   to the restriction  $0<b<-\frac{(d-4)^2}{2}+1$.   See Claim \ref{C:221} for the details.  
\end{remark}

For focusing energy critical NLS, the ground state is given by the smooth bounded function 
\begin{equation}
	W_0(x)=(1+\frac{|x|^2}{d(d-2)})^{-\frac{d-2}{2}},\notag
\end{equation}
which  was also proved to be the  threshold of scattering   in the earlier work \cite{Dodson,KM,KV}, except for  $d=3$  within the class of radial data. They showed that for the Cauchy problem with   the initial datum  satisfies the a priori condition   $E(u_0)<E(W_0)$: (i) if  $ \|\nabla u_0\|_{L^2}< \|\nabla W_0\|_{L^2}  $, then the solution exists globally in time and scatters; (ii) if  $ \|\nabla u_0\|_{L^2}> \|\nabla W\|_{L^2}  $ and    $u_0$ is radial or  $u_0$ has finite variance ($xu_0\in L^2$), then finite time blowup  occurs.      Later, Duyckaerts-Merle \cite{2008GFA} studied the case of  $E(u_0)=E(W_0)$ in dimensions  $d=3,4,5$  for radial initial data.   They demonstrated the existence of  special threshold solutions  $W_0^\pm $  exponentially approach  the ground state  $W_0$, and proved that solutions with threshold energy either behave as in the subthreshold case, or it agrees with  $W_0,W_0^+,W_0^-$ up to the symmetries of the equation.

Our aim of this paper is to extend the classification results  of \cite{2008GFA}  for the classical Schr\"odinger equation ($b=0$) to the inhomogeneous case ($b>0$). 
For the   inhomogeneous NLS (\ref{nls}), the presence of the  inhomogeneity  $|x|^{-b}$  makes substantial differences. It breaks the translation symmetry of the equation and, at the same time, creates nontrivial singularity at the origin.  
As a  consequence  of  the singularity of  $|x|^{-b}$, we see that the ground state (\ref{W}),   which is also a stationary solution of  (\ref{nls}),  becomes singular at the origin. This will make the spectral analysis of the linearized operator (Proposition \ref{H}) and the estimates of the   modulation parameters  (Lemma \ref{L:modulation}) more difficult.    Similar situations also occur in the study of threshold solutions for the intercritical inhomogeneous NLS  and the energy critical  NLS with inverse square  potential (see \cite{Campos-Murphy,YZZ}).    To address this issue, we adapt the argument of \cite{Campos-Murphy,2008GFA,YZZ} within the Lorentz framework and work with a restricted range of \( b \) throughout the paper to ensure better regularity of \( W \) at the origin.

Furthermore, the construction of special threshold solutions  $W^\pm$ in Theorem \ref{T1} relies heavily on the expansion of \(J(W^{-1}v_k)\), where \(J\) is real-analytic for \(\{|z|<1\}\) (see (\ref{213})), and the function \(v_k\) is the difference between the approximate solution and the ground state (see (\ref{3272})). Unlike \cite{Campos-Murphy,CFRoud:threshold,2008GFA,2009JFA}, where either \(v_k\) is a Schwartz function or the nonlinearity is polynomial, in our case,   these conditions no longer hold. We therefore need to make additional efforts to expand \(J(W^{-1}v_k)\). Specifically, we work in  dimensions  $3\le d\le 5$ and use Sobolev embedding to prove that the eigenfunctions \(\mathcal{Y}_\pm \in L^\infty(\mathbb{R}^d)\)(see Lemma \ref{l:eigen}).   Then, using the spectral properties of the linearized operator, we inductively construct \( v_k \) such that it belongs to \( L_x^\infty \) and also  includes a time exponential decay factor. 
Consequently,  utilizing the real analyticity of  $J$ for  $|z|<1$, \(J(W^{-1}v_k)\) can be expanded when time is sufficiently large (see Lemma \ref{l:app} and the last part of Appendix \ref{App:3} for details).

On the other hand,  although the inhomogeneity  $|x|^{-b}$ breaks the translation symmetry (thus breaking conservation of momentum and Galilean invariance),  it  also brings some advantages. 
In fact, due to the decay of the inhomogeneous coefficient \(|x|^{-b}\) at infinity, we are able to construct scattering solutions associated with   initial data involving translation parameters \(x_n\) with \(|x_n| \to \infty\) (Proposition \ref{P:embed}). Therefore, in Section \ref{S:subcritical}, when we apply profile decomposition to nonscattering subcritical threshold solutions, the resulting moving spatial center \(x(t)\) must be bounded. Consequently, after a translation, we can choose  \(x(t) \equiv 0\) for the nonscattering compact solution. This effectively places us in the same situation as in the radial case. Hence, in Theorem \ref{T2}, we do not need to assume radial symmetry for the initial data. Recall that for the classical Schr\"odinger equation in dimensions \(d = 3, 4\), the initial data still needs to be radially symmetric (\cite{CFRoud:threshold,2008GFA,2009JFA,Su}).

The argument for Theorems \ref{T1} and \ref{T2} proceeds as follows:

The first main step (carried out in Section \ref{s:3}) is to construct  the  special threshold solutions  $W^\pm$  in Theorem \ref{T1}  and prove that they    exponentially  approach  the ground state in the positive time direction.  
The analysis starts with linearizing (\ref{nls}) around the ground state  $W$ and obtaining the linearized equation (\ref{Linear Eq}) with the  linearized  Schr\"odinger operator  $\mathcal{L}$.  Based on  the spectral properties  of  $\mathcal{L}$ (Lemma \ref{l:eigen}),  we construct approximate solutions  $W^a_k$ of (\ref{nls}) in suitable Lorentz spaces which is quite different from 
previous construction (Lemma \ref{l:app}).    Finally, in  Proposition \ref{P:approximate}, we upgrade  the approximate solutions  $W^a_k$ to true solutions  $W^a$ via a fixed point argument. 
The solutions  $W^a$ are essentially the special threshold solutions  $W^\pm$    appearing in Theorem \ref{T1} (Corollary \ref{cor1}).

The second main step (carried out in Section \ref{S:subcritical} and Section \ref{s4}) is to classify  forward-global threshold solutions  in certain scenarios.  
In  Proposition  \ref{p2}, we first show that if a forward-global subcritical  threshold solution  $u$  fails to scatter, then there exist  $\theta \in \mathbb{R} $ and  $\mu,c>0$ such that for all  $t\ge0$ 
\begin{equation}
	 \|u(t,x)-W_{[\theta ,\mu]}(x)\|_{\dot H^1}\lesssim  e^{-ct}, \label{2281}
\end{equation}   
where  $W_{[\theta ,\mu]}(x):=e^{i\theta }\lambda^{-\frac{d-2}{2}}W(\lambda^{-1}x)$.   
The idea is first to use the concentration compactness arguments to  show it satisfies a compactness property in  $\dot H^1$ (Proposition \ref{p1}).  Then we  combine the Virial estimates (Lemma \ref{L:virial}) and modulation analysis (Lemma \ref{L:modulation}) to establish the desired convergence   (\ref{2281}).    In  Proposition \ref{p3}, we prove the exponential convergence similar to (\ref{2281}) for forward-global supercritical  threshold solutions, relying once again on Virial estimates and modulation analysis.   The Virial estimates and the modulation analysis are prerequisite for both  arguments.   When the solution is away from the  orbit of the ground state,  we use the monotonicity formula arising from Virial to control the solution. 
When the solution  approaches the orbit of  $W$, we use modulation analysis to obtain  a suitable decomposition to control the solution. 
These estimates will eventually  ensure that the distance between the solution and the ground state
\begin{equation}
	\mathbf{d} (u(t)):=\left|  \int_{\mathbb{R} ^d} \left(|\nabla u(t,x)|^2-|\nabla W(x)|^2\right)dx \right| \notag
\end{equation}
converges to zero as \( t \rightarrow \infty \), thus   the modulation decomposition   (\ref{814w1}) holds  for sufficiently large \( t \). Furthermore, combining the estimates in Lemma \ref{L:modulation} with    $\lim_{t\rightarrow\infty }\mathbf{d} (u(t))=0$, we see that the parameters in   (\ref{814w1}) converge as \( t \rightarrow \infty \). Finally, by replacing the modulation parameters in (\ref{814w1}) with their corresponding limit functions, we obtain (\ref{2281}).

The last step (carried out in Section \ref{s:8}) is to use the positivity of the quadratic  form $Q$ (Lemma \ref{L:G})  to analyze
the property of the exponentially small solution of the linearized equation, then apply it to establish  the  uniqueness property for solutions converging exponentially to the ground state. The uniqueness property shows that 
for any threshold solution  $u$ that satisfies (\ref{2281}), there exists  $a\in \mathbb{R} $  such that  $u=(W^a)_{[\theta ,\mu]}$. 
   As a corollary of the uniqueness property (Corollary \ref{cor1}),  all of the solutions  $W^a$ constructed in Lemma \ref{l:app} are in fact  equal to  $W^+$ or  $W^-$   (up to the symmetries). 
Therefore, combining the first and second steps above, we can then obtain the desired  classification results  in Theorem \ref{T2}.   

The outline of the paper is as follows. In Section \ref{s:2}, we recall the Cauchy theory for (\ref{nls}) and the variational property of the ground state. We also  analyze  the linearized equation associated with  (\ref{nls}) near the ground  $W$, and  perform the   
detailed  spectral analysis of the linearized   operator  $\mathcal{L}$.   
 In Section \ref{s:3}, we use the contraction argument to construct    special threshold solutions  $W^\pm$ in Theorem \ref{T1}.  In Section \ref{s:4},  we perform the modulation analysis for solutions around the ground state. In Section \ref{s:5}, we establish Virial estimates by incorporating the modulation estimates developed in Section \ref{s:4}. In Section \ref{S:subcritical} and Section \ref{s4}, we study the forward-global threshold solutions  in the subcritical case and supercritical case,  respectively. 
In Section \ref{s:8}, we establish the uniqueness of the special solutions and this will imply the classification results of threshold solutions in  Theorem \ref{T2}. In Appendix, we show the asymptotic behavior of  $G(r)$ introduced in  Proposition \ref{H} and give the proof of Lemma \ref{l:eigen} and Lemma \ref{L:modulation}.

\section{Preliminaries}\label{s:2}
Throughout the paper, we fix  $3\le d\le 5$,  $0<b<-\frac{(d-4)^2}{2}+1$ and  $\alpha =\frac{4-2b}{d-2}$.    
We write  $A\lesssim  B$ to denote  $A\le CB$ for some  $C>0$. If  $A\lesssim B$ and  $B\lesssim A$, then we write  $A\approx B$.       Moreover, we use  $O(Y)$ to denote any quantity  $X$  such that  $|X|\lesssim  Y$. We use Japanese bracket  $\langle x \rangle $ to denote  $(1+|x|^2)^{\frac{1}{2}}$.       By  $H^s(\mathbb{R} ^d)$  we denote the usual Sobolev space of smoothness  $s$ in spatial variable.   We write $L_t^q L_x^r$ to denote the Banach space with norm
\[
\|u\|_{L_t^q L_x^r (\mathbb{R} \times \mathbb{R}^d)} := \left( \int_{\mathbb{R}} \left( \int_{\mathbb{R}^d} |u(t, x)|^r \, dx \right)^{q/r} \, dt \right)^{1/q},
\]
with the usual modifications when $q$ or $r$ are equal to infinity, or when the domain $\mathbb{R} \times \mathbb{R}^d$ is replaced by spacetime slab such as $I \times \mathbb{R}^d$. 
\subsection{Lorentz spaces and Strichartz estimates}
Let $f$ be a measurable function on $\mathbb{R}^d$. The distribution function of $f$ is defined by
\begin{equation}
	d_f(\lambda):= |\{x\in \mathbb{R}^d : |f(x)|>\lambda\}|, \quad \lambda>0,\notag
\end{equation}
where $|A|$ is the Lebesgue measure of a set $A$ in $\mathbb{R}^d$. The decreasing rearrangement of $f$ is defined by
\begin{equation}
	f^*(s):= \inf \left\{ \lambda>0 : d_f(\lambda)\leq s\right\}, \quad s>0.\notag
\end{equation}

\begin{definition}[Lorentz spaces] ~\\
	Let $0<r<\infty$ and $0<\rho\leq \infty$. The Lorentz space $L^{r,\rho}(\mathbb{R}^d)$ is defined by
	\begin{equation}
		L^{r,\rho}(\mathbb{R}^d):= \left\{ f \text{ is measurable on } \mathbb{R}^d : \|f\|_{L^{r,\rho}}<\infty\right\}, \notag
	\end{equation}
	where
	\[
	\|f\|_{L^{r,\rho}}:= \left\{
	\begin{array}{cl}
		( \frac{\rho}{r} \int_0^\infty (s^{1/r} f^*(s))^\rho \frac{1}{s}ds)^{1/\rho} &\text{ if } \rho <\infty, \\
		\sup_{s>0} s^{1/r} f^*(s) &\text{ if } \rho=\infty.
	\end{array}
	\right.
	\]
\end{definition}

We collect the following basic properties of $L^{r,\rho}(\mathbb{R}^d)$ in the following lemmas.  
\begin{lemma}[Properties of Lorentz spaces \cite{Grafakos}] ~
	\begin{itemize}
		\item For $1<r<\infty$, $L^{r,r}(\mathbb{R}^d) \equiv L^r(\mathbb{R}^d)$ and by convention, $L^{\infty,\infty}(\mathbb{R}^d)= L^\infty(\mathbb{R}^d)$.
		\item For $1<r<\infty$ and $0<\rho_1<\rho_2\leq \infty$, $L^{r,\rho_1}(\mathbb{R}^d)\subset L^{r,\rho_2}(\mathbb{R}^d)$. 
		\item For $1<r<\infty$, $0<\rho \leq \infty$, and $\theta>0$, $\||f|^\theta\|_{L^{r,\rho}} = \|f\|^\theta_{L^{\theta r, \theta \rho}}$.
		\item For $b>0$, $|x|^{-b} \in L^{\frac{d}{b},\infty}(\mathbb{R}^d)$ and $\||x|^{-b}\|_{L^{\frac{d}{b},\infty}} = |B(0,1)|^{\frac{b}{d}}$, where $B(0,1)$ is the unit ball of $\mathbb{R}^d$.
	\end{itemize}
\end{lemma}
\begin{lemma}[H\"older's inequality \cite{ONeil}] ~
	\begin{itemize}
		\item  Let $1<r, r_1, r_2<\infty$ and $1\leq \rho, \rho_1, \rho_2 \leq \infty$ be such that
		\[
		\frac{1}{r}=\frac{1}{r_1}+\frac{1}{r_2}, \quad \frac{1}{\rho} \leq \frac{1}{\rho_1}+\frac{1}{\rho_2}.
		\]
		Then for any $f \in L^{r_1, \rho_1}(\mathbb{R}^d)$ and $g\in L^{r_2, \rho_2}(\mathbb{R}^d)$
		\[
		\|fg\|_{L^{r,\rho}} \lesssim \|f\|_{L^{r_1, \rho_1}} \|g\|_{L^{r_2,\rho_2}}.
		\]
		\item Let  $1<r_1,r_2<\infty $ and  $1\le \rho_1,\rho_2\le \infty $  be such that 
		\begin{equation}
			1=\frac{1}{r_1}+\frac{1}{r_2},\quad 1\le \frac{1}{\rho_1}+\frac{1}{\rho_2}.\notag
		\end{equation}
		Then  for any $f \in L^{r_1, \rho_1}(\mathbb{R}^d)$ and $g\in L^{r_2, \rho_2}(\mathbb{R}^d)$
		\begin{equation}
			 \|fg\|_{L^1}\lesssim  \|f\|_{L^{r_1, \rho_1}} \|g\|_{L^{r_2,\rho_2}}.\notag
		\end{equation}
	\end{itemize}
\end{lemma}
\begin{lemma}[Interpolation \cite{ONeil}]\label{L:inter}
	Let  $1<p,p_1,p_2<\infty ,1\le q,q_1,q_2\le \infty $ and  $0<\theta <1$ be such that 
	\begin{equation}
		\frac{1}{p}=\frac{\theta }{p_1}+\frac{1-\theta }{p_2}\quad\text{and}\quad  \frac{1}{q}\le \frac{\theta }{q_1}+\frac{1-\theta }{q_2}.\notag
	\end{equation}  
	Then for any  $f\in L^{p_1,q_1}(\mathbb{R} ^d)\cap L^{p_2,q_2}(\mathbb{R} ^d)$ 
	\begin{equation}
		 \|f\|_{L^{p,q}}\lesssim   \|f\|_{L^{p_1,q_1}}^\theta  \|f\|_{L^{p_2,q_2}}^{1-\theta }.\notag   
	\end{equation}
	\end{lemma}
\begin{lemma}[Convolution inequality \cite{ONeil}] ~
	Let $1<r,r_1, r_2<\infty$ and $1\leq \rho, \rho_1, \rho_2 \leq \infty$ be such that
	\[
	1+\frac{1}{r} =\frac{1}{r_1}+\frac{1}{r_2}, \quad \frac{1}{\rho}\leq \frac{1}{\rho_1}+\frac{1}{\rho_2}.
	\]
	Then 	for any $f\in L^{r_1, \rho_1}(\mathbb{R}^d)$ and $g\in L^{r_2, \rho_2}(\mathbb{R}^d)$ 
	\[
	\|f\ast g\|_{L^{r,\rho}} \lesssim \|f\|_{L^{r_1,\rho_1}} \|g\|_{L^{r_2,\rho_2}}. 
	\]
\end{lemma}
Next,  in Lemma \ref{L:sobolev}--Lemma \ref{L:6141},  we  recall the Sobolev embedding, product rule, and chain rule in Lorentz spaces. We start by introducing the following definition.

Let  $s\ge0,1<r<\infty $ and  $1\le\rho\le \infty $. We define the Sobolev-Lorentz spaces 
\begin{equation}
	W^sL^{r,\rho}(\mathbb{R} ^d)=\left\{f\in \mathcal{S}'(\mathbb{R} ^d):(1-\Delta )^{s/2}f\in L^{r,\rho}(\mathbb{R} ^d) \right\},\notag
\end{equation}  
\begin{equation}
	\dot W^sL^{r,\rho}(\mathbb{R} ^d)=\left\{f\in \mathcal{S}'(\mathbb{R} ^d):(-\Delta )^{s/2}f\in L^{r,\rho}(\mathbb{R} ^d) \right\},\notag
\end{equation}
where  $\mathcal{S}'(\mathbb{R} ^d)$ is the space of tempered distributions on  $\mathbb{R} ^d$ and 
\begin{equation}
	(1-\Delta )^{s/2}f=\mathcal{F}^{-1}\left((1+|\xi|^2)^{s/2}\mathcal{F} (f)\right),\qquad (-\Delta )^{s/2}f=\mathcal{F} ^{-1}(|\xi|^s\mathcal{F} (f))\notag
\end{equation}  
with  $\mathcal{F} $ and  $\mathcal{F} ^{-1}$ the Fourier and its inverse Fourier transforms respectively. The spaces  $W^sL^{r,\rho}(\mathbb{R} ^d)$ and  $\dot W^sL^{r,\rho}(\mathbb{R} ^d)$ are endowed respectively with the norms
\begin{equation}
	 \|f\|_{W^sL^{r,\rho}} = \|f\|_{L^{r,\rho}}+ \|(-\Delta )^{s/2}f\|_{L^{r,\rho}},\qquad  \|f\|_{\dot WL^{r,\rho}} = \|(-\Delta )^{s/2}f\|_{L^{r,\rho}}.\notag   
\end{equation}    
For simplicity, when  $s=1$ we write  $WL^{r,\rho}:=W^1L^{r,\rho}(\mathbb{R} ^d)$  and  $\dot WL^{r,\rho}:=\dot W^1L^{r,\rho}(\mathbb{R} ^d)$.  

\begin{lemma}[Sobolev embedding\cite{DinhKe}]\label{L:sobolev}
	Let  $1<r<\infty ,1\le \rho\le \infty $ and  $0<s<\frac{d}{r}$. Then 
	\begin{equation}
		 \|f\|_{L^{\frac{dr}{d-sr},\rho}}\lesssim   \|(-\Delta )^{s/2}f\|_{L^{r,\rho}} \quad\text{for any}\quad   f\in \dot W^sL^{r,\rho}(\mathbb{R} ^d).\notag
	\end{equation}  
	\end{lemma}
\begin{lemma}[Product rule\cite{Cruz}]\label{L:leibnitz}
	Let  $s\ge0,1<r,r_1,r_2,r_3,r_4,<\infty $, and   $1\le\rho,\rho_1,\rho_2,\rho_3,\rho_4\le\infty $ be such that 
	\begin{equation}
		\frac{1}{r}=\frac{1}{r_1}+\frac{1}{r_2}=\frac{1}{r_3}+\frac{1}{r_4},\qquad \frac{1}{\rho}=\frac{1}{\rho_1}+\frac{1}{\rho_2}=\frac{1}{\rho_3}+\frac{1}{\rho_4}.\notag
	\end{equation} 
	Then for any  $f\in \dot W^sL^{r_1,\rho_1}(\mathbb{R} ^d)\cap L^{r_3,\rho_3}$  and  $g\in \dot W^sL^{r_4,\rho_4}\cap L^{r_2,\rho_2}(\mathbb{R} ^d)$, we have 
	\begin{equation}
		 \|(-\Delta )^{s/2}(fg)\|_{L^{r,\rho}} \lesssim   \|(-\Delta )^{s/2}f\|_{L^{r_1,\rho_1}} \|g\|_{L^{r_2,\rho_2}}+ \|f\|_{L^{r_3,\rho_3}} \|(-\Delta )^{s/2}g\|_{L^{r_4,\rho_4}}.\notag    
	\end{equation}
	\end{lemma}
	\begin{lemma}[Chain rule\cite{Aloui}]\label{L:6141}
		Let  $s\in [0,1],F\in C^1(\mathbb{C},\mathbb{C})$ and  $1<p,p_1,p_2<\infty $,  $1\le q,q_1,q_2<\infty $ be  such that 
		\begin{equation}
			\frac{1}{p}=\frac{1}{p_1}+\frac{1}{p_2},\qquad\frac{1}{q}=\frac{1}{q_1}+\frac{1}{q_2}.\notag
		\end{equation}   
		Then
		\begin{equation}
			\|(-\Delta )^{s/2}F(f)\|_{L^{p,q}(\mathbb{R} ^d)}\lesssim  \|F'(f)\|_{L^{p_1,q_1}(\mathbb{R} ^d)} \|(-\Delta )^{s/2}f\|_{L^{p_2,q_2}(\mathbb{R} ^d)}.\notag   
		\end{equation}
	\end{lemma}
	At the end of this subsection, we recall the Strichartz estimates in the Lorentz spcae.   
\begin{definition}[Admissibility] 
	A pair $(p,n)$ is said to be Schr\"odinger admissible, for short $(p,n)\in \Lambda $, where
	\begin{equation}
		\Lambda=\left\{(p,n): 2\le p,n\le \infty,\  	\frac{2}{p}+\frac{d}{n}=\frac{d}{2},\ (p,n,d)\neq (2,\infty ,2)\right\}.\notag
	\end{equation}
\end{definition}

\begin{proposition}[Strichartz estimates\cite{Keel-Tao,Taggart}]\label{P:SZ}\ \\
	\begin{itemize}
		\item Let $(m,n)\in \Lambda $ with $r<\infty$. Then for any $f\in L^2(\mathbb{R}^d)$
		\begin{align}  
			\|e^{it\Delta }f\|_{L^m_t L^{n,2}_x(\mathbb{R}\times \mathbb{R}^d)} \lesssim \|f\|_{L^2_x(\mathbb{R}^d)}.\notag
		\end{align}
		
		\item Let $(q_1, r_1), (q_2,r_2)\in \Lambda $ with $r_1, r_2<\infty$, $t_0\in \mathbb{R}$ and $I\subset \mathbb{R}$ be an interval containing $t_0$. Then 		for any $F\in L_t^{q_2'}L^{r_2',2}_x(I\times \mathbb{R}^d)$ 
		\begin{align} 
			\left\|\int_{t_0}^t e^{i(t-\tau)\Delta } F(\tau) d\tau\right\|_{L^{q_1}_tL^{r_1,2}_x(I\times \mathbb{R}^d)} \lesssim \|F\|_{L^{q_2'}_tL^{r_2',2}_x(I\times \mathbb{R}^d)}.\notag
		\end{align}
	\end{itemize}
\end{proposition}

\subsection{Preliminaries on the Cauchy problem.}\label{s:2.2}
In this subsection, we recall some results on the Cauchy problem  (\ref{nls}). Let  $I$ be an interval and denote 
 \begin{equation}
 	\|f\|_{S(I)}:=\|f\|_{L^ \gamma _tL_x^{p,2}(I\times \mathbb{R} ^d)},\ \|f\|_{Z(I)}:=\|\nabla f\|_{L^ \gamma _tL_x^{\rho,2}(I\times \mathbb{R} ^d)},\ \|f\|_{N(I)}:=\|f\|_{L^2_tL_x^{\frac{2d}{d+2},2}(I\times \mathbb{R} ^d)},\notag
 \end{equation}
 where 
  \begin{equation}
 	\gamma:=2(\alpha +1),\qquad \rho:=\frac{2d(\alpha +1)}{d+2-2b+2\alpha },\qquad p:=\frac{2d(\alpha +1)}{d-2b}\notag
 \end{equation}
 satisfy 
 	\begin{equation}
 	\frac{d+2}{2d}=\frac{b}{d}+\frac{\alpha }{p}+\frac{1}{\rho}=\frac{b+1}{d}+\frac{\alpha +1}{p},\ \frac{1}{2}=\frac{\alpha +1}{\gamma }.\label{zb}
 \end{equation}

\begin{theorem}[\cite{Aloui,Liu-Zhang}]\label{T:CP}
	For any  $u_0\in \dot H^1(\mathbb{R} ^d)$ and  $t_0\in \mathbb{R} $, there exists  a unique  maximal solution  $u:(-T_-(u_0),T_+(u_0))\times \mathbb{R} ^d\rightarrow \mathbb{C}$  to (\ref{nls}) with  $u(t_0)=u_0$.     This solution also has the following properties: \\
	(a) If  $T_+=T_+(u_0)<\infty $, then  $ \|u\|_{S(0,T_+)} =+\infty $. An analogous result holds for  $T_-(u_0)$.  \\
	(b) If  $ \|u\|_{S(0,T_+)}<+\infty  $, then  $T_+=\infty $ and  $u$ scatters as  $t\rightarrow+\infty $.      An analogous result holds for  $T_-(u_0)$. \\
	(c)  For any $\psi\in\dot H^1(\mathbb{R}^d)$, there exist $T>0$ and a solution $u:(T,\infty)\times\mathbb{R}^d\to\mathbb{C}$ to \eqref{nls} obeying $e^{-it\Delta}u(t)\to \psi$ in $\dot H^1$ as $t\to\infty$.  The analogous statement holds backward in time.\\
	(d) There exists  $\eta_0>0$ such that if  $ \|u_0\|_{\dot H^1}<\eta_0 $, then  $u$ is a global solution and scatters to  $0$ in  $\dot H^1(\mathbb{R} ^d)$.            
\end{theorem}

	Next, we record the following stability result of the Cauchy problem (\ref{nls}). 

	\begin{proposition}[\cite{Liu-Zhang}] \label{P:stab} Suppose $\tilde u:I\times\mathbb{R}^d\to\mathbb{C}$ obeys
		\begin{equation}
			\|\tilde u\|_{L_t^\infty \dot H_x^1(I\times\mathbb{R}^d)} + \|\tilde u\| _{S(I)}\leq E<\infty. \notag
		\end{equation}
		There exists $\varepsilon _1 = \varepsilon_1(E)>0$ such that if
		\begin{align*}
			\| \nabla\bigl\{(i\partial_t+\Delta)\tilde u + |x|^{-b}|\tilde u|^\alpha \tilde u\bigr\} \|_{L^2_tL_x^{\frac{2d}{d+2},2}(I\times \mathbb{R} ^d)}& \leq \varepsilon<\varepsilon_1,\\
			\|e^{i(t-t_0)\Delta}[u_0-\tilde u|_{t=t_0}]\|_{S(I)}& \leq \varepsilon<\varepsilon_1,
		\end{align*}
		for some $t_0\in I$ and $u_0\in\dot H^1(\mathbb{R}^d)$ with $\|u_0-\tilde u|_{t={t_0}}\|_{\dot H^1}\lesssim_E 1$, then there exists a unique solution $u:I\times\mathbb{R}^d\to\mathbb{C}$ to (\ref{nls}) with $u|_{t=t_0}=u_0$, which satisfies
		\begin{equation}
			\|u-\tilde u\|_{S(I)}\lesssim \varepsilon \quad\text{and}\quad \|u\|_{L_t^\infty \dot H_x^1(I\times\mathbb{R}^d)}+\|u\|_{S(I)}\lesssim_E 1.\notag
		\end{equation}
	\end{proposition}
	
	Given \(\phi \in \dot H^1(\mathbb{R}^d)\) and a diverging sequence \(\{x_n\}_{n=1}^{\infty }\),     we can deduce from Proposition   \ref{P:stab} that,  for  \(n\)  sufficiently large, the solution  $v_n$ to the Cauchy problem (\ref{nls}) with initial data \(\phi(x - x_n)\) exists globally and scatters. In fact, by approximating \(e^{it\Delta} \phi(x)\) with a smooth, compactly supported function \(\psi(t, x)\) and  using the decay of  \(|x|^{-b}\) at infinity, we can directly verify that the scattering solutions \(\widetilde{u}_n(t, x) =: e^{it\Delta} \phi(x - x_n)\)  are 	 good approximate solutions of  $v_n$ when $ n $ is  sufficiently large.  
	Furthermore, performing a spatial rescaling, we can obtain the following result:
	\begin{proposition}[\cite{Liu-Zhang}]\label{P:embed} Let $\lambda_n \in (0,\infty)$, $x_n\in \mathbb{R}^d$, and $t_n\in\mathbb{R}$ satisfy
		\[
		\lim_{n\to\infty} \tfrac{|x_n|}{\lambda_n} = \infty \quad\text{and}\quad t_n\equiv 0  \quad\text{or}\quad  t_n\to\pm\infty.
		\]
		Let $\phi\in \dot H^1(\mathbb{R}^d)$ and define
		\begin{equation}
			\phi_n(x) = \lambda_n^{-\frac{d-2}2}[e^{it_n\Delta}\phi](\tfrac{x-x_n}{\lambda_n}). \notag
		\end{equation}
		Then for all $n$ sufficiently large, there exists a global solution $v_n$ to \eqref{nls} satisfying
		\[
		v_n(0) = \phi_n  \quad\text{and}\quad   \|v_n\|_{S(I)} \lesssim  \|\nabla v_n\|_{Z(I)}\lesssim 1,
		\] 	
		with implicit constant depending only on $\|\phi\|_{\dot H^1}$.
	\end{proposition}

\subsection{Variational property of  the ground state  $W$. }
The ground state  (\ref{W})  solves  the nonlinear elliptic equation 
\begin{equation}
	\Delta W+|x|^{-b}W^{\alpha +1}=0,\label{Eq:W}
\end{equation}
and is characterized as the optimizers in Sobolev embedding inequality (Proposition  \ref{P:GN}).

Multiplying (\ref{Eq:W}) by   $W$ and integrating by parts yields the Pohozhaev's identity
\begin{equation}
	\|\nabla W\|_{L^2}^2= \||x|^{-b}W^{\alpha +2}\|_{L^1}.\label{identity:pohozhaev}
\end{equation}  

\begin{proposition}\label{P:GN}
	Let  $d\ge 3$. Then for any  $f\in \dot H^1(\mathbb{R} ^d)$ 
	\begin{equation}
		\||x|^{-b}|f|^{\alpha +2}\|_{L^1}\le  \||x|^{-b}W^{\alpha +2}\|_{L^1} \|\nabla W\|_{L^2}^{-(\alpha +2)} \|\nabla f\|_{L^2}^{\alpha +2}   .\label{2161}
	\end{equation}   
	Moreover, equality holds in  (\ref{2161}) if and only if  $f(x)=zW(\lambda x)   $  for some   $z\in \mathbb{C}$  and some  $\lambda>0$.  
\end{proposition}

The main tool that we need to prove  Proposition  \ref{P:GN}  is the following bubble decomposition  of \cite{Gerard}. 
\begin{lemma}[\cite{Gerard}]\label{L:bubble}
	For any bounded sequence  $\left\{f_n \right\}$ in  $\dot H^1$, the following holds up to a subsequence. There exist  $J^{*}\in \mathbb{N} \cup\left\{\infty  \right\}$; profiles  $\phi^j\in \dot H^1\setminus\left\{0 \right\}$; scaling parameters   $\lambda_n^j\in (0,\infty )$; space translation parameters  $x_n^j\in \mathbb{R} ^d$;  and remainders  $w_n^J$ so that the following decomposition holds for  $1\le J\le J^*$: 
	\begin{equation}
		f_n(x)=\sum_{j=1}^{J} (\lambda_n^j)^{-\frac{d-2}{2}}\phi^j(\frac{x-x_n^j}{\lambda_n^j})+w_n^J,\label{11231}
	\end{equation}       
	satisfying
	\begin{equation}
		\limsup_{J\rightarrow J^*}\limsup _{n\rightarrow \infty }  \|w_n^J\|_{L^{\frac{2d}{d-2}}} =0,\label{1124x2} 
	\end{equation}
	\begin{equation}
		\lim_{n\rightarrow\infty }\{ \|\nabla f_n\|_{L^2}^2-\sum_{j=1}^{J} \|\nabla \phi^j\|_{L^2}^2- \|\nabla w_n^J\|_{L^2}^2   \}=0,\label{11232}
	\end{equation}
	\begin{equation}
		\lim_ {n\rightarrow \infty } \{\int |x|^{-b}|f_n|^{\alpha +2}dx-\sum_{j=1}^{J}\int |x|^{-b}|\phi^j(x-\frac{x_n^j}{\lambda_n^j})|^{\alpha +2}dx-\int |x|^{-b}|w_n^J|^{\alpha +2}dx\}=0.\label{11233}
	\end{equation}
	In addition, we may assume that either  $x_n^j\equiv0$ or  $\frac{|x_n^j|}{\lambda_n^j}\rightarrow \infty $.   
\end{lemma}
\begin{remark}
	In \cite{Gerard}, the potential energy decoupling in (\ref{11233}) is given in terms of the  $L^{\frac{2d}{d-2}}$  norm. 
	However, the same arguments suffices to establish decoupling for the functional appearing in (\ref{11233}).  
\end{remark}
\begin{proof}[\textbf{Proof of  Proposition \ref{P:GN}}.] 
	By the standard rearrangement inequalities (c.f. \cite[Ch. 3]{Lieb}), the optimal constant can determined by the consideration of
	radial functions alone. Moreover, as   $\int |x|^{-b}|f(x)|^{\alpha +2}$  is strictly monotone under rearrangement, any optimizer must be radial.
	Let  $f_n$  be an optimizing sequence of radial functions for the problem
	 \begin{equation}
	 	\text{maximize} \qquad J(f):= \||x|^{-b}|f|^{\alpha +2}\|_{L^1}\div  \|\nabla f\|_{L^2}^{\alpha +2}\quad\text{subject to the constraint }\quad    \|f\|_{\dot H^1}=1.\notag 
	 \end{equation}
	 Applying Lemma \ref{L:bubble}, and passing to the requisite subsequence yields 
	 \begin{equation}
	 	\sup _{f}J(f)=\lim_{n\rightarrow\infty }J(f_n)\le \limsup_{J\rightarrow J^*}\limsup _{n\rightarrow\infty }  \sum_{j=1}^{J} \int |x|^{-b}|\phi^j(x-\frac{x_n^j}{\lambda_n^j})|^{\alpha +2}dx\le \sup_f J(f)\sum_{j=1}^{J^*} \|\phi^j\|_{\dot H^1}^{\alpha +2},\label{2162} 
	 \end{equation}
	 where in the first inequality we  used (\ref{11233}) and 
	 	\begin{equation}
	 	\limsup_{J\rightarrow J^*}\limsup _{n\rightarrow \infty }  \int |x|^{-b}|w_n^J|^{\alpha +2}dx=0.\label{241}
	 \end{equation}
	In fact,  using H\"older's inequality  and the interpolation in Lemma \ref{L:inter}, we have 
	\begin{equation}
		\int |x|^{-b}|w_n^J|^{\alpha +2}dx\lesssim   \||x|^{-b}\|_{L^{\frac{d}{b},\infty }} \|w_n^J\|^{\alpha +2}_{L^{\frac{2d}{d-2},\alpha +2}}\lesssim    \|w_n^J\|^{\alpha +2-\frac{d\alpha }{2}}_{L^{\frac{2d}{d-2},2}}\|w_n^J\|^{\frac{d\alpha }{2}}_{L^{\frac{2d}{d-2},\frac{2d}{d-2}}},\notag
	\end{equation} 
	where  $\alpha +2-\frac{d\alpha }{2}>0$ by the assumptions made at the begining of section \ref{s:2}.  
	This inequality together with  (\ref{1124x2}) and the embedding  $ \|w_n^J\|_{L^{\frac{2d}{d-2},2}} \lesssim  \|\nabla w_n^J\|_{L^2}\lesssim 1 $ yields (\ref{241}). 

On the other hand,  the kinetic decoupling (\ref{11232})  guarantees that
\begin{equation}
	\sum_{j=1}^{J^*}  \|\phi^j\|_{\dot H^1}^2\le \limsup_{n\rightarrow\infty }  \|\nabla f_n\|_{L^2}=1. \label{2163}
\end{equation}
Combining (\ref{2162}) and (\ref{2163}), we see that  $J^*=1$ and  $ \|\phi^1\|_{\dot H^1}=1 $.   
Since  $f_n$  is an optimizing sequence, $\phi^1$ must be an optimizer for  $J$  and hence for the embedding  (\ref{2161}).  

 The existence of optimizers is known, we turn to their characterization.  Let   $0\neq f\in \dot H^1$   denote such an  optimizer.  Replacing  $f$  by  $\beta f$  for some  $\beta>0$ , if necessary, we may assume that 
\begin{equation}
	 \|f\|_{\dot H^1(\mathbb{R} ^d)}^2=\int _{\mathbb{R} ^d}|x|^{-b}|f(x)|^{\alpha +2}dx.\notag 
\end{equation}
By assumption,   $f$ maximizes  $\int |x|^{-b}|f(x)|^{\alpha +2}dx$  among all functions that  subject to the constraint  $ \|f\|_{\dot H^1}=1 $.   Thus, it  satisfies  the Euler–Lagrange equation   $\Delta f+|x|^{-b}|f|^\alpha f=0$; here we exploited the normalization of  $f$  to determine the Lagrange multiplier.  Let  $\phi=r^{\frac{d-2}{2}}f$. Recalling that all optimizers must be radial, we obtain the equation for  $\phi:$ 
\begin{equation}
	\partial_{rr}\phi=-\frac{\partial_{r}\phi}{r}+(\frac{d-2}{2})^2\frac{\phi}{r^2}-\frac{1}{r}|\phi|^\alpha \phi.\notag
\end{equation}  
Multiplying the above equation  by  $2r^2\partial_{r}\phi$ and then integrating, we get 
\begin{equation}
	r^2(\partial_{r}\phi)^2=(\frac{d-2}{2})^2\phi^2-\frac{2}{\alpha +2}|\phi|^{\alpha +2}+c. \label{216w1}
	\end{equation} 
	As  $f\in \dot H^1(\mathbb{R} ^d)$ , there is a sequence  $r_n\rightarrow\infty $ so that   $|(\partial_{r}f_n)(r_n)|+|\frac{1}{r_n}f(r_n)|=o(r_n^{-d/2})$.  
	 Thus, the constant  $c$ in (\ref{216w1})   is  zero.  The resulting first-order ODE is separable:
	 \begin{equation}
	 	\int \frac{d\phi}{\sqrt{(\frac{d-2}{2})^2\phi^2-\frac{2}{\alpha +2}|\phi|^{\alpha +2}}}=\pm \int \frac{dr}{r}.\notag
	 \end{equation} 
	  Letting  $g=\sqrt{1-\frac{4}{(d-2)(d-b)}|\phi|^\alpha }$ and carrying out the  requisite integrals, we then deduce that   $f(x)=\lambda^{\frac{d-2}{2}}W(\lambda x)$ for some  $\lambda>0$.
\end{proof}

Finally, we  show the persistence of the Kinetic energy. 
\begin{lemma}
	\label{l:coercive}
	Let  $u\in C(I,\dot H^1(\mathbb{R} ^d))$  be a solution of (\ref{nls}) with initial data  $u_0$, and  $I=(-T_-,T_+)$ its maximal interval of existence. Assume that  $E(u_0)=E(W)$\\
	(a) if  $ \|\nabla u_0\| _{L^2}< \|\nabla W\| _{L^2}   $, then  $ \|\nabla u(t)\|_{L^2}< \|\nabla W\|_{L^2}  $ for  $t\in I$;\\
	(b) if  $ \|\nabla u_0\|_{L^2}= \|\nabla W\|_{L^2}  $, then  $u=W$ up to the symmetry of the equation;\\
	(c) if  $ \|\nabla u_0\|_{L^2}> \|\nabla W\|_{L^2}  $, then  $ \|\nabla u(t)\|_{L^2}> \|\nabla W\|_{L^2}  $    for  $t\in I$.  
\end{lemma}
\begin{proof}
	Case  $(b)$ is a direct consequence of the variational characterization of  $W$ given by  Proposition \ref{P:GN}.  We now prove (a). If  $ \|\nabla u(t_0)\|_{L^2}= \|\nabla W\|_{L^2}  $ for some  $t_0\in I$, then  (b) implies that    $u=W$ up to the symmetry of the equation. This contradicts $ \|\nabla u_0\| _{L^2}< \|\nabla W\| _{L^2}   $. The proof of  (c) is similar to (a), so we omit the details.   
\end{proof}
 \subsection{The linearized operator around the ground state  $W$.}
 Consider  a solution  $u$ of
(\ref{nls}) close to  $W$  and let  $v(t,x):=u(t,x)-W(x)$, then  $v=v_1+iv_2$ satisfies the Schr\"odinger equation  
 \begin{equation}
	i\partial_{t}v+\Delta v+\mathcal{V}(v)+iR(v)=0,\notag
\end{equation}
where the linear operator 
 \begin{equation}
	\mathcal{V}(v):=(\alpha +1)|x|^{-b}W^\alpha  v_1+i|x|^{-b}W^\alpha  v_2, \label{V}
\end{equation}
and the remainder 
\begin{equation}
	R(v):=-i|x|^{-b}|W+v|^\alpha (W+v)+i|x|^{-b}W^{\alpha +1}+i(\alpha +1)|x|^{-b}W^\alpha v_1-|x|^{-b}W^\alpha v_2.\label{Rv}
\end{equation}

We will always  write equally    $f=\left(\begin{array}{    c   }
	f_1 \\ 
	f_2  
\end{array}\right)$ for a complex valued
function  $f$ with real part  $f_1$ and imaginary part  $f_2$.  Then  $v$ is a solution of the equation  
\begin{equation}
	\partial_{t}v+\mathcal{L}(v)+R(v)=0,\qquad  \mathcal{L}:=\left(\begin{array}{  c   c   }
		0	& \Delta +|x|^{-b}W^{\alpha } \\ 
		-\Delta -(\alpha +1)|x|^{-b}W^\alpha 	& 0   
	\end{array}\right) \label{Linear Eq}
\end{equation}

We will use the following linear and nonlinear estimates  about the linearized equation (\ref{Linear Eq}). Recall that the function spaces  $N(I),Z(I)$  are introduced in subsection \ref{s:2.2}.  
 \begin{lemma}[Linear estimates]
 	\label{lestimate}
Let  $\mathcal{V}$ be defined by (\ref{V}) and  $I$ be a finite time interval of length  $|I|$.  Then 
 	\begin{equation}
 		\|\nabla \mathcal{V}(f)\|_{N(I)}\lesssim  |I|^{\frac{\alpha }{\gamma }}  \|f\|_{Z(I)}.    \notag
 	\end{equation}
 \end{lemma}
 \begin{proof}
 By  $|\nabla W|\lesssim  |x|^{-1}W$, we get the pointwise bound 
 \begin{equation}
 	|\nabla \mathcal{V}(f)|\lesssim  |x|^{-b-1}W^\alpha |f|+|x|^{-b}W^\alpha |\nabla f|.\notag
 \end{equation}
It then follows  from (\ref{zb}), H\"older's inequality and the embedding  $\dot W^{1,\rho}(\mathbb{R} ^d)\hookrightarrow L^p(\mathbb{R} ^d)$ that 
\begin{equation}
		 \|\nabla \mathcal{V}(f)\|_{N(I)}\lesssim \|W\|_{L_t^\gamma L_x^{p,2}}^\alpha ( \|f\|_{L^\gamma _tL^{p,2}_x}+ \|\nabla f\|_{L_t^\gamma L_x^{\rho,2}}    )\lesssim   |I|^{\frac{\alpha }{\gamma }}  \|f\|_{Z(I)}.   \notag
\end{equation}
Lemma \ref{lestimate} is proved.  
 \end{proof}
 \begin{lemma}[Non-linear estimates]
 	\label{nestimate}
Let  $R$ be defined by (\ref{Rv}) and  $I$ be a finite time interval.  Then 
 	\begin{equation}
 		 \|R(f)-R(g)\| _{L^{\frac{2d}{d+2},2}}\lesssim   \|f-g\|_{L^{\frac{2d}{d-2},2}}\left[  \|f\|_{L^{\frac{2d}{d-2},2}}+ \|g\|_{L^{\frac{2d}{d-2},2}}+ \|f\|_{L^{\frac{2d}{d-2},2}}^\alpha + \|g\|_{L^{\frac{2d}{d-2},2}}^\alpha    \right] \label{2121}
 	\end{equation}
 and 
 	\begin{eqnarray}
 		\|\nabla R(f)-\nabla R(g)\|_{N(I)}\lesssim  \|f-g\|_{Z(I)} \left[|I|^{\frac{\alpha -1}{\gamma }} ( \|f\|_{Z(I)} + \|g\|_{Z(I)}   )+ \|f\|_{Z(I)} ^\alpha + \|g\|_{Z(I)} ^\alpha    \right]. \notag
 	\end{eqnarray}
 	Recall that  $\alpha \ge1 $ by the assumptions made at the begining of section \ref{s:2}.   
 \end{lemma}
 \begin{proof}
 	We have 
 \begin{eqnarray}
 		R(f)&=&-i|x|^{-b}\left[|W+v|^\alpha (W+v)-W^{\alpha +1}-\frac{\alpha +2}{2}W^\alpha f-\frac{\alpha }{2}W^\alpha \overline{f}\right]\notag\\
 		&=&-i|x|^{-b}W^{\alpha +1}J(W^{-1}f),\label{211}
 \end{eqnarray}
 	where  $J$ is the function defined on  $\mathbb{C}$ by 
 	\begin{equation}
 		J(z)=|1+z|^\alpha (1+z)-1-\frac{\alpha +2}{2}W^\alpha z-\frac{\alpha }{2}W^\alpha \overline{z}.\label{212}
 	\end{equation}  
 	Since $\alpha \ge1$,  $J$ is of class  $C^2$ on  $\mathbb{C}$ and  $J(0)=\partial_{z}J(0)=\partial_{\overline{z}}J(0)=0$.      
Hence 
 	\begin{equation}
 		|J'(z)|\lesssim  |z|+|z|^\alpha \quad\text{and}\quad |J''(z)|\lesssim  1+|z|^{\alpha -1}.\label{1281}
 	\end{equation}   
 	By (\ref{1281}),  we get the pointwise bound 
 	\begin{equation}
 		|R(f)-R(g)|\lesssim  |x|^{-b}\left[ W^{\alpha -1}(|f|+|g|)+|f|^\alpha +|g|^\alpha \right]|f-g|,\notag
 	\end{equation}
 	which yields (\ref{2121}) using H\"older's inequality 
 	\begin{equation}
 		 \||x|^{-b}abc^{\alpha -1}\|_{L^{\frac{2d}{d+2},2}}\lesssim   \||x|^{-b}\|_{L^{\frac{d}{b},\infty }} \|a\|_{L^{\frac{2d}{d-2},2}}    \|b\|_{L^{\frac{2d}{d-2},2}}    \|c\|_{L^{\frac{2d}{d-2},2}}   ^{\alpha -1}.\notag  
 	\end{equation}
 	On the other hand, by (\ref{1281}) and $|\nabla W|\lesssim  |x|^{-1}W$, we have 
 	\begin{eqnarray}
 		&& |\nabla R(f)-\nabla R(g)|\notag\\
 		&\lesssim  &|x|^{-b-1}[W^{\alpha -1}(|f|+|g|)+|f|^\alpha +|g|^\alpha ]|f-g|\notag\\
 		 &&+|x|^{-b}[W^{\alpha -1}(|\nabla f|+|\nabla g|)+(|f|^{\alpha -1}+|g|^{\alpha -1})(|\nabla f|+|\nabla g|)]|f-g|\notag\\
 		 &&+|x|^{-b} [W^{\alpha -1}(|f|+|g|)+|f|^\alpha +|g|^\alpha ]|\nabla (f-g)|.\notag
 	\end{eqnarray} 
It then follows from (\ref{zb}),  H\"older's inequality and Sobolev's embedding that 
 	\begin{eqnarray}
 		 &&\|\nabla R(f)-\nabla R(g)\|_{N(I)}\notag\\
 		 &\lesssim & \left[ \|W\|_{L^\gamma _t L_x^{p,2}} ^{\alpha -1}( \| f\|_{L^\gamma _t L_x^{p,2}}   + \|g\|_{L^\gamma _t L_x^{p,2}}) + \|f\|_{L^\gamma _t L_x^{p,2}}^\alpha + \|g\|_{L^\gamma _t L_x^{p,2}}^\alpha\right] \|f-g\|_{L^\gamma _t L_x^{p,2}} \notag\\
 		 &&+\left [ \|W\|_{L^\gamma _t L_x^{p,2}} ^{\alpha -1}( \|\nabla f\|_{L^\gamma _t L_x^{\rho,2}} + \|\nabla g\|_{L^\gamma _t L_x^{\rho,2}} ) + (\|f\|_{L^\gamma _t L_x^{p,2}} ^{\alpha -1}+ \|g\|_{L^\gamma _t L_x^{p,2}}^{\alpha -1}) \right. \notag\\
 		 &&\qquad \times  \left.( \|\nabla f\|_{L^\gamma _t L_x^{\rho,2}}+ \|\nabla g\|_{L^\gamma _t L_x^{\rho,2}}  )\right] \|f-g\|_{L^\gamma _t L_x^{p,2}}\notag\\
 		 &&+[ \|W\|_{L^\gamma _t L_x^{p,2}}^{\alpha -1}( \|f\|_{L^\gamma _t L_x^{p,2}}+ \|g\|_{L^\gamma _t L_x^{p,2}}  )+ \|f\|_{L^\gamma _t L_x^{p,2}}^\alpha + \|g\|_{L^\gamma _t L_x^{p,2}}^\alpha    ] \|\nabla (f-g)\|_{L^\gamma _t L_x^{\rho,2}}\notag\\
 		 &\lesssim &     \|f-g\|_{Z(I)} \left[|I|^{\frac{\alpha -1}{\gamma }} ( \|f\|_{Z(I)} + \|g\|_{Z(I)}   )+ \|f\|_{Z(I)} ^\alpha + \|g\|_{Z(I)} ^\alpha    \right],\notag
 	\end{eqnarray}
 	which yields the second estimate in Lemma \ref{nestimate}.  
 \end{proof}
 
 By the Strichartz estimate, Lemma \ref{lestimate} and Lemma \ref{nestimate},  we have  
 \begin{lemma}\label{L:small solution}
 	Let  $v$ be a solution of (\ref{Linear Eq}). Assume for some  $c_0>0$ 
 	\begin{equation}
 		\|v(t)\|_{\dot H^{1}}\lesssim e^{-c_0t},\qquad \forall t\ge0.\notag
 	\end{equation}
 	Then for any Strichartz couple  $(p,q)$,  
 	\begin{equation}
 		\|v\|_{Z(t,+\infty )}+\|\nabla v\|_{L^p(t,+\infty ;L^{q,2})}\lesssim  e^{-c_0t},\qquad \forall t\ge0.\notag
 	\end{equation}
 \end{lemma}
 \begin{proof}
 	For small  $\tau_0$, by the Strichartz estimate and Lemmas \ref{lestimate}--\ref{nestimate}, we have on  $I_{\tau}=[\tau,\tau+\tau_0]$ 
 	\begin{eqnarray}
 			\|v\|_{Z(I_{\tau} )}+\|\nabla v\|_{L^p(I_{\tau} ;L^{q,2})}&\lesssim &  e^{-c_0 \tau}+ \|\nabla \mathcal{V}(v)\| _{N(I_{\tau})}+ \|\nabla R(v)\|_{N(I_{\tau})}\notag\\
 			&\lesssim &  e^{-c_0 \tau}+\tau_0^{\frac{\alpha }{\gamma }} \|v\|_{Z(I_{\tau})}+ \|v\|_{Z(I_{\tau})}^2+ \|v\|_{Z(I_{\tau})}^{\alpha +1}.\notag    
 	\end{eqnarray} 
 	By choosing sufficiently small  $\tau_0$, the continuous argument gives that 
 	\begin{equation}
 	\|v\|_{Z(I_{\tau} )}+\|\nabla v\|_{L^p(I_{\tau} ;L^{q,2})}\lesssim e^{-c_0 \tau}.\label{3271}
 	\end{equation} 
Summing up (\ref{3271}) at time  $\tau=t,\tau=t+\tau_0,\tau=t+2\tau_0,\cdots ,$ and using the triangle inequality, we get desired estimate in Lemma \ref{L:small solution}.   
 \end{proof}
 
 \subsection{Spectral properties of the linearized operator.}
 Since  $W$ is   a critical point of the energy  $E$, we have the following development of the energy near  $W$: For any  $g\in \dot H^1(\mathbb{R}^d)$ with  $\left\|g\right\|_{\dot H^1(\mathbb{R}^d)}$ small,
 \begin{equation}
 	E(W+g)=E(W)+Q(g)+O(\left\|g\right\|_{\dot H^1}^3), \label{2211}
 \end{equation}         
 where  $Q$    is  the quadratic form   on  $\dot H^{1}(\mathbb{R}^d)$  defined by 
\begin{equation}
	Q(g):= \frac{1}{2}\int |\nabla g|^2-\frac{1}{2}\int |x|^{-b}W^\alpha ((\alpha +1)(\text{Re} g)^2+(\text{Im} g)^2)=\frac{1}{2}\text{Im}\int (\mathcal{L}g)\overline{g},\label{E:Q}
\end{equation}
with the operator  $\mathcal{L}$ defined by (\ref{Linear Eq}).

Let us specify the  important coercivity properties of  $Q$, which will be used in the later Sections.   Consider the three
orthogonal directions  $W,iW$ and 
\begin{equation}
	W_1:=\left.-\frac{d}{d \lambda}\left(\frac{1}{\lambda^{\frac{d-2}{2}}}W(\frac{x}{\lambda})\right)\right|_{\lambda=1}=\frac{d-2}{2}W+x\cdot \nabla W \label{E:W_1}
\end{equation}
  in the real Hilbert space  $\dot H^{1}=\dot H^{1}(\mathbb{R}^d,\mathbb{C})$. 
Let  
\begin{equation}
	H:=\text{span }\left\{W,\ iW,\ W_1\right\}\notag
\end{equation} and  $H^{\bot}$ its orthogonal subspace in  $\dot H^{1}$ for the usual product. 
Inspired by \cite{YZZ}, we can show that 
 the three directions,  $W_1,\ iW,\ W$ are the only nonpositive directions of  the    quadratic form  $Q$.   
\begin{proposition}
	\label{H}
	There is a constant  $c>0$ such that for all function  $f\in H^{\bot}$, we have  
	\begin{equation}
		Q(f)\ge c\|{f}\|_{\dot H^{1}}^2.\notag
	\end{equation}
\end{proposition}
\begin{proof}
	Let  $f_1:=\text{Re}f,\ f_2=\text{Im}f$. We have 
	\begin{equation}
		Q(f)=\frac{1}{2}(L_+f_1,f_1)_{L^2}+\frac{1}{2}(L_-f_2,f_2)_{L^2},\notag
	\end{equation} 
	where 
\begin{equation}
	L_+:=-\Delta -(\alpha +1)|x|^{-b}W^\alpha \quad\text{and }\quad L_-:=-\Delta -|x|^{-b}W^\alpha .\notag
\end{equation}

We first consider the operator  $L_+$ and show that there is only one negative direction in the sense that for any real scalar valued function  $v\in \dot H^1(\mathbb{R} ^d)$ and 
\begin{equation}
	(-\Delta v,W)_{L^2}=0,\label{1221}
\end{equation}  
we have 
\begin{equation}
	(L_+v,v)_{L^2}\ge0. \label{1222}
\end{equation}
Indeed, we will see that this is an implication of the fact that   $W$ is the constrained maximizer.  Define  the trajectory 
\begin{equation}
	 l(s):=\frac{ \|W\|_{\dot H^1}^2 }{( \|W\|_{\dot H^1}^2+s^2 \|v\|_{\dot H^1}^2  )^{1/2}}(W+sv)\quad\text{such that }\quad  \|l(s)\|_{\dot H^1}^2= \|W\|_{\dot H^1}^2.\notag  
\end{equation} 
It can be computed that 
\begin{equation}
	l(0)=W,\ l_s(0)=v,\ l_{ss}(0)=-\frac{ \|v\|_{\dot H^1}^2 }{ \|W\|_{\dot H^1}^2 }W.\notag
\end{equation}
From here and noting that  by  Proposition \ref{P:GN},  $W$ is the constrained maximizer:
\begin{equation}
	\int |x|^{-b}W^{\alpha +2}=\sup _{ \|g\|_{\dot H^1}= \|W\|_{\dot H^1}  }\int |x|^{-b}|g|^{\alpha +2}dx,\notag
\end{equation} 
we have 
\begin{eqnarray}
	0&\ge&\left.  \frac{d^2}{ds^2}\int _{\mathbb{R} ^d}|x|^{-b}|l(s)|^{\alpha +2}dx\right |_{s=0}\notag\\
	&=& (\alpha +2)(\alpha +1)\int _{\mathbb{R} ^d}|x|^{-b}l(0)^\alpha l_s^2(0)+(\alpha +2)\int _{\mathbb{R} ^d}|x|^{-b}l(0)^{\alpha +1}l_{ss}(0)dx\notag\\
	&=& (\alpha +2)(\alpha +1)\int _{\mathbb{R} ^d}|x|^{-b} W^\alpha v^2dx-(\alpha +2) \frac{ \|v\|_{\dot H^1}^2 }{ \|W\|_{\dot H^1}^2 }\int _{\mathbb{R} ^d}|x|^{-b}W^{\alpha +2}dx\notag\\
	&=&   -(\alpha +2)\int _{\mathbb{R} ^d} (-\Delta v-(\alpha +1)|x|^{-b}W^\alpha v)\cdot vdx =-(\alpha +2)(L_+v,v)_{L^2}\notag
\end{eqnarray}
and  (\ref{1222}) is proved.  

Next we investigate the null direction of $L_+$ and it is more convenient to work in $L^2$ setting instead of $\dot H^1$ setting. The operator $L_+$ having only one negative direction in $\dot H^1(\mathbb{R}^d)$ implies $ (-\Delta )^{-1/2}L_+(-\Delta )^{-1/2}$ has only one negative direction in $L^2(\mathbb{R}^d)$. Easily we can write 
\begin{align*}
(-\Delta )^{-1/2}L_+(-\Delta )^{-1/2}=I-(\alpha +1)(-\Delta )^{-1/2}|x|^{-b}W^\alpha (-\Delta )^{-1/2} :=I-K. 
\end{align*}
We have the following result for $K$:
\begin{claim}\label{comp}
	$K:\ L^2(\mathbb{R} ^d)\to L^2(\mathbb{R} ^d)$ is a compact operator. 
\end{claim}

Postponing the proof for the moment, using this claim we know that $I-K$ has at most finitely many  eigenvalues in $(-\infty, \frac 12]$ which can be ordered as
\begin{equation}
	\lambda_1\le \lambda_2\le\cdots\le\lambda_N\notag
\end{equation}
counting multiplicity. 

From the previous discussion and recall that
\begin{equation}
	(I-K)(-\Delta )^{-1/2}W_1=0,\notag
\end{equation}
we know 
\begin{equation}
	\lambda_1<0 \quad\text{and}\quad \lambda_2=0. \notag
\end{equation}
Our goal now is to show $\lambda_3>0$. Note as $I-K$ is symmetric we can choose eigenfunctions as the orthonormal basis of $L^2(\mathbb{R} ^d)$ and evaluate the $L^2$ bilinear form $ ( (I-K)u, u)_{L^2}$. Switching back to $\dot H^1$ setting, we immediately get the desired estimate for $L_+$:
\begin{equation}
(L_+u,u)_{L^2}\ge \lambda_3 \|u\|_{\dot H^1}^2 ,\qquad \forall u \bot_{\dot H^1} W, W_1.\notag
\end{equation}

Therefore it remains to show $\lambda_3>0$ or the kernel of $I-K$ is only one-dimensional in $L^2(\mathbb{R} ^d)$. This is equivalent to showing the kernel of $L_+$ is one dimensional in $\dot H^1(\mathbb{R} ^d)$.   The proof    relies on  the spherical harmonics expansion and careful study on the spatial asymptotics of the resulted ODEs.  

Consider the equation 
\begin{equation}
	L_+u=0,\notag
\end{equation}
we write $u$ in the spherical harmonic expansion:
\begin{equation}
	u(r,\theta)=\sum_{j=0}^\infty f_j(r)Y_j(\theta). \notag
\end{equation}
Here, $Y_j(\theta)$ is the $jth$ spherical harmonics and $\{ Y_j(\theta)\}_{j=0}^\infty$ form an orthonormal basis of $L^2(\mathbb S^{d-1})$. Recall that 
\begin{gather*}
	-\Delta_{\mathbb S^{d-1}} Y_j(\theta) =\mu_j Y_j(\theta), \ j=0, 1, 2,\cdots \\
	0=\mu_0<\mu_1\le\mu_2\le \cdots\ \to \infty, \ Y_0=1, \mu_1=d-1. 
\end{gather*}
In spherical harmonic expansion, we have 
\begin{equation}
	L_+u=-\sum_{j=0}^{\infty }\left((\partial_{rr}+\frac{d-1}{r}\partial_{r}-\frac{\mu_j}{r^2}+\frac{\alpha +1}{r^b}W^\alpha )f_j(r)\right)Y_j(\theta ).\notag
\end{equation}
Therefore we can discuss the contribution to the kernel from each spherical harmonic starting from $j=0$. 

\textbf{Case 1.} $j=0$.

As  $Y_0=1$ , the kernel function in this mode must be a spherically symmetric function  $u(r)$ satisfying  $L_+u=0$, which in the radial coordinate, takes the form
\begin{equation}
	u_{rr}+\frac{d-1}{r}u_r+\frac{\alpha +1}{r^b}W^\alpha u=0.\notag
\end{equation} 
Supposing  $u$  is a solution independent of the known radial solution  $W_1$ , from Abel's theorem, we have 
\begin{equation}
	u_rW_1-(W_1)_ru =\frac{C}{r^{d-1}}.\label{1223}
\end{equation}
In the small neighborhood of  $r=0,W_1\neq 0$,   we can divide both sides of  (\ref{1223}) by  $W_1^2$  and obtain
\begin{equation}
	\left(\frac{u}{W_1}\right)_r=\frac{C}{r^{d-1}W_1^2},\qquad 0<r<\varepsilon .\notag
\end{equation}
Recalling  that by (\ref{E:W_1}) $W_1(r)=O(1)$ as  $r\rightarrow 0^+$  and integrating the above equation from  $r$  to  $\varepsilon $, we have  
\begin{equation}
	\partial_{r}u(r)=O(\frac{1}{r^{d-1}})\quad\text{as}\quad r\rightarrow0^+\notag
\end{equation}
and  $u$  is certainly not an  $\dot H^1$ function. Therefore,  $W_1$  is the unique radial kernel. 

\textbf{Case 2.} $\{j\in \mathbb{N} ,\mu_j=\mu_1= d-1\}$. 

In this case, we claim   that for any  $G(r)\in \dot H^1_{\text{rad}}(\mathbb{R}^d)$,   
\begin{equation}
	\text{if} \qquad L_+(G(r)Y_j(\theta))=0 \quad\text{then}\quad G(r)=0.\label{2201}
\end{equation}

Assume by contradiction that there exists  $0\neq G(r)\in \dot H^1_{\text{rad}}(\mathbb{R}^d)$  such that   $L_+(G(r)Y_j(\theta))=0$.     
 Writing the Laplacian operator in a spherical coordinate, we have
\begin{equation}
	0=L_+(G(r)Y_j(\theta ))=(-\partial_{rr}-\frac{d-1}{r}\partial_{r}+\frac{\mu_1}{r^2}-\frac{\alpha +1}{r^b}W^\alpha )G(r)\cdot Y_j(\theta ),\notag
\end{equation}
which implies 
\begin{equation}
	G(r)\in \text{Ker }(-\Delta +\mu_1 |x|^{-2}-(\alpha +1)|x|^{-b}W^\alpha ).\label{1224}
\end{equation}
Our first goal toward getting a contradiction is to show the positivity of  $G$. To this end, we take any  $v\in \dot H^1(\mathbb{R} ^d)$ in the spherical harmonic expansion
\begin{equation}
	v:=\sum_{k=1}^{\infty }v_k(r)Y_k(\theta ).\notag
\end{equation}
Since  $v_k(r)Y_1(\theta )\bot_{\dot H^1}W=WY_0(\theta ) $, it follows from (\ref{1221}) and (\ref{1222}) that 
 \begin{eqnarray}
 	&&( (-\Delta +\mu_1 |x|^{-2}-(\alpha +1)|x|^{-b}W^\alpha ) v_k(r),v_k(r))_{L^2}\notag\\
 	&=& ((-\Delta +\mu_1 |x|^{-2}-(\alpha +1)|x|^{-b}W^\alpha ) v_k(r)\cdot Y_1(\theta ),v_k(r)\cdot Y_1(\theta ))_{L^2}\notag\\
 	&=& (L_+(v_k(r)Y_1(\theta )),v_k(r)Y_1(\theta ))_{L^2}\ge0.\label{220w1}
 \end{eqnarray}
Using (\ref{220w1}), we can evaluate 
\begin{eqnarray}
		&&((-\Delta +\mu_1 |x|^{-2}-(\alpha +1)|x|^{-b}W^\alpha)v,v)_{L^2}\notag\\
		&=&\sum_{k=1}^{\infty }( (-\Delta +\mu_1 |x|^{-2}-(\alpha +1)|x|^{-b}W^\alpha )v_k(r),v_k(r))_{L^2}\notag\\
		&&+\sum_{k=1}^{\infty } \mu _k\int _{\mathbb{R} ^d} \frac{|v_k(x)|^2}{|x|^2}dx\ge0.\notag
\end{eqnarray}
This together with (\ref{1224}) implies that  $0$  is the first eigenvalue. Hence,   $G(r)>0$.  

We now turn to looking at the equation of   $G$ and  $-W'$,  
\begin{eqnarray}
	&-G''-\frac{d-1}{r}G'+\frac{d-1}{r^2}G-\frac{\alpha +1}{r^b}W^\alpha G=0,\label{E:G}\\
	&-W'''-\frac{d-1}{r}W''+\frac{d-1}{r^2}W'+\frac{b}{r^{b+1}}W^{\alpha +1}-\frac{\alpha +1}{r^b}W^\alpha W'=0.\label{E:W}
\end{eqnarray}
Computing  $[(\ref{E:G})\cdot r^{d-1}W'-(\ref{E:W})\cdot r^{d-1}G]$, we obtain 
\begin{equation}
			r^{d-1} W'''G+(d-1)r^{d-2}W''G-r^{d-1}W'G''-(d-1)r^{d-2}W'G'-br^{d-b-2}W^{\alpha +1}G=0,\notag
\end{equation}
which can be further written into
\begin{equation}
	\frac{d}{dr}[r^{d-1}(W''G-W'G')]-br^{d-b-2}W^{\alpha +1}G=0.\notag
\end{equation}
As  $G>0$, we obtain 
\begin{equation}
	\int _0^\infty \frac{d}{dr}[r^{d-1}(W''G-W'G')]dr=b\int_0^\infty  r^{d-b-2}W^{\alpha +1}Gdr>0.\label{1226}
\end{equation} 
Recalling the asymptotics of  $W$ and  $G$ from   Appendix \ref{App:1}
\begin{equation}
	\begin{cases}
		\text{As }r\rightarrow0^+,G(r)=O(r),\ G'(r)=O(1),\ -W'(r)=O(r^{1-b}),\\
		\text{As }r\rightarrow\infty ,G(r)=O(r^{-(d-1)}),\ G'(r)=O(r^{-d}),\ -W'(r)=O(r^{-(d-1)}),
	\end{cases}\notag
\end{equation}  
we have 
\begin{equation}
	\lim_{r\rightarrow0^+}   r^{d-1}(W''G-W'G')=\lim_{r\rightarrow\infty }r^{d-1}(W''G-W'G')=0,\notag
\end{equation}
which contradicts (\ref{1226}).  Claim (\ref{2201}) is proved. 

\textbf{Case 3.} $\{j\in \mathbb{N},\mu_j>\mu_1\}$.

In this case, we take any function in the form  $G(r)Y_j(\theta),G\neq0$, and compute 
\begin{equation}
	L_+\left(G(r)Y_j(\theta)\right)=(-\Delta +\mu_1 |x|^{-2}-(\alpha +1)|x|^{-b}W^\alpha )G(r)Y_j(\theta)+\frac{\mu_j-\mu_1}{r^2}G(r)Y_j(\theta). \notag
\end{equation}  
Using (\ref{220w1}) we immediately get 
\begin{eqnarray}
		&&(L_+\left(G(r)Y_j(\theta)),G(r)Y_j(\theta)\right))_{L^2} \notag\\
	&=&((-\Delta +\mu_1 |x|^{-2}-(\alpha +1)|x|^{-b}W^\alpha )G(r),G(r))_{L^2}+(\mu_j-\mu_1)\int _{\mathbb{R}^d}\frac{|G(x)|^2}{|x|^2}dx>0.  \notag
\end{eqnarray}
This shows there is no kernel function of  $L_+$ associated to  $j$th spherical  harmonics for those  $j$  such that  $\mu_j>\mu_1$.

The positivity of  $\lambda_3$  is finally proved, and we end the discussion on the operator   $L_+$.

On the other hand,  we can get the results for  $L_-$ quickly.    By H\"older inequality and    Proposition  \ref{P:GN},  we have, for any real-valued  $v\in \dot H^1$,  
\begin{eqnarray}
	\int |\nabla v|^2dx-\int |x|^{-b}W^\alpha v^2dx&\ge& \int |\nabla v|^2dx-\left(\int |x|^{-b}W^{\alpha +2}dx\right)^{\frac{\alpha }{\alpha +2}}\left(\int |x|^{-b}|v|^{\alpha +2}dx\right)^{\frac{2}{\alpha +2}}\notag\\
	&\ge& \left(1-\frac{\int |x|^{-b}W^{\alpha +2}dx}{ \|W\|_{\dot H^1}^2 }\right)\int |\nabla v|^2dx=0,\notag
\end{eqnarray}
with equality if and only if  $v\in \text{Span }\left\{ W\right\}$. This shows that  $\int |\nabla f_2|^2-\int |x|^{-b}W^\alpha |f_2|^2>0$ for  $f_2\neq 0,\ f_2\bot_{\dot H^1} W$. Note that the quadratic form  $\int |\nabla \cdot|^2-\int |x|^{-b}W^\alpha |\cdot|^2$ is a compact perturbation of  $\int |\nabla \cdot |^2$. Therefore, there exists  $c_2>0$ such that for any  $f_2\in \dot H^1, f_2\bot_{\dot H^1} W$,
\begin{equation}
	(L_-f_2,f_2)\ge c_2 \|f_2\|_{\dot H^1}^2.\notag 
\end{equation}       
Combining the two parts together,   we proved the estimate for  $Q(v)$.   

Finally, we complete the proof by verifying Claim \ref{comp}.  Indeed, note that as 
 $	(-\Delta )^{-1/2}:L^2(\mathbb{R} ^d)\rightarrow \dot H^1(\mathbb{R} ^d)$ 
is an isometric operator and the embedding  $L^{\frac{2d}{d+2}}(\mathbb{R} ^d)\hookrightarrow \dot H^{-1}(\mathbb{R} ^d)$ is continuous, it suffices to show that  $|x|^{-b}W^\alpha :\dot H^1(\mathbb{R} ^d)\rightarrow L^{\frac{2d}{d+2}}(\mathbb{R} ^d)$ is compact.

We first claim that 
\begin{equation}
	\left\||x|^{-b}W^\alpha \right\|_{WL^{\frac{d}{2},2}(\mathbb{R}^d)}<+ \infty.\label{221w1}
\end{equation}
In fact, using  Lemma \ref{L:leibnitz} and  $W(x)=O({ \langle x \rangle }^{-(d-2)})$, we have 
\begin{eqnarray}
		&&\left\||x|^{-b}W^\alpha \right\|_{WL^{\frac{d}{2},2}}  \notag\\
		&\lesssim& \left\||x|^{-b}\right\|_{L^{\frac{d}{b}, \infty}}(\left\|W^{\alpha-1} \nabla W\right\|_{L^{\frac{d}{2-b},2}}+\left\|W^\alpha \right\|_{L^{\frac{d}{2-b},2}})+\left\||x|^{-b-1}\right\|_{L^{\frac{d}{b+1}, \infty}}\left\|W^\alpha \right\|_{L^{\frac{d}{1-b},2}}<+ \infty, \notag
\end{eqnarray}  
which yields (\ref{221w1}).

Fix $\varepsilon >0$ sufficiently small.
Then for  $v\in \dot H^1$, we  have, by applying (\ref{221w1})
\begin{eqnarray}
	 &&\||\nabla |^\varepsilon (|x|^{-b}W^\alpha v)\|_{L^{\frac{2d}{d+2}}} \notag\\
	 &\lesssim &   \||\nabla |^\varepsilon (|x|^{-b}W^\alpha )\|_{L^{\frac{d}{2},2}}  \|v\|_{L^{\frac{2d}{d-2},2}}+ \| |x|^{-b}W^\alpha \|_{L^{\frac{d}{2-\varepsilon },2}}  \||\nabla |^\varepsilon v\|_{L^{\frac{2d}{d-2+2\varepsilon },2}}\notag\\
	 &\lesssim &    \||\nabla |^\varepsilon (|x|^{-b}W^\alpha )\|_{L^{\frac{d}{2},2}}  \|\nabla v\|_{L^2}\lesssim    \|v\|_{\dot H^1},\notag 
\end{eqnarray}
 and 
 \begin{equation}
 	 \|\chi_R|x|^{-b}W^\alpha v\|_{L^{\frac{2d}{d+2}}}\lesssim R^{-b} \|W^\alpha \|_{{L^{\frac{d}{2},2}}} \|v\|_{L^{\frac{2d}{d-2},2}}\lesssim  R^{-b} \|v\|_{\dot H^1}.\notag    
 \end{equation}
 The desired compactness of  $|x|^{-b}W^\alpha $ is proved, hence Claim \ref{comp}.  
 Proposition  \ref{H} is finally proved. 
\end{proof}

Following the arguments in \cite{Campos-Murphy,2008GFA,DuyRouden:NLS:ThresholdSolution,2009JFA}, we have the following spectral properties of  $\mathcal{L}$ that defined in (\ref{Linear Eq}).  For the sake of completeness, we will  give the proof in Appendix \ref{App:3}. 
\begin{lemma}\label{l:eigen}
	Let  $\sigma(\mathcal{L})$ denote the spectrum of the operator  $\mathcal{L}$, defined on  $L^2(\mathbb{R} ^d)$ with domain  $H^2(\mathbb{R} ^d)$  and let  $\sigma_{\text{ess}}(\mathcal{L})$ its essential spectrum. Then we have \\
	(a)    	The operator  $\mathcal{L} $ admits two eigenfunctions    $\mathcal{Y}_+,\mathcal{Y}_-\in H^2(\mathbb{R}^d)$ with	real eigenvalues  $\pm e_0,\ e_0>0$, i.e.  $\mathcal{L}\mathcal{Y}_{\pm}=\pm e_0\mathcal{Y}_\pm$,  $\mathcal{Y}_+=\overline{\mathcal{Y}}_-$.    \\
	(b)	If  $\varphi\in C_c^\infty (\mathbb{R}^d\setminus\left\{0\right\} )$, then 
	\begin{equation}
		\|\varphi (\frac{x}{R})\mathcal{Y}_\pm\|_{H^k}\lesssim _{\phi,k,l}\frac{1}{R^l},\qquad \forall R\ge1.\label{decay1}
	\end{equation}
	Moreover,  $\mathcal{Y}_\pm \in  L^\infty (\mathbb{R} ^d)\cap W^3L^{\frac{2d}{d+2},2} (\mathbb{R} ^d)$.\\
	(c) If  $\lambda \in \mathbb{R}\setminus \sigma(\mathcal{L} )$ and  $F\in L^2(\mathbb{R}^d)$ is such that 
	\begin{equation}
		\|\psi (\frac{x}{R})F\|_{H^k}\lesssim _{\psi,k,l}\frac{1}{R^l} \qquad \forall R\ge1, \label{12231}
	\end{equation}
	for any  $\psi\in C_c^\infty (\mathbb{R}^d\setminus\left\{0\right\} )$, then the solution  $f\in H^2(\mathbb{R}^d)$ to 
	\begin{equation}
		\mathcal{L} f-\lambda f=F\notag
	\end{equation}  
	also satisfies  
	\begin{equation}
		\|\phi(\frac{x}{R})f\|_{H^k}\lesssim _{\phi,k,l}\frac{1}{R^l}\label{decay2}
	\end{equation}
	for any  $\phi\in C_c^\infty (\mathbb{R}^d\setminus\left\{0\right\} )$. Moreover, if for some  $\varepsilon >0$ sufficiently small,  $F\in H^{\frac{1}{2}+\varepsilon} (\mathbb{R} ^d)\cap WL^{\frac{2d}{d+2},2} (\mathbb{R} ^d)$, then  $f\in L^\infty (\mathbb{R} ^d) \cap  W^3L^{\frac{2d}{d+2},2} (\mathbb{R} ^d)$. \\
	(d) $\sigma_{\text{ess}}(\mathcal{L})=\{i\xi: \xi\in \mathbb{R} \}$,  $\sigma(\mathcal{L})\cap \mathbb{R} =\{-e_0,0,e_0\}$.  
\end{lemma}

At the end of this subsection, we utilize the spectral properties of    $\mathcal{L}$ from Lemma  \ref{l:eigen} to give a subspace  $G^{\bot}$ of  $\dot H^{1}$,  in which  $Q$ is positive definite.  

Consider the symmetric bilinear form  $B$ on  $\dot H^{1}$ such that  $Q(f)=B(f,f)$:
\begin{align}
		B(f,g):= \frac{1}{2}\text{Im} \int (\mathcal{L} f)\overline{g} 
		=\frac{1}{2}&\int \nabla f_1\nabla g_1-\frac{\alpha +1}{2}\int |x|^{-b}W^\alpha  f_1g_1 \notag\\
	&+\frac{1}{2}\int \nabla f_2 \nabla g_2-\frac{1}{2}\int |x|^{-b} W^\alpha  f_2g_2.\label{B}
\end{align}
As a consequence of the definition of  $B$, we have 
\begin{equation}
	B(f,g)=B(g,f),\ B(iW,f)=B(W_1,f)=0,\ \forall f,g\in \dot H^1, \label{12241}
\end{equation} 
\begin{equation}
	B(\mathcal{L}f,g)=-B(f,\mathcal{L}g),\ \forall f,g\in \dot H^1,\ \mathcal{L}f,\mathcal{L}g\in \dot H^1\label{12242}
\end{equation}
\begin{equation}
	Q(\mathcal{Y}_+)=Q(\mathcal{Y}_-)=0.\label{12243}
\end{equation}
Based on Proposition \ref{H} and (\ref{12241})--(\ref{12243}), we have the following coercivity of \(Q\) on \(G^{\bot}\):
\begin{lemma}\label{L:G}
	Let  $G^\bot=\left\{v\in \dot H^{1}:(iW,v)_{\dot H^{1}}=(W_1,v)_{\dot H^{1}}=B(\mathcal{Y}_+,v)=B(\mathcal{Y}_-,v)=0 \right\} $. There  exists  $c>0$ such that 
	\begin{equation}
		Q(f)\ge c\|f\|_{\dot H^{1}}^2,\qquad \forall f\in G^{\bot}.\notag
	\end{equation}
\end{lemma}
\begin{proof}
We first claim that  $B(\mathcal{Y}_+,\mathcal{Y}_-)\neq0$. In fact, if   $B(\mathcal{Y}_+,\mathcal{Y}_-)=0$, then   $Q$  would be identically  $0$  on   $\text{Span }\left\{iW,W_1,\mathcal{Y}_+,\mathcal{Y}_-\right\}$ which is of dimension  $4$. But  $Q$ is, by  Proposition  \ref{H},  positive definite on  $H^\bot$,  which is of codimension   $3$, yielding a contradiction. 

We next claim that   $Q(h)>0$ on  $G^\bot \setminus \left\{0 \right\}$.  Assume by contradiction that   there exists  $h\in G^\bot \setminus \left\{0 \right\}$ such that   $Q(h)\le 0$.  Then by (\ref{12241})--(\ref{12243})
\begin{equation}
	Q|_{\text{Span }\left\{iW,W_1,\mathcal{Y}_+,h\right\}} \le 0.\notag
\end{equation}
If for some  $\alpha , \beta, \gamma ,\delta \in \mathbb{R} $  
\begin{equation}
	\alpha iW+\beta W_1+\gamma \mathcal{Y}_++\delta h=0,\notag
\end{equation}
then   $\gamma  B(\mathcal{Y}_+,\mathcal{Y}_-)=0$, which implies  $\gamma =0$.  Therefore the vectors  $iW,W_1,\mathcal{Y}_+,h$ are independent, since  $iW,W_1$ and  $h$    
are orthogonal in the real Hilbert space $\dot H^1$.   The fact that  $Q$ is nonpositive on a subspace of dimension  $4$ contradicts  Proposition \ref{H}.   

Finally, we prove the coercivity by a compactness argument. Suppose, by contradiction that there exists   $\left\{f_n \right\}\in G^\bot$  such that 
	\begin{equation}
		\lim_{n\rightarrow \infty }Q(f_n)=0\quad\text{and}\quad   \|f_n\|_{\dot H^1}=1.\notag
	\end{equation}
	Up to a subsequence, we may assume  $f_n {\rightharpoonup}f^*$ weakly in  $\dot H^1$. This implies  $f^*\in G^\bot$ and  
	\begin{equation}
		Q(f^*)\le \liminf_{n\rightarrow \infty }Q(f_n)=0.\label{28w2}
	\end{equation}      
Using H\"older's inequality and the decay of  $|x|^{-b}$ at infinity, it is easy to see that  $\int |x|^{-b}W^\alpha |\cdot|^2$ is a compact operator. 
Therefore   
\begin{eqnarray}
	&&\frac{1}{2}\int _{\mathbb{R} ^d}|x|^{-b}W^\alpha \left((\alpha +1)(\text{Re}f^*)^2+(\text{Im}f^*)^2\right)\notag\\
	&=&\lim_{n\rightarrow\infty } \frac{1}{2}\int _{\mathbb{R} ^d}|x|^{-b}W^\alpha \left((\alpha +1)(\text{Re}f_n)^2+(\text{Im}f_n)^2\right)\notag\\
	&=& \lim_{n\rightarrow\infty }\frac{1}{2} \|\nabla f_n\|_{L^2}-Q(f_n)=\frac{1}{2},\notag 
\end{eqnarray}
	which implies that  $f^*\neq0$. However, this together with  (\ref{28w2}) contradicts the strict positivity of  $Q$ on  $G^\bot \setminus\left\{0 \right\}$.    
\end{proof}
\section{Existence of special threshold solutions  $W^{\pm}$ }\label{s:3}
In this section, we
show the existence of the solutions  $W_\pm$ of Theorem \ref{T1}.
Following the arguments in \cite{2008GFA}, we first  construct approximate solutions  $W_k^a$
of (\ref{nls}) by use of the spectral property of the linearized operator  $\mathcal{L}$. Then we prove the existence of special threshold solutions  $W^a$ and  $W^{\pm}$ by a fixed point argument around  approximate solutions.

\subsection{A family of approximate solutions converging to  $W$.}\label{s6.1}
\begin{lemma}\label{l:app}
	Let  $a\in \mathbb{R}$. There exist functions  $(\Phi_j^a)_{j\ge 1}$ in  $L^\infty (\mathbb{R} ^d)\cap H^2(\mathbb{R}^d)\cap W^3L^{\frac{2d}{d+2},2}(\mathbb{R}^d)$, satisfying (\ref{12231}), such that  $\Phi_1^a=a\mathcal{Y}_+$ and if  
	\begin{equation}
		W^a_k(t,x):=W(x)+\sum_{j=1}^{k} e^{-je_0t}\Phi_j^a(x),\notag
	\end{equation} 
	then as  $t\rightarrow +\infty $, 
	\begin{equation}
		i\partial_{t}W_k^a+\Delta W^a_k+|x|^{-b}|W^a_k|^\alpha W^a_k=O(e^{-(k+1)e_0t})\quad \text{in  } \quad   WL^{\frac{2d}{d+2},2}(\mathbb{R}^d).\label{88w1}
	\end{equation}
\end{lemma}
\begin{proof}
	For simplicity, we omit the superscript  $a$.  We will construct the functions  $\Phi_j=\Phi_j^a$ by induction on  $j$. Assume that  $\Phi_1,\cdots \Phi_k$ are known, and let 
	\begin{equation}
		v_k:=W_k-W=\sum_{j=1}^{k}e^{-je_0t}\Phi_j(x).\label{3272}
	\end{equation}
	By (\ref{Linear Eq}), assertion (\ref{88w1})  is equvalient to 
	\begin{equation}
		\varepsilon _k:=\partial_{t}v_k+\mathcal{L} (v_k)+R(v_k)=O(e^{-(k+1)e_0t})\quad \text{in  } \quad   WL^{\frac{2d}{d+2},2}(\mathbb{R}^d).\label{88w2}
	\end{equation}   
	
	\noindent\textbf{Step 1:}  k=1.  Let  $\Phi_1:=a\mathcal{Y}_+$   and  $v_1(t,x):=e^{-e_0t}\Phi_1(x)$. We have  $\partial_{t}v_1+\mathcal{L} (v_1)=0$ and thus 
	\begin{equation}
		\partial_{t}v_1+\mathcal{L} (v_1)+R(v_1)=R(v_1).\notag
	\end{equation}
	Since $v_1=ae^{-e_0t}\mathcal{Y}_+$ and $\mathcal{Y}_+\in H^2(\mathbb{R}^d)\cap W^3L^{\frac{2d}{d+2},2}(\mathbb{R}^d)$, 
it follows from Lemma \ref{nestimate} that  $R(v_1)=O({e^{-2e_0t}})\  \text{in }WL^{\frac{2d}{d+2},2}(\mathbb{R}^d)$.

	\noindent\textbf{Step 2:}  Induction. Let us assume that $\Phi_1, \ldots, \Phi_k$ are known and satisfy (\ref{88w2}) for some $k\geq 1$. To construct $\Phi_{k+1}$, we first claim that there exists $\Psi_k\in H^{\frac{1}{2}+\varepsilon} (\mathbb{R} ^d)\cap  WL^{\frac{2d}{d+2},2}(\mathbb{R}^d)$, satisfying (\ref{12231}),  such that for large $t$
	\begin{equation}
		\varepsilon _k(t,x)=e^{-(k+1)e_0 t}\Psi_k(x)+O\big(e^{-(k+2)e_0t}\big) \quad \text {in} \quad   WL^{\frac{2d}{d+2},2}(\mathbb{R}^d).\label{88w3}
	\end{equation}
	Indeed, substituting  $v_k=\sum_{j=1}^{k} e^{-je_0t}\Phi_j(x) $ into (\ref{88w2}), we obtain 
	\begin{equation}
		\varepsilon _k(t,x)=\sum_{j=1}^k e^{-je_0t}\big( {-j}e_0\Phi_j(x)+\mathcal{L}\Phi_j(x)\big) +R(v_k(t,x)).\notag
	\end{equation}
	Note that  $R(v_k)=-i|x|^{-b}W^{\alpha +1}J(W^{-1}v_k)$, where  $J$ defined in (\ref{212}) is real-analytic for  $\{|z|<1\}$  and satisfies  $J(0)=\partial_{z}J(0)=\partial_{\overline{z}}J(0)=0$.  For  $|z|\le 1/2$, we can expand 
	\begin{equation}
		J(z)=\sum_{j_1+j_2\ge2}a_{j_1j_2}z^{j_1}\overline{z}^{j_2},\label{213}
	\end{equation}    
	with normal convergence of the series and all its derivatives. 
	All the functions $\Phi_j\in L^\infty (\mathbb{R} ^d)$, so that for large  $t$, and all  $x$,  $|v_k(t,x)|\le \frac{1}{2}W(x)$.    Using (\ref{213}) to expand  $R(v_k)$, we found that  there exist  $F_j\in WL^{\frac{2d}{d+2},2}(\mathbb{R}^d)(1\le j\le k)$ and  $F_{k+1}\in H^{\frac{1}{2}+\varepsilon}  (\mathbb{R} ^d)\cap WL^{\frac{2d}{d+2},2}(\mathbb{R}^d)$ , satisfying (\ref{12231}) such that  
	\begin{equation}
		\varepsilon _k(t,x)=\sum_{j=1}^{k+1} e^{-je_0t} F_{j}(x)  + O\big(e^{-(k+2) e_0t}\big) \quad \text{ in } \quad   WL^{\frac{2d}{d+2},2}(\mathbb{R}^d).\notag
	\end{equation}
	By (\ref{88w2}) at rank  $k$,  $F_j=0$ for  $j\le k$ which shows (\ref{88w3}) with 
	$\Psi_k=F_{k+1}$.

By Lemma \ref{l:eigen}, $(k+1) e_0$ is not in the spectrum of $\mathcal{L}$.  Define 
	\begin{equation}
		\Phi_{k+1}:=-(\mathcal{L}-(k+1)e_0)^{-1} \Psi_k\notag
	\end{equation}
	which belongs to  $L^\infty (\mathbb{R} ^d)\cap H^2(\mathbb{R} ^d)\cap W^3L^{\frac{2d}{d+2},2}(\mathbb{R}^d)$ and satisfies (\ref{decay2}) by Lemma \ref{l:eigen}.  By definition,  $v_{k+1}=v_k+e^{-(k+1)e_0t}\Phi_{k+1}$. Furthermore 
	\begin{eqnarray}
		\varepsilon _{k+1}&:=&\partial_t v_{k+1}+\mathcal{L}v_{k+1}+R(v_{k+1})\notag\\
		&=& \partial_{t}v_k+\mathcal{L}v_k+R(v_k)+(\mathcal{L}-(k+1)e_0)\Phi_{k+1}e^{-(k+1)e_0t}+R(v_{k+1})\notag\\
 		&=&\varepsilon _k-e^{-(k+1)e_0t} \Psi_{k}+R(v_{k+1})-R(v_k).\notag
	\end{eqnarray}
	By (\ref{88w3}), $\varepsilon _k-e^{-(k+1)e_0t} \Psi_{k}=O\big(e^{-(k+2)e_0t}\big)$ in $WL^{\frac{2d}{d+2},2}(\mathbb{R}^d)$.   Writing as before, 
	\begin{equation}
		R(\cdot)=-i|x|^{-b}W^{\alpha +1}J(W^{-1}\cdot)\notag
	\end{equation}
	 and using the development (\ref{213}) of  $J$,   we get that 
\begin{equation}
	R(v_{k+1})-R(v_k)=O\big(e^{-(k+2)e_0t}\big) \quad\text{in}\quad  WL^{\frac{2d}{d+2},2}(\mathbb{R}^d),\notag
\end{equation}
which yields (\ref{88w2}) at rank  $k+1$.   The proof is complete.  
\end{proof}
\subsection{Construction of special   threshold solutions  $W^a$.}\label{s6.2}
In this subsection, we apply the fixed point argument to show the existence of special   threshold solutions  $W^a$.
\begin{proposition}\label{P:approximate}
	Let  $a\in \mathbb{R}$. There exists  $k_0>0$ such that for any  $k\ge k_0$, there exists  $t_k\ge 0$ and a solution  $W^a$ of (\ref{nls}) such that for  $t\ge t_k$,
	\begin{equation}
		\|W^a-W^a_k\|_{Z(t,+\infty )}\le e^{-(k+\frac{1}{2})e_0t}.\label{885}
	\end{equation}
	Furthermore,  $W^a$ is the unique solution of (\ref{nls}) satisfying (\ref{885}) for large  $t$;  $W^a$ is also independent of  $k$ and satisfies for large  $t$,
	\begin{equation}
		\|W^a(t)-W-ae^{-e_0t}\mathcal{Y}_+\|_{\dot H^{1}}\lesssim  e^{-\frac{3}{2}e_0t}.\label{88w4}
	\end{equation}
	Finally, $W^a\in L^2(\mathbb{R} ^d)$ if  $d=5$.   
\end{proposition}
\begin{proof}[Sketch of the proof.]
	The proof is exactly the same as \cite[ Proposition 6.3]{2008GFA}. 
	For the convenience of the readers, we briefly  sketch   the  proof of the existence of   $W^a$.   
	Let  $v_k:=W^a_k-W$ and  by (\ref{88w2}) it satisfies 
	\begin{equation}
		\varepsilon _k:=\partial_{t}v_k+\mathcal{L}(v_k)+R(v_k)=O(e^{-(k+1)e_0t}) \quad\text{in}\quad  WL^{\frac{2d}{d+2},2}(\mathbb{R}^d).\label{813x1}
	\end{equation}
	Let  $v^a:=W^a-W$ and  by (\ref{Linear Eq}),  $W^a$ is a solution of (\ref{nls}) if and only if   $v^a$ satisfies the equation  
	\begin{equation}
		\partial_{t}v^a+\mathcal{L}(v^a)+R(v^a)=0.\label{813x2}
	\end{equation}
	Let  $h:=W^a-W^a_k$, then  $h=v^a-v_k$.  Combining (\ref{813x1}) and  (\ref{813x2}), we deduce that 
	\begin{equation}
		i\partial_t h+\Delta h=-\mathcal{V}(h)-i R(v_k+h)+i R(v_k)+i\varepsilon _k,\notag
	\end{equation}
	where the linear operator  $\mathcal{V}$ is defined in (\ref{V}).   
	Thus the existence of a solution  $W^a$  of (\ref{nls}) satisfying (\ref{885}) for  $t\ge t_k$ may be written as the following fixed-point problem
	\begin{equation}
		\forall t\ge t_k,\quad h(t) =\mathcal{M}_k(h)(t) \quad\text{and}\quad \|h\|_{Z(t,+\infty )}\le e^{-(k+\frac{1}{2})e_0t}\quad\text{where}\notag
	\end{equation}
	\begin{equation}
		\mathcal{M}_k(h)(t):=-\int _t^{+\infty } e^{i(t-s)\Delta }[i\mathcal{V}(h(s))-R(v_k(s)+h(s))+R(v_k(s))-\varepsilon _k(s)]ds.\notag
	\end{equation}
	Let us fix  $k$ and  $t_k$. Consider the Banach space 
	\begin{equation}
		B^k_Z:=\left\{ h\in Z(t_k,+\infty );\sup _{t\ge t_k}e^{(k+\frac{1}{2})e_0t}\| h\|_{Z(t,+\infty )}\le1 \right\}.\notag 
	\end{equation}
	By using the Strichartz estimate, (\ref{813x1}) and   Lemma \ref{nestimate}, we can show that if  $t_k$ and  $k$ are large enough, the mapping  $\mathcal{M}_k$  is a contraction on  $B^k_Z$.      This proves the existence of a solution  $W^a$  of (\ref{nls}) satisfying (\ref{885}) for  $t\ge t_k$.
	
	Finally, we show that  $W^a\in L^2(\mathbb{R} ^d)$ if  $d=5$.    Define a positive radial function  $\psi$ on  $\mathbb{R} ^d$ such that 
	 $\psi=1$ if  $|x|\le1$ and  $\psi=0$ if  $|x|\ge2$. For  $R>0$ and large  $t$, define 
	 \begin{equation}
	 	F_R(t):=\int _{\mathbb{R} ^d} |U^a(t,x)|^2\psi(\frac{x}{R})dx.\notag
	 \end{equation}        
	 Then we have 
	 \begin{eqnarray}
	 	F_R'&=& \frac{2}{R}\text{Im}\int W^a\nabla \overline{W}^a\cdot (\nabla \psi)(\frac{x}{R})dx=\frac{2}{R}\text{Im}\int W\nabla (\overline{W}^{a}-W)\cdot (\nabla \psi) (\frac{x}{R})dx\notag\\
	 	&&+\frac{2}{R}\text{Im}\int (W^a-W)\nabla W\cdot (\nabla \psi)(\frac{x}{R})dx+\frac{2}{R}\text{Im}\int (W^{a}-W)\nabla (\overline{W}^a-W)\cdot (\nabla \psi)(\frac{x}{R})dx.\notag
	 \end{eqnarray}
	 Applying (\ref{88w4}) and Hardy's inequality, we obtain 
	 \begin{equation}
	 	|F_R'(t)|\lesssim   \|U^a(t)-W\|_{\dot H^1}( \|U^a(t)\|_{\dot H^1}+ \|W\|_{\dot H^1})  \lesssim e^{-e_0t}.\notag 
	 \end{equation} 
	 Integrating the above inequality from sufficiently large  $t$ to  $+\infty $, we get 
	 \begin{equation}
	 	\left|F_R(t)-\int _{\mathbb{R} ^d}|W(x)|^2\psi (\frac{x}{R})dx\right|\lesssim e^{-e_0t}.\notag
	 \end{equation}  
	 Letting  $R\rightarrow+\infty $, we get  $ \|W^a(t)\|_{L^2}= \|W\|_{L^2}  $  and  $W^a(t)\in L^2(\mathbb{R} ^d)$ when  $d=5$, which completes the proof by applying mass conservation law.     
\end{proof}

\subsection{Construction of  $W^{\pm}$. }\label{s6.3}
\begin{proof}
	[Proof of Theorem \ref{T1}.] 	Let $\mathcal{Y}_1:=\text{Re} \mathcal{Y}_+=\text{Re}  \mathcal{Y}_-$. We first claim that 
	$(W,\mathcal{Y}_1)_{\dot H^{1}}\neq 0.$
	In fact, if $(W,\mathcal{Y}_1)_{\dot H^{1}}=0$, then by the equation (\ref{Eq:W}) and the definition of  $B(\cdot,\cdot)$ in (\ref{B}), we  have 
\begin{equation}
	B(W,\mathcal{Y}_\pm )=\frac{1}{2}\int \nabla W\nabla \mathcal{Y}_1-\frac{\alpha +1}{2}\int |x|^{-b}W^{\alpha +1}\mathcal{Y}_1=-\frac{\alpha }{2}\int \nabla W\nabla \mathcal{Y}_1=0,\notag
\end{equation}
	so that $W\in G^{\bot}$ and  thus $Q(W)>0$ by Lemma \ref{L:G}.  
	However, by Pohozhaev's identity (\ref{identity:pohozhaev}):
	\begin{equation}
		Q(W)=\frac{1}{2}\int |\nabla W|^2-\frac{\alpha +1}{2}\int |x|^{-b}W^{\alpha +2}=-\frac{\alpha }{2}\int |x|^{-b}W^{\alpha +2}<0.\label{QW}
	\end{equation}

	 Replacing  $\mathcal{Y}_\pm$ by   $-\mathcal{Y}_\pm$ if necessary, we may assume
	\begin{equation}
		(W,\mathcal{Y}_1)_{\dot H^{1}}>0.\label{88w5}
	\end{equation}
	Let  
	\begin{equation}
		W^\pm:=W^{\pm1},\notag
	\end{equation}
	which yields two solutions of (\ref{nls}) for large  $t>0$. Then all the conditions
	of Theorem \ref{T1} are satisfied.  Indeed,  the limits 
	\begin{equation}
		 \|W^{\pm}(t)-W\|_{\dot H^1}\lesssim e^{-e_0t},\qquad t\ge0,\notag 
	\end{equation}
	are an immediate consequence of (\ref{88w4}), while 
	$E(W^\pm)=E(W)$ follows from the conservation of the energy and the fact that  $W^a$ tends to  $W$ in  $\dot H^{1}$.
	Furthermore, again by (\ref{88w4}) 
	\begin{equation}
		\|W^a\|_{\dot H^{1}}^2=\|W\|_{\dot H^{1}}^2+2ae^{-e_0t}(W,\mathcal{Y}_1)_{\dot H^{1}}+O(e^{-\frac{3}{2}e_0t}), \notag
	\end{equation} 
	which together with (\ref{88w5}) shows that for large  $t>0$, 
	\begin{equation}
		\|\nabla W^+(t)\|_{L^2}> \|\nabla W\|_{L^2} \quad\text{and}\quad  \|\nabla W^-(t)\|_{L^2}< \|\nabla W\|_{L^2} .\notag
	\end{equation}
	From Lemma \ref{l:coercive}, these inequalities remain valid for every  $t$  in the intervals of existence of   $W^{\pm}$.   
	Finally,  $W^-(t)$ scatters in the negative time direction follows from (\ref{1119w1}), and  $W^+(t)$ blows up in finite negative  time when  $d=5$  follows from Corollary \ref{c:210}.  
\end{proof}

\section{Modulation analysis.}\label{s:4}
In this section, we perform the modulation analysis for solutions in the small neighborhood of the  ground state. On energy surface of the ground state, the distance to this manifold is controlled by 
\begin{equation}
	 \mathbf{d} (u)=\left|  \int_{\mathbb{R} ^d} \left(|\nabla u(x)|^2-|\nabla W(x)|^2\right)dx \right|, \notag
\end{equation} 
as shown in the following result. The same result in the case of pure-power NLS can be found in \cite{R: Aubin,Lions85,R: Talenti}. \\
 \noindent \textbf{Notation.} 
If  $v$ is a function defined on  $\mathbb{R} ^d$, as a convention,  we write 
\begin{equation}
	v_{[\lambda_0]}	(x):= {\lambda_0^{-\frac{d-2}{2}}}v(\frac{x}{\lambda_0})\quad\text{and}\quad v_{[\theta _0,\lambda_0]}(x):=e^{i\theta _0} {\lambda_0^{-\frac{d-2}{2}}} v(\frac{x}{\lambda_0}).\notag
\end{equation} 
\begin{proposition}
	\label{P:W}
	There exists a function  $\varepsilon =\varepsilon ( \mathbf{d} )$, satisfying   $\lim _{ \mathbf{d} \rightarrow 0}\varepsilon ( \mathbf{d} )=0$, such that for any  
	  $u\in \dot H^1(\mathbb{R} ^d)$ with  $E(u)=E(W)$, the following inequality holds
	\begin{equation}
		\inf _{\theta \in \mathbb{R} ,\mu>0} \|u_{[\theta ,\mu]}-W\|_{\dot H^1}\le \varepsilon ( \mathbf{d}  (u)).\notag 
	\end{equation}  
\end{proposition} 

\begin{proof}
 Suppose  by contradiction that the claim does not hold; then there must exist  $\varepsilon _0>0$ and a sequence of  $\dot H^1$ functions   $\{f_n\}$     such that 
\begin{equation}
	E(f_n)=E(W) \quad\text{and }\quad \mathbf{d}(f_n)\rightarrow 0,\label{1124x1}
\end{equation}
but 
\begin{equation}
\inf _{\theta \in \mathbb{R} ,\mu>0} \|{(f_n)}_{[\theta ,\mu]}-W\|_{\dot H^1}>\varepsilon _0.\label{1123x1}
\end{equation}
  Applying bubble decomposition (Lemma \ref{L:bubble}) to  $\{f_n\}$, we obtain a subsequence in  $f_n$, (which for the sake of convenience is still denoted by  $f_n$)  satisfying the decomposition (\ref{11231}) and the properties (\ref{1124x2})--(\ref{11233}).  
  Since 
  \begin{equation}
  	\|\nabla f_n\|_{L^2}^2-\sum_{j=1}^{J} \|\nabla \phi^j\|_{L^2}^2- \|\nabla w_n^J\|_{L^2}^2 +\sum_{j=1}^{J} \|\nabla \phi^j\|_{L^2}^2\le  \|\nabla f_n\|_{L^2}^2,\notag 
  \end{equation}
  it follows from (\ref{11232}) and  $ \mathbf{d} (f_n)\rightarrow0$ that 
  \begin{equation}
  	\sum_{j=1}^{J^*} \|\phi^j\|_{\dot H^1}^2\le  \|W\|_{\dot H^1}^2.  \label{11234}
  \end{equation} 
  
  We next claim that 
   \begin{equation}
  	\|W\|_{\dot H^1}^{\alpha +2} \le \sum_{j=1}^{J^*} \|\phi^j\|_{\dot H^1}^{\alpha +2}.\label{11235}
  \end{equation}
In fact,  noting that by (\ref{1124x1}) 
  \begin{equation}
  	\int |x|^{-b}W^{\alpha +2}dx=\lim _{n\rightarrow \infty } \int |x|^{-b}|f_n|^{\alpha +2}dx,\notag
  \end{equation}
we deduce from (\ref{11233}),  (\ref{241})   and  Proposition  \ref{P:GN} that 
  \begin{equation}
  	\int |x|^{-b}W^{\alpha +2}dx\le  \limsup_{J\rightarrow J^*}\limsup _{n\rightarrow\infty } \sum_{j=1}^{J}\int |x|^{-b}|\phi^j(x-\frac{x_n^j}{\lambda_n^j})|^{\alpha +2}dx \le \sum_{j=1}^{J^*} \frac{ \|\phi^j\|_{\dot H^1}^{\alpha +2}}{ \|W\|_{\dot H^1}^{\alpha +2} }\int |x|^{-b}W^{\alpha +2}dx.\notag
  \end{equation}
This proves  (\ref{11235}).

Combining (\ref{11234}) and (\ref{11235}), we obtain  $J^*=1$ and 
  \begin{equation}
	f_n=(\lambda_n)^{-\frac{d-2}{2}}\varphi  (\frac{x-x_n}{\lambda_n})+w_n\quad\text{where}\quad  \||x|^{-b}|w_n|^{\alpha +2}\|_{L^1} + \|w_n\|_{\dot H^1}\rightarrow0.\label{12102}
\end{equation} 
Furthermore, either  $x_n\equiv0$ or  $\frac{|x_n|}{\lambda_n}\rightarrow\infty $ and   
\begin{equation}
	\|\varphi \|_{\dot H^1}= \|W\|_{\dot H^1},\ \int |x|^{-b}|\varphi |^{\alpha +2}dx=\int |x|^{-b}W^{\alpha +2}dx.\notag
\end{equation} 
Therefore, by   Proposition \ref{P:GN},  there exist  $\theta \in \mathbb{R} ,\mu>0$  such that 
\begin{equation}
\varphi =W_{[\theta ,\mu]}.\label{1271}
\end{equation}

\begin{claim}\label{C:361}
	The space parameter  $x_n$ in (\ref{12102}) must satisfy  $x_n\equiv0$.  
\end{claim}
\begin{proof}
	 Assume by contradiction that  $\frac{|x_n|}{\lambda_n}\rightarrow \infty $. We first prove 
	\begin{equation}
		\lim_{n\rightarrow \infty } \int |x|^{-b}|(\lambda_n)^{-\frac{d-2}{2}}\varphi (\frac{x-x_n}{\lambda_n})|^{\alpha +2}dx=\lim_{n\rightarrow\infty }\int |x+\frac{x_n}{\lambda_n}|^{-b}|\varphi (x)| ^{\alpha +2}dx=0.\label{12101}
	\end{equation}
	In fact, noting that for any  $\phi \in C_c^\infty (\mathbb{R} ^d)$,
	\begin{eqnarray}
		&&\int |x+\frac{x_n}{\lambda_n}|^{-b}|\phi (x)|^{\alpha +2}dx\notag\\
		&\lesssim & \int _{|x+\frac{x_n}{\lambda_n}|<\frac{1}{2}\frac{|x_n|}{\lambda_n}}    |x+\frac{x_n}{\lambda_n}|^{-b}|\phi (x)|^{\alpha +2}dx+\int _{|x+\frac{x_n}{\lambda_n}|>\frac{1}{2}\frac{|x_n|}{\lambda_n}} |x+\frac{x_n}{\lambda_n}|^{-b}|\phi (x)|^{\alpha +2}dx \notag\\
		&\lesssim &  \int _{|x|>\frac{1}{2}\frac{|x_n|}{\lambda_n} }  |x+\frac{x_n}{\lambda_n}|^{-b}|\phi (x)|^{\alpha +2}dx+(\frac{|x_n|}{\lambda_n})^{-b} \|\phi \|_{L^{\alpha +2}}^{\alpha +2} \notag
	\end{eqnarray} 
	tends to  $0$ as  $n\rightarrow\infty $,    we obtain (\ref{12101}) by a standard approximation argument. 
	
	Combining (\ref{12102}) and (\ref{12101}), we obtain the contradiction
	\begin{equation}
		\int |x|^{-b}W^{\alpha +2}dx=\lim_{n\rightarrow\infty }\int |x|^{-b}|f_n|^{\alpha +2}dx=0.\notag
	\end{equation} 
\end{proof}
It then follows from  (\ref{12102}),  (\ref{1271}) and Claim \ref{C:361} that 
  \begin{equation}
  	f_n(x)=\lambda_n^{-\frac{d-2}{2}}W_{[\theta ,\mu]}(\frac{x}{\lambda_n})+w_n \quad\text{with}\quad   \|w_n\|_{\dot H^1}\rightarrow0,\notag 
  \end{equation} 
  which contradicts (\ref{1123x1}).     Proposition \ref{P:W} is proved.  
\end{proof}

 Proposition \ref{P:W} together with the implicit theorem gives the following orthogonal decomposition. 
\begin{lemma}\label{l:implicit}
	There exists  $\delta_0>0$ such that for all  $u\in \dot H^{1}$ with  $E(u)=E(W)$,  $ \mathbf{d} (u)<\delta_0$, there exists a couple  $\theta ,\mu$ in  $\mathbb{R}\times (0,+\infty )$ with 
	\begin{equation}
		u_{[\theta ,\mu]}\bot iW,\qquad u_{[\theta ,\mu]}\bot W_1.\notag
	\end{equation}
	The parameters   $(\theta , \mu)\in \mathbb{R} \times (0,+\infty )$ are unique in  $\mathbb{R}/2\pi \mathbb{Z} \times \mathbb{R}$, and the mapping  $u\rightarrow (\theta ,\mu)$ is  $C^1$.  
\end{lemma}
\begin{proof}
Consider the following functionals on  $\mathbb{R}\times (0,\infty )\times \dot H^{1}$: 
	\begin{equation}
		J_0:(\theta ,\mu,f)\mapsto (f_{[\theta ,\mu]},iW)_{\dot H^{1}} \quad\text{and}\quad J_1:(\theta ,\mu,f)\mapsto (f_{[\theta ,\mu]},W_1)_{\dot H^{1}}.\notag 
	\end{equation}
By simple calculation,  we have  $J_0(0,1,W)=J_1(0,1,W)=0$ and 
	\begin{equation}
		\frac{\partial J_0}{\partial \theta }(0,1,W)=\int |\nabla W|^2,\qquad 	\frac{\partial J_0}{\partial \mu}(0,1,W)=0\notag
	\end{equation}
	\begin{equation}
		\frac{\partial J_1}{\partial \theta }(0,1,W)=0\qquad 	\frac{\partial J_1}{\partial \mu }(0,1,W)=-\int |\nabla W_1|^2.\notag
	\end{equation}
Thus by the implicit function
	theorem there exist  $\varepsilon _0,\eta_0>0$ such that for  $h\in \dot H^{1}$ with  $ \|h-W\|_{\dot H^1} <\varepsilon _0$, there exists a unique  $(\widetilde{\theta }(h),\widetilde{\mu}(h))\in C^1$ such that  $|\widetilde{\theta }|+|\widetilde{\mu}-1|<\eta_0$ and 
	\begin{equation}
 (h_{[\widetilde{\theta} ,\widetilde{\mu}]},iW)_{\dot H^{1}}=(h_{[\widetilde{\theta} ,\widetilde{\mu}]},W_1)_{\dot H^{1}}=0.		\label{291}
	\end{equation}   
		On the other hand,  by  Proposition \ref{P:W}, there exist a function  $\varepsilon $ and  $\theta _1,\mu_1$ such that 
	\begin{equation}
		\|u_{[\theta _1,\mu_1]}-W\|_{\dot H^1}\le \varepsilon ( \mathbf{d}  (u)).\notag
	\end{equation}  
	Therefore, for   $ \mathbf{d} (u)$  sufficiently small, we deduce from (\ref{291}) that  there exists    $(\widetilde{\theta }_1(u),\widetilde{\mu}_1(u))$ such that 
\begin{equation}
	((u_{[\theta _1,\mu_1]})_{[\widetilde{\theta }_1,\widetilde{\mu}_1]},iW)_{\dot H^1}=	((u_{[\theta _1,\mu_1]})_{[\widetilde{\theta }_1,\widetilde{\mu}_1]},W_1)_{\dot H^1}=0,\notag
\end{equation}
This completes the proof by taking  $\theta =\widetilde{\theta }_1+\theta _1$ and  $\mu=\widetilde{\mu}_1\mu_1$.    
\end{proof}

Let $u$ be a solution of (\ref{nls}) on an interval  $I$ such that  $E(u_0)=E(W)$ and write 
\begin{equation}
	 \mathbf{d}  (t):= \mathbf{d}  (u(t))=\left|  \int_{\mathbb{R} ^d} \left(|\nabla u(t,x)|^2-|\nabla W(x)|^2\right)dx \right|. \notag
\end{equation}
  According to  Lemma  \ref{l:implicit},  if $ \mathbf{d} (t)<\delta_0$  for all  $t\in I$,   there exist real parameters  $\theta (t),\mu(t)>0$ such that
\begin{equation}
	u_{[\theta (t),\mu(t)]}(t)=(1+\beta (t))W+\widetilde{u}(t),\label{814w1}
\end{equation}
where 
\begin{equation}
	1+\beta (t)=\frac{1}{\|W\|_{\dot H^{1}}^2}(u_{[\theta (t),\mu (t)]},W)_{\dot H^{1}}\quad\text{such that}\quad \widetilde{u}(t)\in H^\bot.\notag
\end{equation}
Define  $v(t)$ by 
\begin{equation}
	v(t):=\beta  (t)W+\widetilde{u}(t)=u_{[\theta (t),\mu (t)]}-W.\label{89w3}
\end{equation}
We can obtain  the following estimates regarding the parameter functions  in (\ref{814w1}) and (\ref{89w3}).  
\begin{lemma}
\label{L:modulation}
	Taking a smaller  $\delta_0$ if necessary, we have the following estimates on  $I$:
	\begin{equation}
		|\beta  (t)|\approx \|v(t)\|_{\dot H^{1}}\approx \|\widetilde{u}(t)\|_{\dot H^{1}}\approx  \mathbf{d} (u(t))\label{89w1}
	\end{equation}
	\begin{equation}
		|\beta  '(t)|+|\theta '(t)|+|\frac{\mu'(t)}{\mu(t)}|\lesssim \mu^2(t) \mathbf{d} (u(t)).\label{89w2}
	\end{equation}
\end{lemma}
The proof of Lemma \ref{L:modulation} is a consequence of  Proposition \ref{H} and of the equation satisfied by  $v$,  which will be discussed explicitly in the Appendix \ref{App:2}. 

\section{Global analysis-Virial.}\label{s:5}
In the previous section, we develop the modulation analysis which enables us to control the solution near the two dimensional manifold   $\{W_{[\theta ,\mu]}\}.  $ When the solution is away from the manifold, we use the monotonicity formula arising from Virial to control the solution. To this end, in this section we establish Virial estimates by incorporating the modulation estimates developed in Section \ref{s:4}. 

Let  $\varphi(x)$ be a smooth radial function such that 
\begin{equation}
	\varphi(r)=r^2,\ r\le1,\ \varphi(r)\ge 0\quad\text{and}\quad  \frac{d^2\varphi}{dr^2}(r)\le 2,\ r\ge0.\label{121w1}
\end{equation} 
Define  the truncated Virial 
\begin{equation}
	V_R(t):=\int_{\mathbb{R} ^d} \varphi_R(x)|u(t,x)|^2dx,\label{E:virial}
\end{equation}
where $\varphi _R(x)=R^2\varphi (\frac{x}{R})$.    For a solution  $u(t)$ of (\ref{nls}) with  $E(u)=E(W)$, the time derivatives of  $V_R(t)$  are computed as  
\begin{equation}
	\partial_{t}V_R(t,x)=2\text{Im} \int \overline{u}(t,x)\nabla u(t,x)\cdot \nabla  \varphi _R(x)dx,\qquad t\in \mathbb{R},\notag
\end{equation}
and 
\begin{equation}\label{dV_R}
	\partial_{tt}V_R(t)=\begin{cases}
		\ \ 4\alpha  \mathbf{d} (u(t))+A_R(u(t)) \qquad &\text{if }  \|u(t)\|_{\dot H^1}< \|W\|_{\dot H^1} , \\
		-4\alpha   \mathbf{d} (u(t))+A_R(u(t)) &\text{if }  \|u(t)\| _{\dot H^1}> \|W\| _{\dot H^1}.
	\end{cases} 
\end{equation}  
where
\begin{align*}
		A_{R}&(u(t)):=4\int _{|x|\ge R}(\frac{1}{r}\frac{\partial_{r}\varphi_R }{\partial r}-2) |\nabla u|^2dx+4\int_{|x|\ge R} (\frac{\partial_{rr}\varphi _R}{r^2}-\frac{\partial_{r}\varphi _R}{r^3})|x\cdot \nabla u|^2dx\notag\\
	&-\int \Delta ^2\varphi _R|u|^2dx-\frac{2\alpha }{\alpha +2}\int_{|x|\ge R} \left[\partial_{rr}\varphi _R+(d-1+\frac{2b}{\alpha })\frac{\partial_{r}\varphi _R}{r}-\frac{4(\alpha +2)}{\alpha }\right]|x|^{-b}|u|^{\alpha +2}]dx. \notag
\end{align*}
Indeed, an explicit calculation together with equation (\ref{nls}) yields 
\begin{eqnarray}
	\partial_{tt}V_R(t)&=&4\int \frac{\partial_{r}\varphi _R}{r}|\nabla u|^2dx+4\int (\frac{\partial_{rr}\varphi _R}{r^2}-\frac{\partial_{r}\varphi _R}{r^3})|x\cdot \nabla u|^2dx-\int |u|^2\Delta ^2\varphi _R\notag\\
	&&-\frac{2\alpha }{\alpha +2}\int \left[\partial_{rr}\varphi _R+(d-1+\frac{2b}{\alpha })\frac{\partial_{r}\varphi _R}{r}\right]|x|^{-b}|u|^{\alpha +2}dx\notag\\
	&=& 8(\int_{\mathbb{R}^d}|\nabla u(t)|^2-\int _{\mathbb{R}^d}|x|^{-b}|u(t)|^{\alpha +2} )+A_R(u(t)).\notag
\end{eqnarray}
By  $E(u)=E(W)$ and the Pohozhaev's identity (\ref{identity:pohozhaev}), we have 
\begin{equation}
	\int |\nabla u(t)|^2-\int |x|^{-b}|u|^{\alpha +2}dx=\begin{cases}
		\ \ \frac{\alpha }{2} \mathbf{d} (u(t))\qquad &\text{if }  \|u(t)\|_{\dot H^1}< \|W\|_{\dot H^1}  \\
		-\frac{\alpha }{2} \mathbf{d} (u(t))\qquad &\text{if }  \|u(t)\|_{\dot H^1}>\|W\|_{\dot H^1} \\
	\end{cases}\notag
\end{equation} 
which yields  (\ref{dV_R}).

The rest of this Section is devoted to giving proper estimates on  $\partial_{t}V_R(t)$ and  $A_R(t)$. 

\begin{lemma}[Virial estimate]
	\label{L:virial}
	Let  $u(t)$ be an  $\dot H^1$ solution of (\ref{nls}) with  $E(u)=E(W)$. For those  $t$ satisfying  $ \mathbf{d} (u(t))<\delta _0$, let 
	\begin{equation}
		u(t)_{[\theta (t),\mu(t)]}=W+v(t)\label{1211}
	\end{equation}     
	be the orthogonal decomposition of  $u(t)$ given by (\ref{814w1}) with the  bounds (\ref{89w1}) and (\ref{89w2}). We have 
	\begin{equation}
		|\partial_{t}V_R(t)|\lesssim R^2 \mathbf{d} (u(t)),\label{11291}
	\end{equation}  
	\begin{equation}
		A_R(u(t))\lesssim \int _{|x|\ge R}(\frac{|u(t,x)|^{\alpha +2}}{|x|^b}+\frac{|u(t,x)|^2}{|x|^2})dx,\label{11292}
	\end{equation}
	\begin{equation}
		|A_R(u(t))|\lesssim  \begin{cases}
			\int _{|x|\ge R} (|\nabla u(t)|^2+|x|^{-b}|u(t,x)|^{\alpha +2}+|x|^{-2}|u(t,x)|^2)dx,\\
			(\mu(t)R)^{-\frac{d-2}{2}} \mathbf{d} (u(t))+ \mathbf{d} (u(t))^2 \quad\text{if}\quad   \mathbf{d} (u(t))<\delta _0\quad\text{and}\quad |\mu (t)R|\gtrsim1.
		\end{cases}\label{11293}
	\end{equation}
\end{lemma}
\begin{proof}
	We first estimate  $\partial_{t}V_R(t)$.  By H\"older and Hardy's inequality, we obtain  
	\begin{equation}
		|\partial_{t}V_R(t)|\lesssim  \int _{|x|\le 2R} R^2\frac{|u(t,x)|}{|x|}|\nabla u(t,x)|dx \lesssim R^2 \|u\|_{L_t^\infty \dot H_x^1}^2\lesssim R^2 ( \mathbf{d} (u(t))+ \|W\|_{\dot H^1}^2 ).\notag
	\end{equation}  
	This proves  (\ref{11291}) in the case of  $ \mathbf{d} (u(t))\ge \delta _0$.  To get the bound in the case  $ \mathbf{d} (u(t))<\delta _0$,  we make  the change of variable  $x=y/\mu (t)$ and then apply the decomposition (\ref{1211}) to obtain 
	\begin{eqnarray}
		\partial_{t}V_R(t)&=&2\text{Im} \frac{R}{\mu (t)}\int \mu (t)^{-\frac{d-2}{2}}\overline{u}(t,\frac{y}{\mu (t)})\mu (t)^{-\frac{d}{2}}(\nabla u)(t,\frac{y}{\mu (t)})\cdot (\nabla \varphi) (\frac{y}{R\mu (t)})dy\notag\\
		&=&2R^2\text{Im} \int \frac{1}{R\mu (t)}(W+\overline{v})\nabla (W+v)\cdot (\nabla \varphi)(\frac{y}{R\mu(t)})dy.\notag
	\end{eqnarray}
	Write  $	\text{Im} [(W+\overline{v})\nabla (W+v)]=\text{Im} (W\nabla v+\overline{v}\nabla W+\overline{v}\nabla v)$ 
	and note that  on the support of  $\nabla \varphi (y/R\mu (t)),\ 1/R\mu (t)$ is bounded by  $2/|y|$. As a consequence of Cauchy–Schwarz and Hardy's inequality, we get
	the bound  $$|\partial_{t}V_R(t)|\lesssim  R^2(\|v(t)\|_{\dot H^{1}}+\|v(t)\|_{\dot H^{1}}^2),$$
	which together with (\ref{89w1}) yields (\ref{11291}) for  $ \mathbf{d} (u(t))\le \delta_0$. 
	
	We now turn to estimating  $A_R(u(t))$. Denote by  $\Omega:=\left\{x\in \mathbb{R} ^d:r\partial_{rr}\varphi _R\ge \partial_{r}\varphi _R \right\}$. Noting that  $\frac{\partial_{r}\varphi _R}{\partial r}\le 2r$ by (\ref{121w1}), we have  
	\begin{eqnarray}
		&&\int _{|x|\ge R}(\frac{1}{r}\frac{\partial_{r}\varphi _R}{\partial_{}r}-2)|\nabla u|^2dx+\int _{|x|\ge R}(\frac{\partial_{rr}\varphi _R}{r^2}-\frac{\partial_{r}\varphi _R}{r^3})|x\cdot \nabla u|^2dx\notag\\
		&\le & \int _{\{|x|\ge R\}\cap \Omega} (\frac{1}{r}\frac{\partial_{r}\varphi _R}{\partial_{}r}-2)|\nabla u|^2dx+\int _{\{|x|\ge R\}\cap  \Omega} (\frac{\partial_{rr}\varphi _R}{r^2}-\frac{\partial_{r}\varphi _R}{r^3}) r^2 |\nabla u|^2dx\notag\\
		&=&\int _{\{|x|\ge R\}\cap \Omega} (\partial_{rr}\varphi _R-2)|\nabla u|^2dx\le0,\notag
	\end{eqnarray}  
	which gives immediately (\ref{11292}) and the first line in (\ref{11293}).  
	
	To get the second bound when  $ \mathbf{d} (u(t))<\delta _0$ and  $\mu (t)R\gtrsim 1$, we recall  $W(x)\approx O(|x|^{-(d-2)})$  for  $|x|\gtrsim1$.  This together with the     decomposition  (\ref{1211}) and Lemma \ref{L:modulation} yields 
	\begin{eqnarray}
		&&|A_R(u(t))|= |A_{\mu (t)R}((u(t))_{[\theta (t),\mu (t)]})|=|A_{\mu (t)R}(W+v(t))-A_{\mu (t)R}(W)|\notag\\
		&\lesssim &   \|\nabla W\| _{L^2(|x|\ge \mu (t)R)} \|\nabla v(t)\|_{L^2}+ \|\nabla v(t)\|_{L^2}^2\notag\\
		&&+ \|v\|_{L^{\frac{2d}{d-2},2}}( \|W\|_{L^{\frac{2d}{d-2},2}(|x|\ge \mu (t)R)}^{\alpha +1}+ \|v(t)\|_{L^{\frac{2d}{d-2},2}}^{\alpha +1}  )+ \|W/|x|\|_{L^2(|x|\ge \mu (t)R)}   \|\nabla v(t)\| _{L^2}\notag\\
		&\lesssim &  (\mu (t)R)^{-\frac{d-2}{2}}  \mathbf{d} (u(t))+( \mathbf{d} (u(t)))^2.\notag
	\end{eqnarray}
	Lemma \ref{L:virial} is proved.  
\end{proof}

\section{Convergence to  $W$ in the Subcritical Case}\label{S:subcritical}
In this section, we focus on characterizing the nonscattering threshold solutions   when the kinetic energy is less than that of the ground state  $W$. The main result is the following.
\begin{proposition}
	\label{p2}
	Let  $u$ be a solution of (\ref{nls}) satisfying
	\begin{equation}
		E(u_0)=E(W)\quad\text{and}\quad \|u_0\|_{\dot H^{1}}<\|W\|_{\dot H^{1}},\label{871}
	\end{equation}
	  and  $I=(T_-,T_+)$  be its maximal interval of existence.  Assume that   $u$  does not scatter for the positive time, i.e. 
	\begin{equation}
		\|u\|_{S(0,T_+)}=+\infty ,\label{872}
	\end{equation} 
	then   $T_+=+\infty $ and  there exists  $\theta _0\in \mathbb{R},\mu_0>0$ and   $c>0$ such that 
	\begin{equation}
		\|u(t)-W_{[\theta _0,\mu_0]}\|_{\dot H^{1}}\lesssim e^{-ct},\qquad t\ge0.\label{87s2}
	\end{equation}
	Moreover, in the negative time direction,  $u$ exists globally and obeys  
	\begin{equation}
		\|u\|_{S(-\infty ,0 )} <\infty .\label{1119w1}
	\end{equation}
\end{proposition}
We first establish some properties for the nonscattering threshold solutions in subsection \ref{s:6.1}, and then give the proof of  Proposition  \ref{p2} in subsection \ref{s:6.2}.  
\subsection{Properties for Nonlinear Subcritical Threshold Solutions.}\label{s:6.1}
\begin{proposition}
	\label{p1}
	Let  $u$ be a solution of (\ref{nls}) satisfying (\ref{871}),  and  $I=(T_-,T_+)$  be its maximal interval of existence.  If  
	\begin{equation}
		 \|u\|_{S(0,T_+)}=+\infty ,\label{213w4} 
	\end{equation}
	then  \\
	(a)  	There exists a function  $\lambda $  on  $[0,T_+ ) $ 	such that the set
	\begin{equation}
		K_+:=\left\{u_{[\lambda (t)]}(t),\ t\in [0,T_+  )\right\}\label{E:compact}
	\end{equation}
	is relatively compact in  $\dot H^{1}$.    \\
	(b)  $T_+=+\infty $. \\
	(c) The function  $\lambda$ in  (a) satisfies 
	\begin{equation}
		\lim_{t\rightarrow  +\infty  }\sqrt {t}\lambda (t)=+\infty. \label{lim}
	\end{equation}
\noindent 	An analogous assertion holds on negative time direction. 
\end{proposition}
\begin{proof}
(a) As the proof follows that of \cite[ Proposition 3.1]{Liu-Zhang}, we will only sketch the main steps. \\
\noindent\textbf{Step 1.}  For any sequence   $\{\tau_n\}_{n\in \mathbb{N} }\subset [0,T_+)$,  there exists  $\lambda_n$  such that
\begin{equation}
	\lambda_n^{-\frac{d-2}{2}}u(\tau_n,\frac{x}{\lambda_n}) \quad\text{converges strongly in }\quad  \dot H^1(\mathbb{R} ^d) \quad\text{(up to a subsequence)}.\label{2142}
\end{equation}
By continuity of  $u$  it suffices to consider  $\tau_n\rightarrow T_+$.   Applying the profile decomposition (c.f. \cite[ Proposition 3.2]{Liu-Zhang}) to  $\{u(\tau_n)\}_{n\in \mathbb{N} }$, we deduce that there exist $J^*\in\mathbb{N}\cup\{\infty\}$; profiles $\phi^j\in \dot H^1\setminus\{0\}$; scales $\lambda_n^j\in(0,\infty)$; space translation parameters $x_n^j\in\mathbb{R}^d$; time translation parameters $t_n^j$; and remainders $w_n^J$ such that the following  decomposition holds for $1\leq J\leq J^*$:
\begin{equation}
	u_n:=u(\tau_n)=\sum_{j=1}^J (\lambda ^j_n)^{-\frac{d-2}{2}}[e^{it_n^j\Delta}\phi^j](\frac{x-x_n^j}{\lambda _n^j}) + w_n^J;\label{213w6}
\end{equation}
here  $w_n^J\in \dot H^1(\mathbb{R}^d)$ obeys
\begin{equation} 
	\limsup_{J\to J^*}\limsup_{n\to\infty} \|e^{it\Delta}w_n^J\|_{S(0,+\infty )}= 0.\notag
\end{equation}
Moreover,  for any  $J\ge1$, the following   energy decoupling properties  hold 
\begin{equation}\label{energy-decoupling}
	\begin{aligned}
		&\lim_{n\to\infty} \bigl\{ \|\nabla u_n\|_{L^2}^2 - \sum_{j=1}^J \|\nabla \phi^j \|_{L^2}^2 - \|\nabla w_n^J\|_{L^2}^2\bigr\} = 0, \\
		&\lim_{n\to\infty} \bigl\{E(u_n) - \sum_{j=1}^J E(e^{it_n^j\Delta}\phi^j)-E(w_n^J)\bigr\} = 0.
	\end{aligned}
\end{equation}
Furthermore,   either $t_n^j\equiv 0$ or $t_n^j\to\pm\infty$, and that either $x_n^j\equiv 0$ or $\frac{|x_n^j|}{\lambda _n^j}\to\infty$.

We now show that   $J^*=1$. Suppose by  contradiction that   $J^*\ge2$.   It then follows from  (\ref{871}), (\ref{energy-decoupling})  and Lemma \ref{l:coercive} that 
\begin{equation}
	E(\phi^j)<E(W)\quad\text{and}\quad  \|\nabla \phi^j\|_{L^2}< \|\nabla W\|_{L^2}\quad\text{for all}\quad    1\le j\le J^*.\label{213w5}
\end{equation}
We now construct scattering solutions to (\ref{nls}) corresponding to each profile. 
First, if   $t_n^j \equiv 0$, then we take $v^j$ to be the solution to \eqref{nls} with initial data $\phi^j$. This solution scatters due to  (\ref{213w5}) and Theorem \ref{T:0}.  If instead   $t_n^j\to\pm\infty$, we let $v^j$ be the solution that scatters to $e^{it\Delta}\phi^j$ as $t\to\pm\infty$ (see Theorem \ref{T:CP}).
In either of these cases, we then define
\begin{equation}
	v_n^j(t,x) = (\lambda_n^j)^{-\frac{d-2}2} v^j\left(\tfrac{t}{(\lambda_n^j)^2}+t_n^j,\tfrac{x}{\lambda_n^j}\right).\notag
\end{equation}

We then  construct a corresponding nonlinear profile decomposition of the form
\begin{equation} 
	u_n^J(t,x) = \sum_{j=1}^J v_n^j(t,x) + e^{it\Delta} w_n^J(x).\notag
\end{equation}	
By construction, we get 
\begin{equation}
	 \lim_{n\rightarrow \infty }\|u_n^J(0)-u_n\|_{\dot H^1}=0.\notag
\end{equation}
 Moreover, the arguments of \cite[Lemma 3.8]{Liu-Zhang} imply 
 	\begin{align}
 	& \limsup_{J\to J^*}\limsup_{n\to\infty}\bigl\{ \|u_n^J\|_{L_t^\infty \dot H_x^1(\mathbb{R} \times \mathbb{R} ^d)} + \| u_n^J\| _{S(\mathbb{R} )}\bigr\}\lesssim 1, \notag \\
 	& \limsup_{J\to J^*}\limsup_{n\to\infty} \|\nabla[(i\partial_t + \Delta)u_n^J + |x|^{-b}|u_n^J|^\alpha  u_n^J] \|_{L^2_tL_x^{\frac{2d}{d+2},2}(\mathbb{R} \times \mathbb{R} ^d)} = 0. \notag
 \end{align}
 Applying the stability result (Proposition~\ref{P:stab}), we derive bounds for the solutions $u_n$ that contradict  (\ref{213w4}). 

Having established  $J^*=1$,  (\ref{213w6}) simplifies to
	\begin{equation}
	u_{n}(x)= \lambda _n^{-\frac{d-2}{2}} [e^{it_n\Delta}\phi](\frac{x-x_n}{\lambda _n})+w_n(x), \label{521}
\end{equation}
where  either $t_n\equiv0$ or  $t_n\rightarrow \pm \infty $ and either  $x_n\equiv0$ or  $\frac{|x_n|}{\lambda _n}\rightarrow \infty $.
We now observe that 
\begin{equation}
	 \|\nabla w_n(x)\|_{L^2}\rightarrow0\quad\text{as}\quad  n\rightarrow+\infty .\label{2141}
\end{equation}
Indeed, otherwise, by (\ref{energy-decoupling}) and Lemma \ref{l:coercive}, we see that  $E(\phi)<E(W)$ and  $ \|\nabla \phi\|_{L^2} < \|\nabla W\|_{L^2} $. By  the same arguments used above, we deduce that   $ \|u\|_{S(0,+\infty )} <+\infty $,  contradicting   (\ref{213w4}).  

To see that the space shifts must obey $x_n\equiv 0$, we note that if $\frac{|x_n|}{\lambda _n}\to\infty$ then Proposition~\ref{P:embed} yields global scattering solutions $v_n$ to \eqref{nls} with $$v_n(0)= \lambda _n^{-\frac{d-2}{2}} [e^{it_n\Delta}\phi](\frac{x-x_n}{\lambda _n}).$$  Applying the stability result (Proposition~\ref{P:stab}), this implies uniform space-time bounds for the solutions $u_n$, contradicting (\ref{213w4}). To see that the time shifts must obey $t_n\equiv 0$, we note that if $|t_n|\to\infty$ then the functions  $e^{it\Delta }u_{n}$ (which has asymptotically vanishing space-time norm) define good approximate solutions obeying global space-time bounds for $n$ large.  In particular, an application of Proposition~\ref{P:stab} would again yield uniform space-time bounds for the $u_n$, contradicting (\ref{213w4}). 

Finally, by (\ref{521}) and (\ref{2141}) we get 
\begin{equation}
	 \|\lambda_n^{\frac{d-2}{2}}u(\tau_n, \lambda_nx)-\phi\|_{\dot H^1}\rightarrow0\quad\text{as}\quad  n\rightarrow+\infty ,\notag
\end{equation}
which complete the proof of (\ref{2142}).  

\noindent\textbf{Step 2.} We define  $\lambda(t)$.  Note that for any  $t\in [0,T_+)$ 
\begin{equation}
	2E(W)=2E(u(t))\le  \|u(t)\|_{\dot H^1}^2\le  \|W\|_{\dot H^1}^2.\label{214x1}  
\end{equation} 
Fixing  $t\in [0,T_+)$, define 
\begin{equation}
	\lambda(t):=\sup \left\{\lambda>0,\quad\text{ such that } \quad\int _{|x|\le 1/\lambda}|\nabla u(t,x)|^2 dx=E(W) \right\}.\notag
\end{equation} 
By (\ref{214x1}),  $0<\lambda(t)<\infty $. Let  $(t_n)_n$ be a subsequence in  $[0,T_+)$.  As proven in Step 1, up to the extraction of a subsequence, there exists a sequence  $(\lambda_n)_n$ such that  $(u_{[\lambda_n]}(t_n))_n$  converges in  $\dot H^1$ to some  $v_0\in \dot H^1$.    One may check directly, using (\ref{214x1}) and the definition of  $\lambda(t)$  that 
\begin{equation}
	C^{-1}\lambda(t_n)\le \lambda_n\le C\lambda(t_n),\notag
\end{equation}         
which shows (extracting again subsequences if necessary) the convergence of $(u_{[\lambda(t_n)]}(t_n))_n$ in  $\dot H^1$. The compactness of  $\overline{K}_+$  is proven, which concludes the proof of  (a).  

The proofs of (b) and (c) are same as the proofs of [Lemma 2.8, Step 2] and [Lemma 3.3, Step 2] in \cite{2008GFA}, so we omit the details. 
\end{proof}

The next observation on  $\lambda(t)$ is that  $\lambda(t)$ is basically comparable to  $\mu (t)$ given by (\ref{814w1}) when the solution  $u(t)$ close to the manifold $\left\{W_{[\theta ,\mu]} \right\}$.  
\begin{lemma}
	\label{L:comparable}
	Let  $u$ be the solution of (\ref{nls}) on the time interval  $I$ satisfying (\ref{E:compact}). Suppose  $ \mathbf{d} (u(t))<\delta _0$ on  $I$, and hence the orthogonal decomposition (\ref{89w3}) holds.   Then there exist constants  $0<c<C<\infty $  such that 
	\begin{equation}
		c<\frac{\lambda(t)}{\mu(t)}<C\qquad \forall t\in I.\notag
	\end{equation}  
\end{lemma}
\begin{proof}
	We argue by contradiction.   Supposing this is not true, there must exist a sequence of times  $t_n\in I_n$  such that  
	\begin{equation}
		\frac{\mu(t_n)}{\lambda(t_n)}\longrightarrow0\quad\text{or}\quad \frac{\mu (t_n)}{\lambda(t_n)}\longrightarrow\infty .\notag
	\end{equation} 
	This implies directly that 
	\begin{equation}
		W_{[\frac{\lambda(t_n)}{\mu (t_n)}]}{\rightharpoonup}0\quad\text{weakly in}\quad \dot H^1.\label{124w7}
	\end{equation}
	From the compactness in (\ref{E:compact}), we can extract a subsequence and find $V\in \dot H^1$  such that  
	\begin{equation}
		(u(t_n))_{[\lambda(t_n)]}\rightarrow V \quad\text{in}\quad  \dot H^1\label{271}
	\end{equation} 
	which along with  $ \mathbf{d} (u(t))<\delta _0$ implies 
	\begin{equation}
		\|V\|_{\dot H^1}^2=\lim_{n\rightarrow \infty }  \|u(t_n)\|_{\dot H^1}^2= \|W\|_{\dot H^1}^2-\lim_{n\rightarrow\infty } \mathbf{d} (u(t_n))> \|W\|_{\dot H^1}^2-\delta _0.\label{124w6}    
	\end{equation} 
	On the other hand, rewriting the decomposition (\ref{89w3}) as
	\begin{equation}
		(u(t_n))_{{\lambda(t_n)}}=(W+v(t_n))_{[-\theta (t_n),\frac{\lambda(t_n)}{\mu (t_n)}]},\notag
	\end{equation}
	we deduce  from (\ref{124w7}) and (\ref{271}) that 
	\begin{equation}
		(v(t_n))_{[-\theta (t_n),\frac{\lambda(t_n)}{\mu (t_n)}]}\stackrel{}{\rightharpoonup}V\quad\text{weakly in}\quad \dot H^1.\notag
	\end{equation}
	Therefore, by the estimate  $ \|v(t)\|_{\dot H^1}\lesssim   \mathbf{d} (u(t)) $ in  (\ref{89w1}) 
	\begin{equation}
		\|V\|_{\dot H^1}^2\le \liminf_{n\rightarrow\infty }  \|(v(t_n))_{[-\theta (t_n),\frac{\lambda(t_n)}{\mu (t_n)}]}\|_{\dot H^1}\lesssim   \mathbf{d} (u(t_n))<\delta _0,\notag 
	\end{equation}
	which contradicts  (\ref{124w6}) by replacing  $\delta _0$ by a smaller one. Lemma \ref{L:comparable} is proved.  
\end{proof}
The precompactness in  Proposition \ref{p1} implies that  $u(t)$  keeps getting closer to the manifold  $\left\{W_{[\theta ,\mu]} \right\}$.  
\begin{lemma}\label{l5}
	Let  $u$ be a solution of (\ref{nls}), defined on  $[0,+\infty )$,   satisfying (\ref{871}) and (\ref{872}). Then  there exists a sequence  $t_n\rightarrow+\infty $ such that  $ \mathbf{d} (u(t_n))$  tends to  $0$ as  $n\rightarrow +\infty $.   
\end{lemma}
\begin{proof}
	The compactness in (\ref{E:compact}) implies directly that for any  $\varepsilon >0$, 
	there exists  $\rho _\varepsilon >0$ sufficiently large so that 
	\begin{equation}
		\sup _{t\in [0,\infty  ) } \int _{|x|>\frac{\rho _\varepsilon}{\lambda (t)} }|\nabla u(t,x)|^2+|x|^{-b}|u(t,x)|^{\alpha +2}+|x|^{-2}|u(t,x)|^2dx<\varepsilon . \label{87s1}
	\end{equation}
	On the other hand,  by (\ref{lim}), there exists  $t_0$ such that 
	\begin{equation}
		\lambda (t)\ge \frac{\rho _\varepsilon }{\varepsilon \sqrt {t}},\qquad \forall t\ge t_0.\notag
	\end{equation}
	Fix  $T\ge t_0$ and let  $R=\varepsilon \sqrt {T}$. Then  $R\ge \frac{\rho _\varepsilon }{\lambda (t)}$ for  $t\in [t_0,T]$.   Applying Lemma \ref{L:virial} for  $t\in [t_0,T]$ and using (\ref{87s1}), we obtain 
	\begin{equation}
		|\partial_{t}V_R(t)|\lesssim R^2=\varepsilon ^2T,\label{1252}
	\end{equation}
	\begin{equation}
		|A_R(u(t))|\lesssim  \int _{|x|>R}|\nabla u|^2+|x|^{-b}|u|^{\alpha +2}+|x|^{-2}|u|^2dx\lesssim \varepsilon.\label{1253}
	\end{equation}
	Substituting (\ref{1253}) into the first line in   (\ref{dV_R}), we obtain  
	\begin{equation}
		\partial_{tt}V_R(t)\ge 4\alpha  \mathbf{d} (u(t))-C\varepsilon .\notag
	\end{equation}
	Integrating it over  $[t_0,T]$ and using (\ref{1252}), we have 
	\begin{equation}
		\int_{t_0}^T \mathbf{d} (u(t))dt\lesssim  \varepsilon (T-t_0)+\varepsilon ^2T.\notag
	\end{equation}
	As $\varepsilon>0$ was arbitrary,  this implies 
	\begin{equation}
		\lim _{T\rightarrow +\infty }\frac{1}{T}\int_0^T \mathbf{d} (u(t))dt=0.\notag
	\end{equation}
	The convergence of  $ \mathbf{d} (u(t))$  along a sequence of time is proved.  
\end{proof}

Lemma \ref{L:comparable} implies that we can replace  $\lambda(t)$ by  $\mu(t)$ on the interval where  $ \mathbf{d} (u(t))<\delta _0$. From the derivative estimate of  $\mu (t)$ in Lemma \ref{L:modulation}, it is reasonable to expect 
\begin{equation}
	\frac{|\lambda'(t)|}{\lambda^3(t)}\lesssim  \mathbf{d} (u(t))\qquad \forall t\ge0.\notag
\end{equation}     
In fact, by the argument in \cite[Lemma A.3]{YZZ},  we can modify  $\lambda(t)$ such that it is differentiable almost everywhere and satisfies 
\begin{equation}\label{E:lambda}
	\left|\frac{1}{\lambda^2(a)}-\frac{1}{\lambda^2(b)}\right|\lesssim  \int_a^b \mathbf{d} (u(t))dt,\qquad \forall [a,b]\subset [0,\infty ).
\end{equation}
Moreover, the compactness  property (\ref{E:compact})  still holds  for the modified  $\lambda(t)$.  

We will revist this estimate later when we prove the uniform lower bound for  $\lambda(t)$. Now we turn to considering the distance function  $ \mathbf{d} (u(t))$   with the goal of proving the exponential decay of  $ \mathbf{d} (u(t))$. We star by showing the following.   
\begin{lemma}
	\label{L:integral for d(u)}
	Let  $u$ be the solution of (\ref{nls}) satisfying (\ref{E:compact}). Then for any  $[a,b]\subset [0,\infty )$, 
	\begin{equation}
		\int_a^b \mathbf{d} (u(t))dt\lesssim  \sup _{t\in [a,b]}\frac{1}{\lambda^2(t)}[ \mathbf{d} (u(a))+ \mathbf{d} (u(b))].\label{1241}
	\end{equation}  
\end{lemma} 
\begin{proof}
	Estimate	(\ref{1241}) is scaling invariant; by rescaling the solution, we only need to prove the estimate with addition assumption  $\min _{t\in [a,b]}\lambda(t)=1$.    
	
	Let  $V_R(t)$ be defined by (\ref{E:virial}). Then 
 $	|\partial_{t}V_R(t)|\lesssim  R^2 \mathbf{d} (u(t))\notag$ 
	and 
	\begin{equation}
		\partial_{tt}V_R(t)=-4\alpha  \mathbf{d} (u(t))+A_R(u(t)).\notag
	\end{equation} 
	If  $ \mathbf{d} (u(t))<\delta _1<\delta _0$, the second line of (\ref{11293}) and Lemma \ref{L:comparable} imply 
	\begin{equation}
		|A_R(u(t))|\lesssim  (R^{-\frac{d-2}{2}}+\delta _1) \mathbf{d} (u(t)).\notag
	\end{equation} 
	If  $ \mathbf{d} (u(t))\ge \delta _1$, the first line of (\ref{11293}) and (\ref{87s1}) give that for  $R\ge \rho_\varepsilon$
	\begin{equation}
		|A_R(u(t))|\lesssim \varepsilon \lesssim \frac{\varepsilon }{\delta _1} \mathbf{d} (u(t)).\notag
	\end{equation}  
	Choosing  $R$  sufficiently large,  $\delta _1$  sufficiently small and then  $\varepsilon $  sufficiently small, we deduce that 
	\begin{equation}
		|A_R(u(t))|\le 2\alpha  \mathbf{d} (u(t)).\notag
	\end{equation}   
	Hence 
	\begin{equation}
		\partial_{tt}V_R(t)\le -2\alpha  \mathbf{d} (u(t))\qquad \forall t\in [a,b].\notag
	\end{equation}
	Integrating the above inequality from  $a$ to  $b$ gives (\ref{1241}).     
\end{proof}

The major obstacle of translating the integration estimate to the pointwise decay of  $ \mathbf{d} (u(t)) $ is the uniform lower bound of  $\lambda(t)$. We will show this is indeed the case knowing  $ \mathbf{d} (u(t))$  converges to  $0$ along a sequence of time, a result that can be deduced again from Virial analysis. We prove these results in the following Lemma. 
\begin{lemma}\label{L:lambda}
	Let  $u$ be the solution  of (\ref{nls}) satisfying (\ref{871}) and (\ref{872}). Then there exists a constant  $c>0$  such that 
	\begin{equation}
		\inf _{t\in [0,\infty )}\lambda(t)\ge c.\notag
	\end{equation}  
\end{lemma}   
\begin{proof}
	Let the sequence  $t_n$ be determined by Lemma \ref{l5}  such that   $ \mathbf{d} (u(t_n))\rightarrow0$  as  $n\rightarrow\infty $. Then for any  $\varepsilon >0$, there exists  $N_\varepsilon >0$  such that 
	\begin{equation}
		\mathbf{d} (u(t_{N_\varepsilon }))+ \mathbf{d} (u(t_m))\le \varepsilon \qquad \forall m\ge N_\varepsilon .\label{124w5}
	\end{equation}    
	Take any  $\tau\in [t_{N_\varepsilon },\infty )$ and any  $m\ge N_\varepsilon $  such that   $\tau\in [t_{N_\varepsilon },t_m]$. Applying (\ref{E:lambda}) on       $[t_{N_\varepsilon },\tau]$ and then using Lemma \ref{L:integral for d(u)}, we estimate 
	\begin{eqnarray}
		\left|\frac{1}{\lambda^2(\tau)}-\frac{1}{\lambda^2(t_{N_\varepsilon })}\right|&\lesssim &  \int _{t_{N_\varepsilon }}^\tau  \mathbf{d} (u(t))dt\lesssim  \int _{t_{N_\varepsilon }}^{t_m} \mathbf{d} (u(t))dt\notag\\
		&\lesssim &  \sup _{t\in [t_{N_\varepsilon },t_m]}\frac{1}{\lambda^2(t)}\times ( \mathbf{d} (u(t_{N_\varepsilon }))+ \mathbf{d} (u(t_m))).\notag
	\end{eqnarray} 
	It then follows from (\ref{124w5}) and the triangle inequality that 
	\begin{equation}
		\frac{1}{\lambda^2(\tau)}\le C\varepsilon  \sup _{t\in [t_{N_\varepsilon },t_m]}\frac{1}{\lambda^2(t)}+\frac{1}{\lambda^2(t_{N_\varepsilon })}\qquad \forall \tau\in [t_{N_\varepsilon },t_m].\notag
	\end{equation}
	Taking  $\varepsilon >0$  sufficiently small yields 
	\begin{equation}
		\sup _{\tau\in [t_{N_\varepsilon },t_m]}\frac{1}{\lambda^2(\tau)}\le \frac{2}{\lambda^2(t_{N_\varepsilon })} \quad\text{and thus}\quad \sup _{\tau\in [t_{N_\varepsilon },\infty )}\frac{1}{\lambda^2(\tau)}\le \frac{2}{\lambda^2(t_{N_\varepsilon })}\notag
	\end{equation}   
	by letting  $m\rightarrow \infty $. The uniform bound for  $\frac{1}{\lambda(\tau)}$ comes from this and the boundedness on the closed interval  $[0,t_{N_\varepsilon }]$. Lemma \ref{L:lambda} is proved.    
\end{proof}
\subsection{The proof of  Proposition  \ref{p2}.}\label{s:6.2}
\begin{proof}
	The assertion    $T_+=+\infty $  follows directly from   Proposition  \ref{p1}.    To prove (\ref{87s2}), the key is to show 
	\begin{equation}
		\lim_{t\rightarrow+\infty } \mathbf{d} (u(t))=0.\label{1232}
	\end{equation}  
	We star by proving 
	\begin{equation}
		\int_t^\infty  \mathbf{d} (u(s))\lesssim  e^{-ct}\qquad \forall t\ge0. \label{1231}
	\end{equation}
	In fact, by Lemmas  \ref{L:integral for d(u)} and \ref{L:lambda} 
	\begin{equation}
		\int _t^{t_n} \mathbf{d} (u(s))ds\lesssim  \mathbf{d} (u(t))+ \mathbf{d} (u(t_n)),\notag
	\end{equation}
	where  $\left\{t_n \right\}$ is the sequence in Lemma \ref{l5}  such that   $ \mathbf{d}(u(t_n))\rightarrow 0$. Letting  $n\rightarrow\infty $ gives immediately 
	\begin{equation}
		\int _t^\infty  \mathbf{d} (u(s))ds\lesssim  \mathbf{d} (u(t))\qquad \forall t\ge0,\notag
	\end{equation}   
	which together with Gronwall's inequality yields (\ref{1231}) for some  $c>0$.
	
	We now prove (\ref{1232}).     Assume that (\ref{1232}) does not hold. Then extracting a subsequence from  $(t_n)$, there exist  $0<\delta _1<\delta _0$  and   $t_n'>t_n$  such that 
	\begin{equation}
		\mathbf{d} (u(t_n'))=\delta _1\quad\text{and}\quad 0< \mathbf{d} (u(t))<\delta _1\qquad \forall t\in (t_n,t_n'),\label{1233}
	\end{equation}  
	where  $\delta _0$ is such that (\ref{814w1}) and Lemma \ref{L:modulation} hold. Let  $\beta  (t)$ be the parameter in the decomposition (\ref{814w1}) on the interval  $(t_n,t_n')$.  By Lemma \ref{L:modulation},  $|\beta  '(t)|\lesssim   \mathbf{d} (u(t))$ for  $t\in (t_n,t_n')$, thus (\ref{1231}) implies that  $\beta  (t_n)-\beta  (t_n')$ tends to  $0$. Furthermore, again by Lemma \ref{L:modulation},  $|\beta  (t)|\approx  \mathbf{d} (u(t))$, which shows that  $ \mathbf{d} (u(t_n'))$ tends to  $0$, contradicting (\ref{1233}). This proves (\ref{1232}).            
	
	As a consequence of (\ref{1232}), we may decompose  $u$ for large  $t$
	\begin{equation}
		u_{[\theta (t),\mu(t)]}=(1+\beta  (t))W+\widetilde{u}(t),\qquad \widetilde{u}(t)\in H^\bot .\notag
	\end{equation}
	Then (\ref{87s2}) is equivalent to the existence of    $\theta _\infty \in \mathbb{R},  \mu_\infty \in (0,\infty )  $ and  $c>0$  such that  
	\begin{equation}
		\mathbf{d} (u(t))+|\beta  (t)|+\|\widetilde{u}\|_{\dot H^{1}}+|\theta (t)-\theta _\infty |+|\mu(t)-\mu_\infty |\lesssim e^{-ct},\qquad t\ge0.\label{124w1}
	\end{equation}

	We star by proving  that there exists  $\mu_ \infty \in (0,\infty ) $  such that  
	\begin{equation}
		\lim_{t\rightarrow +\infty } \mu (t)=\mu _\infty .\label{124w2}
	\end{equation} 
	Combining the estimates (\ref{89w2}),  (\ref{E:lambda}) and (\ref{1231}) we immediately see that   $\frac{1}{\mu^2(t)}$ and  $\frac{1}{\lambda^2(t)}$  converge as  $t\rightarrow \infty $.   Therefore, by Lemma \ref{L:comparable},   the proof of (\ref{124w2}) reduces to preclude the possibility that 
	\begin{equation}
		\lim_{t\rightarrow \infty }\frac{1}{\lambda^2(t)}=0.\label{125x2}
	\end{equation} 
	Assume by contradiction that (\ref{125x2}) holds.  Recalling  $ \mathbf{d} (u(t_n))\rightarrow0$, for any  $\varepsilon >0$, there must exist   $N_0\in \mathbb{N} $  such that  
	\begin{equation}
		\frac{1}{\lambda(t)}<\varepsilon \qquad \forall t\ge t_{N_0}\quad\text{and}\quad  \mathbf{d} (u(t_n))<\varepsilon \qquad \forall n\ge N_0.\notag
	\end{equation}   
	Taking any  $t_*\ge t_{N_0}$ and applying (\ref{E:lambda}), (\ref{1241}) we obtain 
	\begin{eqnarray}
		\left| \frac{1}{\lambda^2(t_*)}-\frac{1}{\lambda^2(t_n)}\right|&\lesssim &  \left|\int _{t_*}^{t_n} \mathbf{d} (u(t))dt\right|\lesssim \int _{t_{N_0}}^{t_n}  \mathbf{d} (u(t))dt\notag\\
		&\le& C\sup _{t\in [t_{N_0},t_n]}\frac{1}{\lambda^2(t)}[ \mathbf{d} (u(t_n))+ \mathbf{d} (u(t_{N_0}))].\notag  
	\end{eqnarray} 
	Letting  $n\rightarrow\infty $ we have 
	\begin{equation}
		\frac{1}{\lambda^2(t_*)}\le C\sup _{t\in [t_{N_0},\infty )}\frac{1}{\lambda^2(t)} \mathbf{d} (u(t_{N_0}))\le C\varepsilon \sup _{t\in [t_{N_0},\infty )}\frac{1}{\lambda^2(t)}.\notag
	\end{equation} 
	Choosing  $C\varepsilon \le \frac{1}{2}$ and taking supremum in  $t_*$ over  $[t_{N_0},\infty )$, we obtain  $\frac{1}{\lambda(t)}=0$ for all  $t\ge t_{N_0}$, which is a contradiction. Therefore, we obtain (\ref{124w2}).  
	
	Estimate (\ref{124w1}) is then a straightforward consequence of  Lemma \ref{L:modulation} and the boundedness of  $\mu$. The proof of  (\ref{87s2}) is complete.

	Finally, we prove (\ref{1119w1}).   Assume by contradiction that  $ \|u\|_{S(T_- ,0)}=+\infty $.  Then by  Proposition \ref{p1} and  Lemma \ref{L:lambda}, applied forward and backward,   $T_-=-\infty $ and  there exits a function  $\lambda(t) $ defined on  $\mathbb{R} $ with uniform lower bound,  such that  the set  $\{(u(t))_{[\lambda(t)]},t\in \mathbb{R} \}$  is relatively compact in  $\dot H^1$. Furthermore
	\begin{equation}
		\lim _{t\rightarrow +\infty } \mathbf{d} (u(t))=\lim _{t\rightarrow -\infty } \mathbf{d} (u(t))=0.\notag
	\end{equation} 
	Then by Lemma \ref{L:integral for d(u)},  we have  $\int _{-\infty }^{+\infty } \mathbf{d} (u(t))dt=\lim _{n\rightarrow+\infty }\int _{-n}^{+n} \mathbf{d} (u(t))dt=0$. Thus  $ \mathbf{d} (u_0)=0$,  which contradicts (\ref{871}). This proves (\ref{1119w1}).    
\end{proof}
\section{Convergence to  $W$ in the Supercritical Case.} \label{s4}
In this section,  we characterize solutions to (\ref{nls})   if the kinetic energy is greater than that of the ground state. More precisely, we prove that if the threshold solution does not blow up  in finite time, then it   converges  exponentially to the ground state. 

\begin{proposition}\label{p3}
	Let  $u$ be a  radial solution to  (\ref{nls}) defined on  $[0,+\infty  ) $ satisfying 
	\begin{equation}
		E(u_0)=E(W) \quad\text{and}\quad \|u_0\|_{\dot H^{1}}>\|W\|_{\dot H^{1}}.\label{886}
	\end{equation} 
	Assume furthermore that  $u_0\in L^2(\mathbb{R}^d)$. Then there 	exist constants   $\theta _0\in \mathbb{R}$,  $\mu_0,c>0$ such that 
	\begin{equation}
		\|u(t)-W_{[\theta _0,\mu_0]}\|_{\dot H^{1}}\lesssim e^{-ct},\qquad \forall t\ge0.\notag		
	\end{equation}
	A similar result holds for negative times if  $u $ is defined on  $( -\infty ,0].$
\end{proposition}
		\begin{corollary}\label{c:210}
	 Let  $u$  be a   radial solution of  (\ref{nls}) satisfying  (\ref{886}) and such that   $u_0\in L^2(\mathbb{R} ^d)$.   Then  $u$  is not defined on  $\mathbb{R} ^d$.  
\end{corollary}
We start by proving the following Lemma. 
\begin{lemma}
	Suppose  $u$ is the solution in  Proposition \ref{p3}.   Then we have the following: \\
	(a) On the interval  $I$ where  $ \mathbf{d} (u(t))<\delta _0$, there exists  $c>0$   such that   $\mu(t)$  appearing in the modulation decomposition (\ref{814w1}) satisfies 
	\begin{equation}
		\mu (t)\ge c,\qquad \forall t\in I.\label{E:mu}
	\end{equation}
	\noindent (b) There exists  $R_0=R_0(\delta _0,W, \|u_0\|_{L^2} )$ such that for all  $R\ge R_0$ 
	\begin{equation}
		A_R(u(t))\le 2\alpha  \mathbf{d} (u(t)),\quad \forall t\in I.\label{C:A_R}
	\end{equation} 
\end{lemma}
\begin{proof}
	We first prove (\ref{E:mu}). Taking the  $L^2$ norm on both sides of (\ref{814w1}) and using  $ \|v(t)\|_{L^{\frac{2d}{d-2}}}\lesssim   \|v\|_{\dot H^1}\le C\delta _0  $ from Lemma \ref{L:modulation},  we have 
	\begin{eqnarray}
		\mu (t) \|u(t)\| _{L^2}&\ge&  \|W+v(t)\| _{L^2(|x|\le 1)}\notag\\
		&\ge & \|W\|_{L^2(|x|\le 1)}-C \|v(t)\|  _{L^{\frac{2d}{d-2}}}\ge  \|W\|_{L^2(|x|\le 1)}-C\delta _0.\notag 
	\end{eqnarray}  
	Inequality (\ref{E:mu}) then follows from the mass conservation. 
	
	We now turning to prove (\ref{C:A_R}).  By the radial Sobolev inequality 
	\begin{equation}
		\sup _{x\in \mathbb{R} ^d}|x|^{\frac{d-1}{2}}|u(x)|\lesssim   \|u\|_{L^2}^{1/2} \|\nabla u\|_{L^2}^{1/2},\notag  
	\end{equation}
	we have 
	\begin{equation}
			\int _{|x|>R} \frac{|u(t,x)|^{\alpha +2}}{|x|^b}\lesssim  R^{-b-\frac{d-1}{2}\alpha } \|u\|_{L^2}^{2+\frac{\alpha }{2}} \|\nabla u\|_{L^2}^{\frac{\alpha }{2}} .\notag
	\end{equation}
	This together with  the bound in (\ref{11292}) and Young's inequality implies 
	\begin{eqnarray}
		A_R(u(t))&\lesssim & \int _{|x|>R}\frac{|u(t,x)|^{\alpha +2}}{|x|^b}dx+\int _{|x|>R}\frac{|u(t,x)|^2}{|x|^2}dx\notag\\
		&\lesssim &  R^{-b-\frac{d-1}{2}\alpha }  ( \mathbf{d} (u(t)) + \|W\|_{\dot H^1}^2 )+R^{-2} \|u_0\|_{L^2}.\notag  
	\end{eqnarray} 
	By taking  $R$ large enough depending on  $ \|u_0\|_{L^2},\delta _0 $ and  $W$, we obtain (\ref{C:A_R}) in the case of  $ \mathbf{d} (u(t))\ge \delta _0$. In the remaining case when  $ \mathbf{d} (u(t))<\delta _0$, (\ref{C:A_R}) follows directly from  (\ref{E:mu}) and the second bound of (\ref{11293}).
\end{proof}

We are ready to prove  Proposition  \ref{p3}.  
\begin{proof}[\textbf{Proof of  Proposition \ref{p3}.}]
	By (\ref{dV_R}) and (\ref{C:A_R}), we have
	\begin{equation}
		\partial_{tt}V_R(t)\le -2\alpha  \mathbf{d} (u(t)),\qquad t\ge0.\label{ineq2}
	\end{equation}
	Thus, since  $\partial_{tt}V_R(t)<0$ and  $V_R(t)>0$ for all  $t\ge0$, it follows that    
	\begin{equation}
		\partial_{t}V_R(t)>0\quad\text{for all}\quad t>0.\label{126w1}
	\end{equation}
	Integrating (\ref{ineq2}) between  $t$ and  $T$, and using (\ref{11291}) we get 
	\begin{equation}
		2\alpha \int_t^T \mathbf{d} (u(s))ds\le  \partial_{t}V_R(t)-\partial_{t}V_R(T)\le \partial_{t}V_R(t)\lesssim  R^2 \mathbf{d} (u(t)).\notag
	\end{equation}    
	Letting  $T$ tend to infinity yields  $\int_t^\infty  \mathbf{d} (u(s))ds\lesssim  \mathbf{d} (u(t))$, and thus by the Gronwall's lemma, 
	\begin{equation}
		\int_t^\infty  \mathbf{d} (u(s))ds\lesssim e^{-ct},\qquad t\ge0. \notag
	\end{equation}  
	As a direct implication, there exists a sequence  $\left\{t_n \right\}\subset (0,\infty )$  such that   $\lim_{n\rightarrow \infty } \mathbf{d} (u(t_n))=0$. Therefore, we can perform the decomposition (\ref{814w1}) in the neighborhood of  $t_n$ for large  $n$. We claim that 
	\begin{equation}
		\mu (t_n)\lesssim 1.\label{1261}
	\end{equation}
	Indeed, if this is not true, passing to a subsequence, we have  $\mu (t_{n_k})\rightarrow \infty $. Along this subsequence we use the H\"older's  inequality and (\ref{89w3}) to estimate (recall that  $\varphi _R(x)=R^2\varphi (\frac{x}{R})$)
	\begin{eqnarray}
		V_R(t_{n_k})&=& \int _{|x|\le \varepsilon }\varphi _R(x)|u(t_{n_k},x)|^2dx+\int _{|x|> \varepsilon }\varphi _R(x)|u(t_{n_k},x)|^2dx\notag\\
		&\lesssim & R^2\varepsilon ^2 \|u(t_{n_k},x)\|_{L_x^{\frac{2d}{d-2}}}^2+R^4 \|u(t_{n_k},x)\|_{L_x^{\frac{2d}{d-2}}(|x|>\varepsilon )}^2\notag\\
		&\lesssim & R^2 \varepsilon ^2( \mathbf{d} (u(t_{n_k}))+ \|\nabla W\|_{L^2}^2 )+ R^4\|(u(t_{n_k}))_{[\theta (t_{n_k}),\mu (t_{n_k})]}\|_{L^{\frac{2d}{d-2}}(|x|\ge \varepsilon \mu (t_{n_k}))}^2\notag\\
		&\lesssim & R^2\varepsilon ^2( \mathbf{d} (u(t_{n_k}))+ \|\nabla W\|_{L^2}^2 )+R^4 \|W\|_{L^{\frac{2d}{d-2}}(|x|\ge \varepsilon \mu (t_{n_k}))}^2+R^4 \|v(t_{n_k})\|_{L^{\frac{2d}{d-2}}}^2.\notag
	\end{eqnarray} 
	Note that  $ \|v(t_{n_k})\|_{\dot H^1} \lesssim   \mathbf{d} (u(t_{n_k}))$ by (\ref{89w1}) and  $ \mathbf{d} (u(t_{n_k}))\rightarrow 0$.  
	Taking  $n_k\rightarrow\infty $ and  then  $\varepsilon \rightarrow 0$,    we obtain  $\lim_{n_k\rightarrow \infty }V_R(t_{n_k})=0$, which contradicts  (\ref{126w1}).  
	
	Next, we prove that  
	\begin{equation}
		\lim_{t\rightarrow \infty }  \mathbf{d} (u(t))=0.\label{124w3}
	\end{equation}    
	We argue by contradiction. If this is not true, there must exists  $c\in (0,\delta _0)$, a subsequence in  $\left\{t_n \right\}$  (for which we use the same notation) and another sequence  $\tau_n$ such that 
	\begin{equation}
		\label{126x1} \tau_n \in (t_n,t_{n+1}], \  \mathbf{d} (u(\tau_n))=c,\  \mathbf{d} (u(t))\in (0,c] \qquad \forall t\in [t_n,\tau_n].
	\end{equation}   
	Taking any  $t\in [t_n,\tau_n]$, we use the derivative estimate from Lemma \ref{L:modulation} and  (\ref{E:mu}) to obtain 
	\begin{equation}
		\left|\frac{1}{\mu (t_n)^2}-\frac{1}{\mu (t)^2}\right|\lesssim  \int _{t_n}^t|\frac{\mu' (t)}{\mu (t)^3}|dt\lesssim  \int _{t_n}^\infty  \mathbf{d} (u(t))dt\rightarrow0.\notag
	\end{equation}
	This together with the control from (\ref{E:mu}) and (\ref{1261}) implies 
	\begin{equation}
		\mu (t)\approx 1\qquad \forall t\in [t_n,\tau_n].\notag
	\end{equation}
	Inserting this into the estimate of  $\beta  (t)$ in (\ref{89w1}) we have 
	\begin{equation}
		|\beta (t_n)-\beta(\tau_n)|\le \int _{t_n}^{\tau_n} |\beta  '(t)|dt\lesssim  \int _{t_n}^{\tau_n}\frac{|\beta  '(t)|}{\mu ^2(t)}\lesssim  \int _{t_n}^\infty  \mathbf{d} (u(t))dt\rightarrow 0\notag
	\end{equation} 
	as  $n\rightarrow \infty $. This is a contradiction since $\beta  (t_n)\approx  \mathbf{d} (u(t_n))\rightarrow0$ and   $\beta  (\tau_n)\approx  \mathbf{d} (u(\tau_n))=c$ from (\ref{126x1}). The convergence of  $ \mathbf{d} (u(t))$ in (\ref{124w3}) is proved. 
	
	Given (\ref{124w3}), we can perform the decomposition for all  $t\ge T_0$ and repeat the same argument as that used to derive   (\ref{1261}) to show that  $\mu (t)\approx 1$. Finally, the exponential convergence of all the parameters follows from the same argument in the Section \ref{S:subcritical}.    
\end{proof}
\begin{proof}[\textbf{Proof of Corollary \ref{c:210}}.]
	 Let  $u$  be a solution of  (\ref{nls}) satisfying the assumptions of the corollary and defined on  $\mathbb{R} $.  Applying the arguments in the proof of  Proposition \ref{p3} to  $\overline{u}(-t)$, we know that (\ref{ineq2}) and (\ref{126w1}) also hold for the negative time. Moreover, we have 
	 \begin{equation}
	 	\lim_{t\rightarrow-\infty } \mathbf{d} (u(t))=0.\label{210x1}
	 \end{equation} 
	 By (\ref{11291}), (\ref{ineq2}) and (\ref{210x1}), we have that  $\partial_{tt}V_R(t)<0$ and  $\partial_{t}V_R(t)\rightarrow0$ as  $t\rightarrow\pm \infty $. This contradicts (\ref{126w1}) and completes the proof of Corollary \ref{c:210}.    
\end{proof}

\section{Uniqueness and the classification result}\label{s:8}
 In this section, we first follows the arguments in \cite{2008GFA} to establish a uniqueness result for threshold solutions converging to the ground state.  Then we  use the uniqueness results  to classify all threshold solutions for the energy critical inhomogeneous NLS (\ref{nls}), which will imply  the proof of   Theorem \ref{T2}. 
\subsection{Estimates on exponential solutions of the linearized equation.}
Let us consider the linearized equation with 
\begin{equation}
	\partial_{t}h+\mathcal{L}h=\varepsilon \label{Linear Eq2}
\end{equation}
where  $h$ and  $\varepsilon $  satisfy,  for  $t\ge 0$ 
\begin{equation}
	\|\varepsilon (t)\|_{L^{\frac{2d}{d+2},2}}+\|\nabla \varepsilon \|_{N(t,+\infty )}\lesssim  e^{-c_1t},\label{883}
\end{equation}
\begin{equation}
	\|h(t)\|_{\dot H^{1}}\lesssim  e^{-c_0t},\label{884}
\end{equation}
with  $0<c_0<c_1$. The following proposition asserts that 	 $h$ must decay
almost as fast as  $\varepsilon $, except in the direction  $\mathcal{Y}_+$ where the decay is of order  $e^{-e_0t}$. 
\begin{proposition}\label{P:decay of h}
	Consider  $h$ and  $\varepsilon $ satisfying (\ref{Linear Eq2}), (\ref{883}) and (\ref{884}). Then for any Strichartz couple  $(p,q)$
	\begin{itemize}
		\item if   $e_0\notin [c_0,c_1)$ , then for any   $\varepsilon >0$ 
		\begin{equation}
			\|h(t)\|_{\dot H^{1}} +\|\nabla h\|_{L^p(t,+\infty ;L^{q,2})}\le C_\eta e^{-(c_1-\eta)t};\notag
		\end{equation}
		\item if  $e_0\in [c_0,c_1)$, there exists  $A_+\in \mathbb{R}$ such that for any  $\eta>0$ 
		\begin{equation}
			\|h(t)-A_+e^{-e_0t}\mathcal{Y}_+\|_{\dot H^{1}}+\|\nabla (h-A_+e^{-e_0t}\mathcal{Y}_+)\|_{L^p(t,+\infty ;L^{q,2})}\le C_\eta e^{-(c_1-\eta)t}.\notag
		\end{equation}
	\end{itemize}
\end{proposition}
\begin{proof}
	As the proof follows that of \cite[Proposition 5.9]{2008GFA}, we will only sketch the main steps. By Lemma \ref{L:small solution}, it suffices to show that if   $e_0\notin [c_0,c_1)$ then 
	\begin{equation}
  \|h(t)\|_{\dot H^1}\le C_\eta e^{-(c_1-\eta)t};\label{33w1} 
	\end{equation}
	and if  $e_0\in [c_0,c_1)$ then 
	\begin{equation}
			\|h(t)-A_+e^{-e_0t}\mathcal{Y}_+\|_{\dot H^{1}}\le C_\eta e^{-(c_1-\eta)t}.\label{33w2} 
	\end{equation} 
	 In the sequel, we will assume without loss of generality   that $c_1\neq e_0$.  
	
	\noindent\textbf{Step 1:} Let us decompose  $h(t)$ as 
	\begin{equation}
		h(t)=\alpha _+(t)\mathcal{Y}_++\alpha _-(t)\mathcal{Y}_-+\widetilde{\alpha }(t)iW+\gamma (t)W_1+g(t),\label{331}
	\end{equation} 
	where 
	\begin{align}
		& \alpha _-:=\frac{B(h,\mathcal{Y}_+)}{B(\mathcal{Y}_+,\mathcal{Y}_-)},\qquad \alpha _+:=\frac{B(h,\mathcal{Y}_-)}{B(\mathcal{Y}_+,\mathcal{Y}_-)},\label{332}\\
		&\widetilde{\alpha }:=\frac{1}{ \|W\|_{\dot H^1}^2 }(h-\alpha _+\mathcal{Y}_+-\alpha _-\mathcal{Y}_-,iW)_{\dot H^1},\notag\\
		&\gamma :=\frac{1}{ \|W_1\|_{\dot H^1}^2 }(h-\alpha _+\mathcal{Y}_+-\alpha _-\mathcal{Y}_-,W_1)_{\dot H^1},\label{333}
	\end{align}
	so that by (\ref{12241}) and (\ref{12243})  $g\in G^\bot$.  Using equation (\ref{Linear Eq2}) and  (\ref{12242}),  we obtain the differential equations on the coefficients in (\ref{332})--(\ref{333}): 
\begin{align}
		\frac{d}{dt}(e^{e_0t}\alpha _+)&=e^{e_0t}\frac{B(\mathcal{Y}_-,\varepsilon )}{B(\mathcal{Y}_+,\mathcal{Y}_-)},\qquad \frac{d}{dt}(e^{-e_0t}\alpha _-)=e^{-e_0t}\frac{B(\mathcal{Y}_+,\varepsilon )}{B(\mathcal{Y}_+,\mathcal{Y}_-)},\label{334}\\
	\frac{dQ(h)}{dt}&=2B(h,\varepsilon ),\qquad \frac{d\widetilde{\alpha }}{dt}=\frac{(iW,\widetilde{\varepsilon })_{\dot H^1}}{ \|W\|_{\dot H^1}^2 },\qquad \frac{d\gamma }{dt}=\frac{(W_1,\widetilde{\varepsilon })_{\dot H^1}}{ \|W_1\|_{\dot H^1}^2 },\label{335}
\end{align}
	where 
	\begin{equation}
		\widetilde{\varepsilon }:=\varepsilon -\frac{B(\mathcal{Y}_-,\varepsilon )}{B(\mathcal{Y}_+,\mathcal{Y}_-)}\mathcal{Y}_+-\frac{B(\mathcal{Y}_+,\varepsilon )}{B(\mathcal{Y}_+,\mathcal{Y}_-)}\mathcal{Y}_--\mathcal{L}g.\label{33x2}
	\end{equation}
	
	\noindent\textbf{Step 2:}  Bounds on  $\alpha _-$ and  $\alpha _+$. We now claim   
	\begin{align}
		|\alpha _-(t)|&\lesssim  e^{-c_1t}, \label{336}\\
			|\alpha _+(t)|&\lesssim  \begin{cases}
			e^{-c_1t},\qquad \qquad&\text{ if } e_0\notin [c_0,c_1),\\
			e^{-c_1t}+e^{-e_0t},\qquad &\text{ if } e_0\in [c_0,c_1).
		\end{cases}\label{337}
	\end{align}
	Let us first show the following general bound on  $B$. 
	\begin{claim} \label{C:33}
		For any finite time-interval  $I$  of length  $|I|\le 1$, we have 
		\begin{equation}
				\int _I|B(f,g)|dt\lesssim |I|^{\frac{1}{2}}(\|\nabla f\|_{N(I)}+ \| f\|_{L^\infty (I,L^{\frac{2d}{d+2},2})} ) \|g\|_{L^\infty (I,H^2)}.\label{33w8} 
		\end{equation}
	\end{claim} 
	\begin{proof}
	Recall the definition of 	 the symmetric bilinear form  $B$ in (\ref{B}). By H\"older's inequality  
\begin{eqnarray}
		\int _I|B(f,g)|dt&\lesssim&  \int_I \int _{\mathbb{R} ^d}\left(|\nabla f||\nabla g|+|x|^{-b}  |f||g|\right)dxdt \notag\\
		&\lesssim   &\|\nabla f\|_{N(I)}  \|\nabla g\|_{L^2(I,L^{\frac{2d}{d-2},2})} + |I|\| f\|_{L^\infty (I,L^{\frac{2d}{d+2},2})} \|g\|_{L^\infty (I,L^{\frac{2d}{d-2-2b},2})}  ,	\notag
\end{eqnarray}
which together with the embedding  $H^2(\mathbb{R} ^d)\hookrightarrow L^{\frac{2d}{d-2-2b},2} (\mathbb{R} ^d)$ yields (\ref{33w8}).   
	\end{proof}
	Assumption (\ref{883}) on  $\varepsilon $, together with (\ref{334}) and Claim \ref{C:33}  yields  
	\begin{equation}
		e^{-e_0t}|\alpha _-(t)| =\int _t^\infty e^{-e_0s}|B(\mathcal{Y}_+,\varepsilon (s))|ds\lesssim  e^{-(e_0+c_1)t}.\notag
	\end{equation}
	This proves  (\ref{336}). 
	
Let us show (\ref{337}).  First assume that   $e_0<c_0$.  Thus by  assumption (\ref{884}) and (\ref{332}),  $e^{e_0t}\alpha _+(t)$ tends to  $0$ when  $t$ tends to infinity. Then using the    same argument as that used to derive (\ref{336}), we obtain 
  $|\alpha _+(t)|\lesssim  e^{-c_1t}$. 
  
  Now assume that  $  e_0\ge c_0$.  By (\ref{334})
	\begin{equation}
		\alpha _+(t)=e^{-e_0t}\alpha _+(0)+\frac{e^{-e_0t}}{B(\mathcal{Y}_+,\mathcal{Y}_-)} \int _0^te^{e_0s}B(\mathcal{Y}_-,\varepsilon (s))ds.\notag
	\end{equation}   
 	Assumption (\ref{883}) on  $\varepsilon $ together with   Claim \ref{C:33}  yields  
\begin{equation}
	|\alpha _+(t)|\lesssim  e^{-e_0t}+e^{-e_0t}  \int_0^t e^{e_0s}e^{-c_1s}ds\lesssim  e^{-e_0t}+e^{-c_1t}.\notag
\end{equation}
Estimate (\ref{337}) is proved.  
 
 \noindent\textbf{Step 3:}  Bounds on  $ \|g\|_{\dot H^1}, \widetilde{\alpha } $ and  $\gamma $.   We next prove 
 \begin{equation}
 	 \|g(t)\|_{\dot H^1}+|\widetilde{\alpha }(t)|+|\gamma (t)|\lesssim  e^{-(\frac{c_0+c_1}{2})t}.\label{33x1}
 \end{equation}
 Integrating the equation on  $Q$ in (\ref{335}) between  $t$ and  $+\infty $, and using Claim \ref{C:33}, assumptions (\ref{883}) and (\ref{884}), we get 
 $	|Q(h(t))|\lesssim  e^{-(c_0+c_1)t}.$ 
 Thus 
  \begin{align}
   	|Q(\alpha _+\mathcal{Y}_++\alpha _-\mathcal{Y}_-+\widetilde{\alpha }iW+\gamma W_1+g)|\lesssim  e^{-(c_0+c_1)t}\notag\\
   	|2\alpha _+\alpha _-B(\mathcal{Y}_+,\mathcal{Y}_-)+Q(g)|\lesssim  e^{-(c_0+c_1)t},\notag
  \end{align}
 which together with the bounds (\ref{336}) and (\ref{337}) implies  $|Q(g)|\lesssim  e^{-(c_0+c_1)t}$. As a consequence of the coercivity of  $Q$ on  $G^\bot$ (Lemma \ref{L:G}), we get the estimate on  $ \|g\|_ {\dot H^1}$ in   (\ref{33x1}). It remains to show the bounds on  $\widetilde{\alpha }$ and  $\gamma $.

 Consider the function   $\widetilde{\varepsilon }$ defined in  (\ref{33x2}). By assumption  (\ref{883}) and the equation (\ref{Eq:W})
 \begin{equation}
 	|(iW,\widetilde{\varepsilon }(s))_{\dot H^1}|\lesssim  e^{-c_1t}+\left|(W,\mathcal{L}g(s))_{\dot H^1}\right|  \lesssim  e^{-c_1t}+|\int |x|^{-b}W^{\alpha +1}\mathcal{L}g|.\notag
 \end{equation}
 Since  $|\nabla W|\lesssim  |x|^{-1}W$ and  $W=O(\langle x \rangle ^{-(d-2)})   $, it follows from H\"older's inequality that 
\begin{align}
	|\int |x|^{-b}W^{\alpha +1}\mathcal{L}g|&  \lesssim  \int |x|^{-b-1}W^{\alpha +1} |\nabla g|+|x|^{-2b}W^{2\alpha +1}|g|\notag\\
		&\lesssim  \int |x|^{-b-1}\langle x \rangle ^{-(d+2-2b)}|\nabla g|+\int |x|^{-2b}\langle x \rangle ^{-(d+6-4b)}|g|\notag \\
		&\lesssim   \|\nabla g\|_{L^2}+ \|g\|_{L^{\frac{2d}{d-2}}}\lesssim  \|g\|_{\dot H^1} .\notag   
\end{align}
 Thus  by the estimate of  $g$ in (\ref{33x1}), we have   $	|(iW,\widetilde{\varepsilon }(s))_{\dot H^1}|\lesssim  e^{-\frac{c_0+c_1}{2}t} $.   In view of  the second equation in (\ref{335}) and (\ref{33x2}),   we get the bound on  $\widetilde{\alpha }$  in (\ref{33x1}).  An analogous proof yields the bound on  $\gamma $.

 \noindent\textbf{Step 4:} Conclusion. Summing up estimates (\ref{336}), (\ref{337}) and (\ref{33x1}), we get, in view of decomposition (\ref{331}) of  $h$:
 \begin{equation}
 	 \|h(t)\|_{\dot H^1}\lesssim \begin{cases}
 	 	e^{-\frac{c_0+c_1}{2}t},\qquad\qquad\qquad &\text{if }e_0\notin [c_0,c_1)\\
 	 	e^{-e_0t}+ e^{-\frac{c_0+c_1}{2}t}&\text{if } e_0\in [c_0,c_1).
 	 \end{cases} \label{33w3}
 \end{equation} 
\noindent  \textbf{Proof of (\ref{33w1}):} If  $e_0\notin [c_0,c_1)$, then by the first line in (\ref{33w3}), the estimate of  $h$ in (\ref{884}) can be improved.    Iterating the argument we obtain the bound  $ \|h(t)\|_{\dot H^1} \le C_\eta e^{-(c_1-\eta)t}$ if  $e_0\notin [c_0,c_1)$, which yields the desired estimate (\ref{33w1}).    \\
\textbf{Proof of (\ref{33w2}):} Let us assume  $e_0\in [c_0,c_1)$.  Then the equation on  $\alpha _+$ in (\ref{332}) shows that  $e^{e_0t}\alpha _+(t)$ has a limit  $A_+$  when  $t\rightarrow +\infty $. Integrating the equation  on  $\alpha _+$ between  $t$ and  $+\infty $, we get (in view of Claim \ref{C:33}) 
\begin{equation}
	A_+-e^{e_0t}\alpha _+(t)=e^{e_0t} \int _t^{+\infty } \frac{B(\mathcal{Y}_+,\varepsilon (s))}{B(\mathcal{Y}_+,\mathcal{Y}_-)}ds=O(e^{(e_0-c_1)t}).\label{33w5}
\end{equation}       
Substituting   (\ref{33w5}) into the decomposition (\ref{331}),  and using the estimates (\ref{336}) and (\ref{33x1}),  we get 
 $ \|h(t)-A_+e^{-e_0t}\mathcal{Y}_+\|_{\dot H^1}\lesssim  e^{-\frac{c_0+c_1}{2}t} $. Furthermore,  $h_1(t):=h(t)-A_+e^{-e_0t}\mathcal{Y}_+$ satisfies, as  $h$, equation (\ref{Linear Eq2}). Thus the estimate (\ref{33w1})  shown in the preceding step implies (\ref{33w2}). The proof of  Proposition \ref{P:decay of h} is complete.    
\end{proof}

\subsection{Uniqueness}
\begin{lemma}\label{L:s61}
	If  $u$ is a solution of (\ref{nls}), defined on  $[t_0,+\infty )$, satisfies   $E(u)=E(W)$ and  
	\begin{equation}
		\|u(t)-W\|_{\dot H^{1}}\lesssim  e^{-ct}, \qquad \forall t\ge t_0,\label{810x1}
	\end{equation}
	for some constant  $c>0$. 
	Then  there exists a unique  $a\in \mathbb{R}$ such that  $u=W^a$, where  $W^a$ is constructed  in  Proposition \ref{P:approximate}.  
\end{lemma}
\begin{proof} 
	Let  $v:=u-W$.  Then by (\ref{Linear Eq}) $v$ satisfies the equation 
	\begin{equation}
		\partial_{t}v+\mathcal{L}v+R(v)=0,\qquad \forall t\ge t_0. \label{361}
	\end{equation}
	\noindent\textbf{Step 1:} We show that there exists  $a\in \mathbb{R} $ such that for any $\eta>0$,
	\begin{equation}
		\|v(t)-ae^{-e_0t}\mathcal{Y}_+\|_{\dot H^1}+ \| v(t)-ae^{-e_0t}\mathcal{Y}_+\|_{Z(t,+\infty )}\le C_\eta e^{-(2-\eta)e_0t}.\label{1127w1}  
	\end{equation}   
	Indeed, we will show 
	\begin{equation}
		\|v(t)\|_{\dot H^1}\lesssim e^{-e_0t},\qquad  \|R(v(t))\|  _{L^{\frac{2d}{d+2},2}}+ \|\nabla (R(v))\| _{N(t,+\infty )}\lesssim e^{-2e_0t},\label{1127w2}
	\end{equation}
	which together with  Proposition \ref{P:decay of h} gives (\ref{1127w1}). 
	By Lemma \ref{nestimate},  it suffices to show the first estimate. 
	
	By (\ref{810x1}) and Lemma \ref{nestimate},  we have 
	\begin{equation}
		\|R(v(t))\| _{L^{\frac{2d}{d+2},2}}+ \|\nabla R(v(t))\| _{N(t,+\infty )}\lesssim e^{-2ct}.\notag
	\end{equation}
	Then  Proposition \ref{P:decay of h} gives that 
	\begin{equation}
		\|v(t)\| _{\dot H^1}\lesssim e^{-e_0t}+e^{-\frac{3}{2}ct}.\notag
	\end{equation}
	If  $\frac{3}{2}c\ge e_0$, we obtain the first inequality in (\ref{1127w2}). If not, an iteration argument gives the first inequality in (\ref{1127w2}). 
	
	\noindent\textbf{Step 2:} Let us show,  for any  $m>0$,  
	\begin{equation}
		\|u(t)-W^a\|_{\dot H^1}+ \|u-W^a\|_{Z(t,+\infty )}\le e^{-m t}.\label{1127w3}  
	\end{equation} 
	This implies that  $u=W^a$, by uniqueness in  Proposition \ref{P:approximate}.  Therefore, the proof of Lemma \ref{L:s61} reduces to show (\ref{1127w3}).
	
	We now prove (\ref{1127w3}). According to Step 1, (\ref{1127w3}) holds for  $m=\frac{3}{2}e_0$. Let us assume (\ref{1127w3}) holds for some  $m=m_1>e_0$. We will show that it holds for  $m=m_1+\frac{e_0}{2}$, which will yield (\ref{1127w3}) by iteration. 
	
	Let   $v^a(t):=W^a(t)-W$. Then by (\ref{Linear Eq}) and (\ref{361})
	\begin{equation}
		\partial_{t}(v-v^a)+\mathcal{L}(v-v^a)=-R(v)+R(v^a).\notag
	\end{equation}    
	We have assumed that (\ref{1127w3}) holds for  $m=m_1$, i.e.  
\begin{equation}
	\|v(t)-v^a(t)\|_{\dot H^1}+ \|v-v^a\|_{Z(t,+\infty )}\le e^{-m_1t}  ,\notag
\end{equation}
which together with  Lemma \ref{nestimate} implies 
	\begin{equation}
		\|R(v(t))-R(v^a(t))\| _{L^{\frac{2d}{d+2},2}}+ \|\nabla (R(v)-R(v^a))\|_{N(t,+\infty )}\lesssim  e^{-(m_1+e_0)t}.\notag 
	\end{equation} 
It then follows from Proposition  \ref{P:decay of h} that 
	\begin{equation}
		\|v(t)-v^a(t)\| _{\dot H^1}+ \|v-v^a\|_{Z(t,+\infty )}\lesssim  e^{-(m_1+\frac{3}{4}e_0)t},\notag 
	\end{equation}
	which yields (\ref{1127w3}) with  $m=m_1+\frac{e_0}{2}$. By iteration, (\ref{1127w3}) holds for any  $m>0$. 
\end{proof}

\begin{corollary}\label{cor1}
	For any  $a\neq 0$, there exists  $T_a\in \mathbb{R}$ such that 
	\begin{equation}
		\begin{cases}
			W^a(t)=W^+(t+T_a) \quad\text{if}\quad  a>0,\\
			W^a(t)=W^-(t+T_a)\quad\text{if}\quad a<0. 
		\end{cases}\label{810x5}
	\end{equation}
\end{corollary}
\begin{proof}
	Let  $a>0$ and then choose  $T_a>0$ such that  $ae^{-e_0T_a}=1$. By (\ref{88w4}), for large  $t$,  
	\begin{equation}
		\|W^a(t+T_a)-W-e^{-e_0t}\mathcal{Y}_+\|_{\dot H^{1}}\lesssim e^{-\frac{3}{2}e_0t},\label{810x4}
	\end{equation}
	which implies  $\|W^a(t+T_a)-W\|_{\dot H^{1}}\lesssim e^{-e_0t}$.  By  Lemma \ref{L:s61},  there exists  $a'\in \mathbb{R} $ such that  $W^a(t+T_a)=W^{a'}(t)$.  Substituting it  into (\ref{810x4}), using (\ref{88w4}) and the uniqueness result  of  $W^a$ in  Proposition \ref{P:approximate},    we see that   $a'=1$. Hence  (\ref{810x5}) holds when  $a>0$.  The proof for  $a<0$ is similar and we omit the details. 
\end{proof}

\subsection{Proof of the classification result.}
Let us turn to the proof of Theorem \ref{T2}. Point  (b) is an immediate consequence of the variational characterization of  $W$ (Proposition 
\ref{P:GN}). 

Let us show (a). Let  $u$ be a solution of (\ref{nls})  satisfying (\ref{871}), and  $I=(T_-,T_+)$ be its maximal interval of existence.     If  $ \|u\|_{S(T_-,T_+)} <\infty $, then  $u$ exists globally and  scatters in  both time directions by Theorem \ref{T:CP}.   Assume now that  $\|u\|_{S(T_-,T_+)}=\infty $. Replacing if necessary  $u(t)$ by  $\overline{u}(t)$, we may assume that  $\|u\|_{S(0,T_+ )}=\infty $. By Proposition  \ref{p2},  $T_+=+\infty $ and    there exist  $\theta _0\in \mathbb{R}$,  $\mu_0>0$ and  $c>0$ such that  $\|u(t)-W_{[\theta _0,\mu_0]}\|_{\dot H^{1}}\lesssim e^{-ct}$. It then follows from Lemma \ref{L:s61} that there exists  $a<0$ such that  $u_{[-\theta _0,\mu_0 ^{-1}]}=W^a$. Thus by Corollary \ref{cor1}, 
\begin{equation}
	u(t)=W^-_{[\theta _0,\mu_0]}(t+T_a)\notag
\end{equation}
for some  $T_a\in \mathbb{R} $,  which shows (a). 

The proof of (c) is similar. Let    $u$ be  a solution of (\ref{nls}) defined on  $I$  such that  $E(u_0)=E(W)$,  $\|u_0\|_{\dot H^{1}}>\|\nabla W\|_{\dot H^{1}}$ and  $u_0\in L^2$ is radial. Assume that  $|I|$ is infinity. Without loss of generality, we may assume that  $u$ is defined on      $[0,+\infty  ) $. Then by Proposition  \ref{p3},  $\|u(t)-W_{[\theta _0,\mu_0]}\|_{\dot H^{1}}\lesssim e^{-ct}$.   Using Lemma \ref{L:s61} and the same argument as before, we have  
\begin{equation}
	u(t)=W^+_{[\theta _0,\mu_0]}(t+T_a)\notag
\end{equation}
for some  $T_a\in \mathbb{R}$, which shows (c).  The proof of Theorem \ref{T2} is complete.

\appendix

\section{Asymptotic behavior of  $G(r)$.}\label{App:1}
\begin{lemma}
	Let  $W$ be the ground state in (\ref{W}) and  $G(x)=G(|x|)\in \dot H^1_{\text{rad}}(\mathbb{R} ^d)$ solving 
	\begin{equation}
		G''+\frac{d-1}{r}G'-\frac{d-1}{r^2}G+\frac{\alpha +1}{r^b}W^\alpha G=0.\label{ODE:G}
	\end{equation}  
	Then 
	\begin{equation}
		\begin{cases}
			\text{As }r\rightarrow0^+,G(r)=O(r),\ G'(r)=O(1),\\
			\text{As }r\rightarrow\infty ,G(r)=O(r^{-(d-1)}),\ G'(r)=O(r^{-d}), 
		\end{cases}\notag
	\end{equation}  
\end{lemma}
\begin{proof}
	It is easy to see  $0$  is the regular-singular point of  the  ODE (\ref{ODE:G}); therefore there must exist two linear independent solutions in the form of a power series: 
	\begin{equation}
		G_1(r)=r\sum_{n=0}^{\infty }a_nr^n,\qquad a_0=1;\notag
	\end{equation}
	\begin{equation}
		G_2(r)=CG_1(r)\ln r+r^{-(d-1)}\sum_{n=0}^{\infty } b_nr^n,\qquad b_0=1.\notag
	\end{equation}
	General solutions to (\ref{ODE:G}) near  $0$  are
	\begin{equation}
		c_1G_1(s)+c_2G_2(s).\notag
	\end{equation}
	Since  $G(x) = G(|x|)\in \dot H^1_{\text{rad}}(\mathbb{R} ^d) $, clearly it must hold that  $c_2=0$  and
	thus we obtain the desired asymptotics of $G$ near  $0$.

	For the asymptotic behavior near infinity, we can reduce the issue into a similar situation by introducing the change of variable  $s=r^{-1}$. Equation (\ref{ODE:G}) in variable  $s$ is 
	\begin{equation}
		G_{ss}-\frac{d-3}{s}G_s-\frac{d-1}{s^2}G+(\alpha +1)s^{b-4}W^\alpha G=0.\notag
	\end{equation}  
	From a similar analysis, it has two linear independent solutions near  $s=0$.  Going back to  $r$  variable and using  $G(r)\in \dot H^1_{\text{rad}}(\mathbb{R} ^d)$ , we are able to select the right 
	asymptotics 
	\begin{equation}
		G(r)=O(r^{-(d-1)})\notag
	\end{equation}
	as $r\rightarrow \infty$.   The lemma is proved.
\end{proof}

\section{Proof of modulation results.}\label{App:2}
\begin{proof}[\textbf{Proof of Lemma \ref{L:modulation}.}]
	We first prove (\ref{89w1}).  	By  $v(t)=\beta  (t)W+\widetilde{u}(t)$ and  $\widetilde{u}\bot W$ in  $\dot H^1$ , we have  
	\begin{equation}
		\|v\|_{\dot H^{1}}^2=\beta  ^2\|W\|_{\dot H^{1}}^2+\|\widetilde{u}\|_{\dot H^{1}}^2.\label{89w4}
	\end{equation}
	Denote by   $\widetilde{u}_1$ and  $\widetilde{u}_2$ the real and
	imaginary parts of  $\widetilde{u}$. By the orthogonality of  $\widetilde{u}_1$ and  $\widetilde{u}_2$ with  $W$ in  $\dot H^{1}$
	and the equation  (\ref{Eq:W}),  we have
	\begin{equation}
		\int \nabla W\cdot \nabla \widetilde{u}_1=\int \nabla  W\cdot\nabla \widetilde{u}_2=\int |x|^{-b}W^{\alpha +1}\widetilde{u}_1=\int |x|^{-b}W^{\alpha +1}\widetilde{u}_2=0.\notag
	\end{equation}
	Hence $B(W,\widetilde{u})=0$ and  
	\begin{equation}
		Q(v)=Q(\widetilde{u}+\beta  W)=Q(\widetilde{u})+\beta ^2Q(W).\label{2212}
	\end{equation}
	From the scaling invariance of the energy, (\ref{89w3}) and (\ref{2211}), we have 
	\begin{equation}
		E(W)=E(u_{[\theta,\mu]})=E(W+v)=E(W)+Q(v)+O(\left\|v\right\|_{\dot H^1}^3).\label{2213}
	\end{equation}
	  As  $Q(W)<0$ by  (\ref{QW}), it follows from (\ref{2212}) and (\ref{2213}) that     
	  \begin{equation}
	  	|\beta  ^2|Q(W)|-Q(\widetilde{u})|=Q(v)\lesssim  \|v\|_{\dot H^{1}}^3. \notag
	  \end{equation}
	   This inequality together with  the coercivity property
	of  $Q$ in  Proposition \ref{H} implies that 
	\begin{equation}
		\|\widetilde{u}\|^2_{\dot H^{1}}\lesssim \|v\|_{\dot H^{1}}^3+\beta  ^2 \quad\text{and}\quad  \beta  ^2\lesssim  \|\widetilde{u}\|_{\dot H^{1}}^2+\|v\|_{\dot H^{1}}^3.\label{89w5}
	\end{equation}
Since $\|v\|_{\dot H^{1}}$ is small when  $ \mathbf{d} (u)$ is small by the variational characterization of  $W$, it follows from   (\ref{89w4}) and (\ref{89w5}) that 
	\begin{equation}
		|\beta  |\approx \|v\|_{\dot H^{1}}\approx \|\widetilde{u}\|_{\dot H^{1}},\notag
	\end{equation}
	for small  $ \mathbf{d} (u)$.
	This is the first part of (\ref{89w1}).  It remains to show the estimate on  $ \mathbf{d} (u)$.
	Developing the equation  $\left|\|W+v\|_{\dot H^{1}}^2-\|W\|_{\dot H^{1}}^2\right| =\mathbf{d} (u) $, we get
	\begin{equation}
		\left|\|v\|_{\dot H^{1}}^2+2(v,W)_{\dot H^{1}}\right|=\left|\|v\|_{\dot H^{1}}^2+2\beta\left\|W\right\|_{\dot H^1}^2\right| = \mathbf{d} (u),\notag
	\end{equation}
	which yields  $|\beta  (t)|\approx  \mathbf{d} (u(t))$.  Estimates (\ref{89w1}) are proved. 
	
	We now prove (\ref{89w2}).  Let us consider the self-similar variables  $y$ and  $s$ defined by
	\begin{equation}
		\mu(t)y=x,\qquad ds=\mu^2(t)dt.\notag
	\end{equation}
	Then (\ref{nls}) may be rewritten  as (for simplicity we drop the  $t$ dependence in  $\theta ,\mu$):
	\begin{eqnarray}
		i\partial_{s}u_{[\theta ,\mu]}+\Delta _yu_{[\theta ,\mu]}+|y|^{-b}|u_{[\theta ,\mu]}|^\alpha u_{[\theta ,\mu]}+\theta _su_{[\theta ,\mu]}+i\frac{\mu_s}{\mu}(\frac{d-2}{2}u_{[\theta ,\mu]}+y\cdot \nabla u_{[\theta ,\mu]})=0.\notag
	\end{eqnarray}
	Inserting  $u_{[\theta ,\mu]}=W+v$,  we get 
	\begin{eqnarray}
	&&	\partial_{s}v-i\Delta v-i(\alpha +1)|y|^{-b}W^\alpha v_1+|y|^{-b}W^\alpha v_2+R(v)\notag\\
	&&\qquad -i\theta _s(W+v)+\frac{\mu_s}{\mu}W_1+\frac{\mu_s}{\mu}(y\cdot\nabla v+\frac{d-2}{2}v)=0.\notag
	\end{eqnarray}
	Writing  $v=\beta  (s)W+\widetilde{u}$,  		we obtain the equation for  $\widetilde{u}=\widetilde{u}_1+i\widetilde{u}_2$:   
	\begin{eqnarray}
		&&\partial_{s}\widetilde{u}_1+i\partial_{s}\widetilde{u}_2+\beta _sW+(\Delta +|y|^{-b}W^\alpha )\widetilde{u}_2-i(\Delta +(\alpha +1)|y|^{-b}W^\alpha )\widetilde{u}_1-\alpha i\beta (s) |y|^{-b}W^{\alpha +1}\notag\\
		&&\quad -\theta _siW+\frac{\mu_s}{\mu}W_1=-R(v)+\theta _siv-\frac{\mu_s}{\mu}(\frac{d-2}{2}v+y\cdot \nabla v).\label{8102} 
	\end{eqnarray}
	\begin{claim}\label{C:21}
		 The  $\dot H^{1}$-scalar products
		of the right-hand term by  $W,\ iW$ and  $W_1$ are bounded up to a constant
		by  $\mathcal{E} (s)$, where  $\mathcal{E} (s)$ is defined by 
		\begin{equation}
			\mathcal{E} (s):=| \mathbf{d} (u(s))|(| \mathbf{d} (u(s))|+|\theta _s(s)|+|\frac{\mu_s}{\mu}(s)|).\notag
		\end{equation}
		\end{claim}
		\begin{proof}[Proof of Claim \ref{C:21}.]
			We only show how to bound the  $\dot H^1$-scalar products of  $R(v)$ by  $W$, since the others can be handed similarly.   Note that by  (\ref{211}) and  (\ref{1281}), 
			\begin{equation}
				|R(v)|\lesssim  |y|^{-b}(W^{\alpha -1}|v|^2+|v|^{\alpha +1});\notag
			\end{equation}
			so that by the equation  (\ref{Eq:W}),  $0<b< \min \left\{ 1,\frac{d-2}{2} \right\}$ and the estimate (\ref{89w1}) that proved before 
 		\begin{eqnarray}
				&&|(R(v),W)_{\dot H^{1}}|=|(R(v),\Delta W)_{L^2}|  
				\lesssim  \int _{\mathbb{R} ^d}|y|^{-2b}(W^{\alpha -1}|v|^2+|v|^{\alpha +1})W^{\alpha +1}dy\notag\\
				&\lesssim &  \int _{\mathbb{R}^d} |y|^{-2b}{\langle y \rangle }^{-(8-4b)}|v|^2+|y|^{-2b}\langle y \rangle^{-(d+2-2b)}|v|^{\alpha +1}dy \notag\\
				&\lesssim&  \||y|^{-2b}{\langle y \rangle}^{-(8-4b)}\|_{L_y^{\frac{d}{2} }}\|v\|_{L^{\frac{2d}{d-2}}}^2+\||y|^{-2b}  \|_{L_y^{\frac{2d}{d-2+2b}}}\|v\|^{\alpha +1}_{L^{\frac{2d}{d-2}}}  \notag\\
				&\lesssim&    \|v(s)\|_{\dot H^1}^2+\|v(s)\|_{\dot H^1}^{\alpha +1} \lesssim   \mathbf{d} (u(s))^2+\mathbf{d} (u(s))^{\alpha+1} \lesssim \mathbf{d} (u(s))^2.\notag
		\end{eqnarray}
		\end{proof}

	Taking the inner product between (\ref{8102})   and  $W,iW$ and  $W_1$   in  $\dot H^{1}$, we obtain 
	\begin{equation}
		\beta _s\|W\|_{\dot H^{1}}^2=-(\Delta \widetilde{u}_2,W)_{\dot H^{1}}-(|y|^{-b}W^\alpha \widetilde{u}_2,W)_{\dot H^{1}}+O(\mathcal{E} (s))\label{814q1}
	\end{equation}
\begin{align}
		-\theta _s\|W\|_{\dot H^{1}}^2=(\Delta &\widetilde{u}_1,W)_{\dot H^{1}}+((\alpha +1)|y|^{-b}W^\alpha \widetilde{u}_1,W)_{\dot H^{1}}\notag\\
	&-\alpha \beta  (s)(|y|^{-b}W^{\alpha +1},W)_{\dot H^{1}}+O(\mathcal{E} (s))\notag
\end{align}
	\begin{equation}
		\frac{\mu _s}{\mu}\|W_1\|_{\dot H^{1}}^2=-(\Delta \widetilde{u}_2,W_1)_{\dot H^{1}}-(|y|^{-b}W^\alpha \widetilde{u}_2,W_1)_{\dot H^{1}}+O(\mathcal{E} (s)).\label{814q3}
	\end{equation}
	\begin{claim}\label{C:221}
		We have the following bounds: 
		\begin{equation}
			|(\Delta \widetilde{u},W)_{\dot H^1}|+|(|y|^{-b}W^\alpha \widetilde{u},W)_{\dot H^1}|\lesssim  \mathbf{d} (u(s)).\notag
		\end{equation}
		The same set of estimates also hold when  $W$ is replaced by  $W_1$.    
		\end{claim}
		\begin{proof}[Proof of Claim \ref{C:221}.]
		Since    $0<b<\min \left\{ 1,\frac{d-2}{2} \right\}$,  it follows from  the  equation (\ref{Eq:W}), H\"older's inequality and (\ref{89w1}) that 
			\begin{eqnarray}
				&&|(\Delta \widetilde{u},W)_{\dot H^1}|=|(\nabla \widetilde{u},\nabla (|y|^{-b}W^{\alpha +1}))_{L^2}|\notag\\
				&\lesssim & \|\nabla \widetilde{u}\|_{L^2} \||y|^{-b-1} {\langle y \rangle}^{-(d+2-2b)}\|_{L^2_y}\lesssim   \|\widetilde{u}\|_{\dot H^1}  \lesssim   \mathbf{d} (u(s)),\notag
			\end{eqnarray}
			and 
			\begin{eqnarray}
				&& |(|y|^{-b}W^\alpha \widetilde{u},W)_{\dot H^1}|=|(|y|^{-b}W^\alpha \widetilde{u},|y|^{-b}W^{\alpha +1})_{L^2}|\notag\\
				&\lesssim & \|\widetilde{u}\|_{L^{\frac{2d}{d-2}}}  \||y|^{-2b} {\langle y \rangle}^{-(d+6-4b)}\|_{L^{\frac{2d}{d+2}}_y}  \lesssim  \|\nabla \widetilde{u}\|_{L^2}  \lesssim  \mathbf{d} (u(s)).\notag
			\end{eqnarray}
			Finally, as  $-\Delta W_1=(\alpha +1)W^\alpha W_1$, $W_1$ is bounded as  $|y|\rightarrow0$ and    decays faster than  $W$ as  $|y|\rightarrow\infty $, we have the same set of estimates when  $W$ is replaced by  $W_1$.      
		\end{proof}
	Consequently all the right-hand terms in the equations (\ref{814q1})--(\ref{814q3}) are bounded up to a constant by  $ \mathbf{d} (u(s))+\mathcal{E}(s)$. Taking  $\delta _0$  sufficiently small, we have   
	\begin{equation}
		|\beta  _s(s)|+|\theta _s(s)|+|\frac{\mu_s}{\mu}(s)|\lesssim  | \mathbf{d} (u(s))|.\notag
	\end{equation}
	Changing back to  $t$ variable we proved (\ref{89w2}).  
\end{proof}
\section{Spectral properties of the linearized operator.}\label{App:3}
In this Appendix, we follow the arguments in \cite{Campos-Murphy,2008GFA,DuyRouden:NLS:ThresholdSolution,2009JFA} to give the proof of  Lemma \ref{l:eigen}. 

\vspace{0.2cm}

\noindent \textbf{(a) Existence and symmetry  of the eigenfunctions.} Note that  $\overline{\mathcal{L} (v)}=-\mathcal{L} (\overline{v})$, so that if  $e_0>0$ is an eigenvalue for  $\mathcal{L} $ with eigenfunction  $\mathcal{Y}_+$,  $-e_0$ is an eigenvalue of  $\mathcal{L} $ with eigenfunction  $\overline{\mathcal{Y}}_+$. Let us show the
existence of  $\mathcal{Y}_+$. Writing  $\mathcal{Y}_1=\text{Re} \mathcal{Y}_+,\ \mathcal{Y}_2=\text{Im} \mathcal{Y}_+$, we must solve 
\begin{equation}
	\begin{cases}
		(\Delta +(\alpha +1)V)\mathcal{Y}_1=-e_0\mathcal{Y}_2,\\
		(\Delta +V)\mathcal{Y}_2=e_0\mathcal{Y}_1,
	\end{cases}\label{892}
\end{equation}
where  $V:=|x|^{-b}W^\alpha $.   
The operator  $-\Delta -V$ on  $L^2$ with domain  $H^2$ is self-adjoint and nonnegative, thus it has a unique square root  $(-\Delta -V)^{1/2}$ with domain  $H^1$.      
Assume that there exists  $f\in H^4$ such that 
\begin{equation}
	Pf=-e_0^2 f,\quad\text{where}\quad P:=(-\Delta -V)^{1/2}(-\Delta -(\alpha +1)V)(-\Delta -V)^{1/2}\notag
\end{equation} 
Then taking 
\begin{equation}
	\mathcal{Y}_1:=(-\Delta -V)^{1/2}f,\qquad \mathcal{Y}_2:=\frac{1}{e_0} (-\Delta -(\alpha +1)V)(-\Delta -V)^{1/2}f\notag
\end{equation}
would yield a solution of system (\ref{892}), showing the existence of  $\mathcal{Y}_+$ and  $\mathcal{Y}_-$. 

It suffices to show that the operator  $P$ on  $L^2$ with domain  $H^4$ has a strictly negative eigenvalue.    Note that  $0<b<\min \left\{ \frac{d}{4},d-2 \right\}$ by the assumptions made at the begining of section \ref{s:2}, it is straightforward to check that 
\begin{equation}
	P=(\Delta +V)^2-\alpha (-\Delta -V)^{1/2}V(-\Delta -V)^{1/2}\notag
\end{equation}
is a relatively compact, self-adjoint perturbation of  $\Delta ^2$; so that its essential spectrum is  $[0,+\infty )$. Thus the proof  reduces to show the following:   
\begin{equation}
	\exists f\in H^4 \quad \text{such that}\quad ((\Delta +(\alpha +1)V)(-\Delta -V)^{1/2}f,(-\Delta -V)^{1/2}f)_{L^2}>0.\label{893}
\end{equation}

We distinguish two cases.  First assume that  $d=3,4$,   so that  $W\in L^2(\mathbb{R} ^d)$.   We will use the localization method to prove (\ref{893}).  Let $W_a(x):=\chi\big( x/a \big) W(x)$, where $\chi$ is a smooth, radial function such that
$\chi(r)=1$ for $r\leq 1$ and $\chi(r)=0$ for $r\geq 2.$
We first claim
\begin{equation}
	\exists a>0\quad\text{such that }\quad  E_a:=\int (\Delta+(\alpha +1)V) W_a W_a >0.\label{894}
\end{equation}
Recall that $\Delta W=-|x|^{-b}W^{\alpha +1}$. Thus
\begin{equation}
	(\Delta+(\alpha +1)V)W_a=\alpha |x|^{-b}\chi\big( x/a \big) W^{\alpha +1} +\frac 2a (\nabla \chi) \big( x/a \big) \cdot \nabla W+\frac{1}{a^2} (\Delta \chi)\big(x/a \big) W.\notag
\end{equation}
Hence
\begin{align}
	&	\int (\Delta +(\alpha +1)V)W_a W_a\notag\\
	&=\alpha \int|x|^{-b}\chi^2(x/a) W^{\alpha +2}+\underbrace{\frac{2}{a} \int(\chi \nabla \chi) \big( x/a \big) \cdot \nabla W\, W}_{(A)}+\underbrace{\frac{1}{a^2}\int (\chi\Delta \chi)\big(x/a \big) W^2}_{(B)}.\notag
\end{align}
According to the explicit expression of  $W$,  $W\lesssim  |x|^{-(d-2)}$ and  $|\nabla W|\lesssim |x|^{-(d-1)}$ at infinity, which gives  $|(A)|+|(B)|\lesssim \frac{1}{a}$. Hence (\ref{894}).

Let us fix $a$ such that (\ref{894}) holds.
Recall that $W$ is not in $L^2$. Thus $\Delta+V$ is a selfadjoint operator on $L^2$, with domain $H^2$, and without eigenfunction.  In particular $\overline{R(\Delta+V)}=\text{Ker}\left\{\Delta +V\right\} ^{\bot}=L^2$. Hence for any  $\varepsilon >0$, we can find  $G_\varepsilon \in H^2$ such that 
\begin{equation}
	\|(\Delta +V)G_\varepsilon -(\Delta +V-1)W_a\|_{L^2}\le \varepsilon,\label{213w2}
\end{equation}
which together with  $ \|(\Delta +V-1)^{-1}\|_{L^2\rightarrow H^2}<+\infty $ implies
\begin{equation}
	\|(\Delta +V-1)^{-1}(\Delta +V)G_\varepsilon -W_a\|_{H^2}\lesssim \varepsilon .\label{895}
\end{equation}
Substituting (\ref{895}) into (\ref{894}), we obtain, for small  $\varepsilon >0$
\begin{equation}
	((\Delta +(\alpha +1)V)(\Delta +V-1)^{-1}(\Delta +V)G_\varepsilon,(\Delta +V-1)^{-1}(\Delta +V)G_\varepsilon)>0.\label{896}
\end{equation}
Since  $(\Delta +V-1)^{-1}(-\Delta -V)^{1/2}G_\varepsilon \in H^3$, there exists  $f\in H^4$ such that 
\begin{equation}
	 \|(\Delta +V-1)^{-1}(-\Delta -V)^{1/2}G_\varepsilon +f\|_{H^3}<\varepsilon ,\notag 
\end{equation}
which implies 
\begin{equation}
	\|(\Delta +V-1)^{-1}(\Delta +V)G_\varepsilon -(-\Delta -V)^{1/2}f\|_{H^2}\lesssim \varepsilon .\label{897}
\end{equation}
Substituting (\ref{897}) into (\ref{896}) and choosing  $\varepsilon >0$ sufficient small, we obtain (\ref{893}).  

Assume now that  $d=5$, so that  $W$  is in  $L^2(\mathbb{R} ^d)$.   In this case  $\text{R}(\Delta +V)^\bot =\text{Ker}(\Delta +V)=\text{Span}\left\{W\right\}$, and thus 
\begin{equation}
	\overline{\text{R}(\Delta +V)}=\{f\in L^2:(f,W)_{L^2}=0\}.\label{213w1}
\end{equation} 
Furthermore,   $\Delta +(\alpha +1)V$  is a self-adjoint compact perturbation of  $\Delta $ and 
\begin{equation}
((\Delta +(\alpha +1)V)W,W)_{L^2}=\alpha \int VW^2dx >0,\notag	
\end{equation} 
 which shows that  $\Delta +(\alpha +1)V$ has a positive eigenvalue.  Let  $Z$  be the eigenfunction for this eigenvalue. Recalling that   $(\Delta +(\alpha +1)V)W_1=0$  we get, for any real number   $\gamma  $ 
\begin{equation}
	\int _{\mathbb{R} ^d}(\Delta +(\alpha +1)V)(Z+\gamma W_1)(Z+\gamma W_1)=\int _{\mathbb{R} ^d}(\Delta +(\alpha +1)V)ZZ>0.\notag
\end{equation}
By explicit calculation,  $(W_1,W)_{L^2}\neq0$,  so that we can choose the real
number  $\gamma$  to have  $(Z+\gamma W_1,W)_{L^2}=0$. Hence by  $(\Delta +V)W=0$,
\begin{eqnarray}
	\left((\Delta +V-1)(Z+\gamma W_1),W\right)_{L^2}&=& \left(Z+\gamma W_1,(\Delta +V-1)W\right)_{L^2}\notag\\
	&=& -(Z+\gamma W_1,W)_{L^2}=0.\notag
\end{eqnarray}
By (\ref{213w1}),  we can choose, for any   $\varepsilon >0$  a function  $G_\varepsilon $  in   $H^2$  such that 
\begin{equation}
	 \|(\Delta +V)G_\varepsilon -(\Delta +V-1)(Z+\gamma W_1)\|_{L^2}<\varepsilon ,\notag 
\end{equation}
which is similar to (\ref{213w2}). 
As in the preceding case, we can find  $f\in H^4$  such that (\ref{893}) holds.  
 This completes the proof of  (a).

\vspace{0.5cm}

\noindent \textbf{(b) Decay of the eigenfunctions at infinty.}
Recall that the eigenfunctions  $\mathcal{Y}_+$ and  $\mathcal{Y}_-$ are complex conjugates. According to system (\ref{892}),   it suffices to show the decay result on  $\mathcal{Y}_1$ only. We first show the following property holds for all  $k$ and  $s$ 
\begin{equation}
	\forall \varphi \in C_0^\infty (\mathbb{R}^d\setminus\left\{0\right\} ),\ \forall R\ge1 ,\ \|\varphi (x/R)\mathcal{Y}_1\|_{H^s}\lesssim \frac{1}{R^k}.\  (\mathcal{P}_{k,s})\notag
\end{equation}
Recall that  $\mathcal{Y}_1=\sqrt {-\Delta -V}f_1$ with  $f_1\in H^4$, so that  $(\mathcal{P}_{0,3})$ is satisfied. We now show that for  $k\ge0, s\ge3$,  $(\mathcal{P}_{k,s})$ implies  $(\mathcal{P}_{k+1,s+1})$. Assume  $(\mathcal{P}_{k,s})$ and consider  $\varphi $ and  $\widetilde{\varphi }$ in  $C_0^\infty (\mathbb{R}^d\setminus\left\{0\right\} )$ such that  $\widetilde{\varphi }$ is  $1$ on the support of  $\varphi $.  Applying  $\Delta +V$ to the first equation of (\ref{892}) and combining the second equation, we obtain  
\begin{equation}
	(\Delta ^2+e_0^2)\mathcal{Y}_1=-V\Delta \mathcal{Y}_1-(\alpha +1)V^2\mathcal{Y}_1-(\alpha +1)\Delta (V\mathcal{Y}_1).\label{8151}
\end{equation}
By the explicit form of  $W$,  $V$ and all its derivatives decay at least as  $1/|x|^{4-b}$ at infinity. Thus (\ref{8151}) implies  $\|\varphi (x/R)(\Delta ^2+e_0^2)\mathcal{Y}_1\|_{H^{s-3}}\lesssim \frac{1}{R^{4-b}}\|\widetilde{\varphi }(x/R)\mathcal{Y}_1\|_{H^s}$. Hence 
\begin{equation}
	\|(\Delta ^2+e_0^2)(\varphi (\frac{x}{R})\mathcal{Y}_1)\|_{H^{s-3}}\lesssim  \frac{1}{R}\|\widetilde{\varphi }(x/R)\mathcal{Y}_1\|_{H^s}.\label{8152}
\end{equation}
By  $(\mathcal{P}_{k,s})$, the right-hand side of (\ref{8152}) is bounded by  $C/R^{k+1}$ for large  $R$. Furthermore,  $\Delta ^2+e_0^2$ is an isomorphism from  $H^{s+1}$ to  $H^{s-3}$, so that (\ref{8152}) implies  $\|\varphi (x/R)\mathcal{Y}_1\|_{H^{s+1}}\lesssim 1/R^{k+1}$, which yields exactly  $(\mathcal{P}_{k+1,s+1})$.  The proof of (\ref{decay1}) is complete. 

With the decay (\ref{decay1}) at infinity, and recalling  $H^2(\mathbb{R} ^d)\hookrightarrow L^2_{\text{loc}}(\mathbb{R} ^d)$, we have  $\mathcal{Y}_\pm\in L^{\frac{2d}{d+2},2}(\mathbb{R} ^d)$. After differentiating (\ref{892}), we obtain  $\mathcal{Y}_\pm \in W^3L^{\frac{2d}{d+2},2}(\mathbb{R} ^d)$, as desired.     

Finally, we prove that  $\mathcal{Y}_\pm \in L^\infty (\mathbb{R} ^d)$. 
In fact, by Lemma \ref{L:leibnitz} and (\ref{221w1}), we have 
\begin{eqnarray}
	 &&\||x|^{-b}W^\alpha \mathcal{Y}_\pm \|_{H^{\frac{1}{2}+\varepsilon }}\notag\\
	 &\lesssim &   \||x|^{-b}W^\alpha \|_{W^{\frac{1}{2}+\varepsilon }L^{\frac{d}{2-\varepsilon },2}} \|\mathcal{Y}_\pm\|_{L^{\frac{2d}{d-4+2\varepsilon },2}}+ \||x|^{-b}W^\alpha \|_{L^{d,2}} \| \mathcal{Y}_\pm\| _{W^{\frac{1}{2}+\varepsilon }L^{\frac{2d}{d-2},2}}\notag\\
	 &\lesssim &   \||x|^{-b}W^\alpha \|_{W^{\frac{1}{2}+2\varepsilon }L^{\frac{d}{2},2}}\|\mathcal{Y}_\pm\|_{H^2} +  \||x|^{-b}W^\alpha \|_{WL^{\frac{d}{2},2}}   \|\mathcal{Y}_\pm\|_{H^{\frac{3}{2}+\varepsilon }} \lesssim    \|\mathcal{Y}_\pm\|_{H^2}<+\infty ,\notag 
\end{eqnarray}
for  $\varepsilon >0$ sufficiently small.  This inequality together with the equation (\ref{892}) and the embedding  $H^{\frac{5}{2}+\varepsilon }(\mathbb{R} ^d)\hookrightarrow L^\infty (\mathbb{R} ^d)$ shows $\mathcal{Y}_\pm \in L^\infty (\mathbb{R} ^d)$.  

\vspace{0.5cm}

\noindent \textbf{(c) and (d).}  The proof of (c) is similar to (b), and (d) follows from  \cite[Corollory 5.3]{2008GFA}, so we omit the details.

\section*{Acknowledgements}
K. Yang was partially supported by   the Jiangsu Shuang Chuang Doctoral Plan and the Jiangsu Provincial Scientific Research Center of Applied Mathematics under Grant BK20233002.
T. Zhang was
partially supported by  NSFC of
China under Grants    11931010 and the Zhejiang Provincial Natural Science Foundation of China under Grant No. LDQ23A010001.

\end{document}